\documentclass[10pt,a4paper]{article}
\usepackage[utf8]{inputenc}
\usepackage[english]{babel}
\usepackage[T1]{fontenc}      
\usepackage[utf8]{inputenc}
\usepackage{a4}
\usepackage{setspace}
\usepackage{mathrsfs}
\usepackage[square,numbers]{natbib}
\usepackage{epsfig}
\usepackage[spaces,obeyspaces,hyphens]{url}

\usepackage{amsmath}
\usepackage{amssymb}
\usepackage{bm}

\usepackage{bbm}
\usepackage{dsfont}
\usepackage{xcolor}
\usepackage{comment}
\usepackage{authblk}

\usepackage{amsfonts}
\usepackage{stmaryrd}
\usepackage{epsf}
\usepackage{amsfonts,bbm}
\usepackage{amsthm}
\usepackage{amscd}
\usepackage{amsopn}

\usepackage{tabularx}
\usepackage{multirow}
\usepackage{pifont}
\usepackage{multirow}
\usepackage{tabulary}
\usepackage{setspace}
\usepackage{mathrsfs}
\usepackage{epsfig}
\usepackage{bm}
\usepackage{bbm}
\usepackage{amssymb}
\usepackage{pifont}
\usepackage[utf8]{inputenc}
\usepackage[T1]{fontenc}
\usepackage{lmodern}
\usepackage{latexsym}
\usepackage{bbm}
\usepackage{amsthm}
\usepackage{relsize}
\usepackage{amsmath}
\usepackage{amssymb}
\usepackage{bigints}
\usepackage{hyperref}
\usepackage[mathcal]{eucal}
\usepackage{multicol}
\usepackage{ulem}
\usepackage{dsfont}
\usepackage{amsfonts,amsmath,amssymb}
\usepackage{fancybox}
\usepackage{amsthm}
\usepackage{amsmath}
\usepackage{amssymb}
\usepackage{mathrsfs}
\usepackage{latexsym}

\usepackage{graphics}  
\usepackage{colortbl}
\usepackage{fancyhdr}
\usepackage{makeidx}
\usepackage{enumitem}
\usepackage{lettrine}
\usepackage{hyperref}
\usepackage{float}
\usepackage{amssymb,wasysym}
\usepackage{stmaryrd}
\usepackage{epsf}
\usepackage{amsfonts,bbm}
\usepackage{amsthm}
\usepackage{amscd}
\usepackage{amsopn}
\usepackage{tabularx}
\usepackage{multirow}
\usepackage{amsfonts}
\usepackage{subcaption}

\newtheorem{theoreme}{Theorem}[section]
\newtheorem{lemme}{Lemma}[section]
\newtheorem{remarque}{Remark}[section]
\newtheorem{proposition}{Proposition}[section]

\newtheorem{definition}{Definition}[section]

\newcommand{\R}{\mathbb{R}}

\newcommand{\N}{\mathbb{N}}

\newcommand{\norm}[1]{\left\Vert#1\right\Vert}

\newcommand{\bc}{\begin{center}}
	\newcommand{\ec}{\end{center}}
\newcommand{\benu}{\begin{enumerate}}
	\newcommand{\eenu}{\end{enumerate}}

\newcommand{\real}{\mathbb R}

\title{\bf {Drift Parameter Estimation for an Ergodic Extension of the Double Heston Model}}

\author[1]{Mohamed Ben Alaya}
\author[1,2]{Houssem Dahbi\thanks{  \sffamily Corresponding author}}
\author[2]{Hamdi Fathallah}
\affil[1]{Univ Rouen Normandie, Normandie Univ, LMRS UMR6085, F-76000 Rouen, France.}
\affil[2]{Universit\'e de Sousse, \'Ecole Sup\'erieure des Sciences et de Technologie de Hammam Sousse, LAMMDA LR16ES13, 4011 Hammam Sousse, Tunisie.}

\DeclareUnicodeCharacter{2212}{-}
\begin{document}
	\maketitle			
	\maketitle

	\begin{abstract}
The double Heston model is a widely used stochastic volatility model in financial mathematics, particularly for option pricing, risk management, and volatility surface calibration. This paper investigates the joint estimation of the drift parameters for an ergodic extension of the double Heston model. We first establish stochastic properties of the underlying process, including the existence of a stationary distribution and ergodicity. We then study maximum likelihood and conditional least squares estimators based on continuous-time observations. For both estimation procedures, we establish consistency and asymptotic normality under suitable ergodicity conditions.
\end{abstract}

	\makeatletter{\renewcommand*{\@makefnmark}{}
		\footnotetext{E-mail addresses:\\
			\url{mohamed.ben-alaya@univ-rouen.fr}\\
			\url{houssem.dahbi@outlook.com}\\
			\url{hamdi.fathallah@essths.u-sousse.tn}}\makeatother}				
	
\section{Introduction}

In this paper, we consider an ergodic extension of the double Heston stochastic volatility model, represented by the three-dimensional affine diffusion process
\(
Z=(Y^{(1)},Y^{(2)},X)^{\top},
\)
defined on the canonical state space
\(
\mathcal{D}:=[0,\infty)^2\times\mathbb{R},
\)
as the unique strong solution of the following stochastic differential equation:
\begin{equation}\label{model01}
	\begin{cases}
		\mathrm{d}Y_t
		=(a-bY_t)\,\mathrm{d}t
		+\operatorname{diag}(\sigma_1)S(Y_t)\,\mathrm{d}B_t,\\[1mm]
		\mathrm{d}X_t
		=(m-\kappa^{\top}Y_t-\theta X_t)\,\mathrm{d}t
		+\sigma_2^{\top}S(Y_t)
		\bigl(\rho\,\mathrm{d}B_t+\bar{\rho}\,\mathrm{d}W_t\bigr).
	\end{cases}
\end{equation}
Here,
\(
Y_t=(Y_t^{(1)},Y_t^{(2)})^{\top}\),
\(a=(a_1,a_2)^{\top}\in[0,\infty)^2\),
\(m\in\mathbb{R},
\) and
\[
b=
\begin{bmatrix}
	b_{11} & 0\\
	b_{21} & b_{22}
\end{bmatrix},
\qquad
b_{21}\in(-\infty,0],
\qquad
b_{11},b_{22}\in\mathbb{R}.
\]
Moreover,
\(
\kappa=(\kappa_1,\kappa_2)^{\top}\in\mathbb{R}^2\), \(\theta\in\mathbb{R}\), and
\[
\operatorname{diag}(\sigma_1)
=
\begin{bmatrix}
	\sigma_{11} & 0\\
	0 & \sigma_{12}
\end{bmatrix},
\qquad
\sigma_1=(\sigma_{11},\sigma_{12})^{\top}\in(0,\infty)^2,
\]
while \(\sigma_2=(\sigma_{21},\sigma_{22})^{\top}\in(0,\infty)^2\).
The function \(S\) is defined by
\[
S(y_1,y_2)
=
\begin{bmatrix}
	\sqrt{y_1} & 0\\
	0 & \sqrt{y_2}
\end{bmatrix},
\qquad
(y_1,y_2)\in(0,\infty)^2.
\]
The matrices $\rho$ and $\bar{\rho}$ are diagonal matrices
in $\mathcal{M}_2$ satisfying $\rho^2+\bar{\rho}^2=I_2$, with
$\vert\rho_{11}\vert<1$, $\vert\rho_{22}\vert<1$ and
$\bar{\rho}_{jj}=\sqrt{1-\rho_{jj}^2}$ for $j\in\{1,2\}$.
Finally, $B$ and $W$ are independent two-dimensional standard
Wiener processes, and the pair $(B,W)$ is independent of
$(Y_0,X_0)$, and we assume that \(\mathbb{P}\bigl(Y_0\in(0,\infty)^2\bigr)=1\).

Stochastic volatility plays a central role in the modeling of a wide range of financial assets, including equities, commodities, fixed-income products, and foreign exchange rates. In contrast to constant-volatility models, stochastic volatility models are able to reproduce important empirical features of financial markets, such as volatility smiles and skews. They have therefore received considerable attention, particularly in the pricing of exotic derivatives whose values depend strongly on the future evolution of the volatility surface.

One of the most influential stochastic volatility models is the Heston model introduced in~\cite{Heston1993}. It extends the Black--Scholes model~\cite{Black1973} by replacing the constant volatility assumption with a stochastic variance process driven by a Cox--Ingersoll--Ross process~\cite{Cox}. Although the Heston model provides a substantially more flexible framework than constant-volatility models, it may fail to accurately reproduce implied volatility smiles and skews, particularly for short maturities. A natural way to improve its flexibility is to introduce an additional stochastic volatility factor.

In this direction, Christoffersen et al.~\cite{Christoffersen2009} introduced the double Heston model, in which the volatility is driven by two stochastic factors. This two-factor structure provides additional flexibility for describing the shape and dynamics of the implied volatility surface. Several aspects of the double Heston model have subsequently been investigated. Zhang and Feng~\cite{Zhang2019} studied the pricing of American put options under the double Heston model. Orzechowski~\cite{Orzechowski2021} compared the Heston and double Heston models in terms of computational efficiency for the pricing of European call options using characteristic functions and inverse Fourier transforms. Fallah and Mehrdoust~\cite{Fallah2019} studied the existence and uniqueness of solutions for a double Heston model defined through two independent variance processes with non-Lipschitz diffusion coefficients.

The model \eqref{model01} belongs to the class of affine stochastic volatility processes, which is a subclass of the affine processes introduced by Duffie et al.~\cite{Duffie}. Affine diffusion processes form an important family of Markov processes and have been characterized, among others, by Filipovi\'c and Mayerhofer~\cite{Filipovic}. In the terminology of affine diffusions, the process \eqref{model01} is an \(\mathit{AD}(2,1)\) process defined on \(
[0,\infty)^2\times\mathbb{R}\). It extends several classical models, including the Vasicek, Cox--Ingersoll--Ross, and Heston models. 

The model \eqref{model01} is not exactly the double Heston model studied by Fallah and Mehrdoust~\cite{Fallah2019}, nor is it a direct extension of the Heston model considered, for example, in~\cite{Barczy}. Nevertheless, when \(\theta=0\),
\(\kappa_1=\frac{\sigma_{21}^2}{2}\),
and \(\kappa_2=\frac{\sigma_{22}^2}{2}\),
the process \(
\widetilde{X}_t:=e^{X_t}\) may be interpreted as an asset price, while \(X_t\) represents its log-price and \(
\sigma_{21}\sqrt{Y_t^{(1)}}\)
 and \(\sigma_{22}\sqrt{Y_t^{(2)}}\)
represent the two stochastic volatility factors. However, the condition \(\theta=0\) eliminates the mean-reverting term in the log-price process and prevents the direct application of ergodic methods. Consequently, the asymptotic analysis of drift estimators becomes substantially more delicate and requires different techniques.

By introducing the mean-reversion parameter \(\theta>0\), the process \(X\) is driven toward a long-term equilibrium determined by the volatility factors. Thus, the resulting asset-price dynamics include a drift term
depending on the log-price, unlike the standard double Heston model. Nevertheless, this extension makes it possible to establish ergodicity of the full process \(
Z=(Y^{(1)},Y^{(2)},X)^{\top}\). The resulting ergodic framework is essential for proving the strong consistency and asymptotic normality of the drift parameter estimators studied in this paper.

In our previous works~\cite{Dahbi,Dahbi2}, we investigated the stationary and ergodic properties of the \(\mathit{AD}(1,n)\) model, consisting of one volatility component and \(n\) real-valued components. We also studied the asymptotic behavior of maximum likelihood and conditional least squares estimators of its drift parameters. Unlike the \textit{AD}$(1,n)$ model, the present \textit{AD}$(2,1)$ framework involves two interacting stochastic variance factors, leading to a substantially more intricate mathematical structure. In particular, the affine transform is characterized by a system of coupled Riccati differential equations, while the drift matrix contains the interaction parameter $b_{21}$, which couples the dynamics of the two volatility components. Consequently, the stationary distribution, the ergodicity analysis, and the asymptotic study of the estimators require a joint treatment of the two interacting variance factors and the real-valued component. These features considerably increase both the analytical and statistical complexity of the model.

The contributions of this paper are fourfold. First, we classify the process \eqref{model01} according to the asymptotic behavior of its first moment. More precisely, the classification is based on the parameters \(b\) and \(\theta\). We call the process subcritical when \(\mathbb{E}(Z_t)\) converges as \(t\to\infty\), critical when \(\mathbb{E}(Z_t)\) exhibits polynomial growth, and supercritical when \(\mathbb{E}(Z_t)\) grows exponentially.

Second, using the affine structure of the process, we derive an explicit characterization of the stationary distribution \(
Z_{\infty}=(Y_{\infty},X_{\infty})\), whose Fourier--Laplace transform is given by 
\[
\mathbb{E}\left(e^{\nu^{\top}Z_{\infty}}\right)
=
\mathbb{E}\left(
e^{-\lambda^{\top}Y_{\infty}+i\mu X_{\infty}}
\right)
=
\exp\left(
a^{\top}\int_0^{\infty}K_s(-\lambda,\mu)\,\mathrm{d}s
+\frac{i\mu m}{\theta}
\right),
\]
where \(\lambda\) is a two-dimensional complex vector with non-negative real parts,
\(\mu\in\mathbb{R}\), and \(K\) is a two-dimensional complex-valued function satisfying the Riccati equation \eqref{EDK}.

We further establish exponential ergodicity by means of a Foster--Lyapunov criterion; see \cite[Section~6, p.~535]{Foster-Lyapunov}. More precisely, under the conditions \(a_1\in(0,\infty)\), \(a_2\in\left(\frac{\sigma_{12}^2}{2},\infty\right)\) and \(b_{11},b_{22},\theta\in(0,\infty)\),
we prove that there exist constants \(\delta>0\) and \(C>0\) such that
\[
\sup_{|g|\leq V+1}
\left|
\mathbb{E}\bigl(g(Z_t)\mid Z_0=z_0\bigr)
-
\mathbb{E}\bigl(g(Z_{\infty})\bigr)
\right|
\leq
C\bigl(V(z_0)+1\bigr)e^{-\delta t},
\]
for all \(t\in [0,\infty)\), all \(z_0\in\mathcal{D}\), and all Borel measurable functions
\(g:\mathcal{D}\to\mathbb{R}\), where \(V\) is an appropriate norm-like function. Further details are provided in Theorem~\ref{ergodicity theorem}. While the existence of the invariant distribution follows from the general affine-process framework of Jin et al.~\cite{Jin2}, the explicit Fourier–Laplace representation and the associated Riccati system derived here are tailored to the present \textit{AD}$(2,1)$ model and will play a central role in the ergodicity analysis.

Third, we study the estimation of the drift parameters based on continuous-time observations. We derive the maximum likelihood estimator and the conditional least squares estimator and establish their strong consistency and asymptotic normality in the ergodic case. As is standard in continuous-time drift estimation, the diffusion parameters \(\sigma_1\), \(\sigma_2\), \(\rho\) and \(\bar{\rho}\) are assumed to be known; see, for example, the discussion in \cite[Section~5, p.~14]{Dahbi}.

Fourth, we illustrate the theoretical results through numerical simulations. In particular, we examine the classification of the process and the finite-sample behavior of the maximum likelihood and conditional least squares estimators in the subcritical, and hence ergodic, case.

To the best of our knowledge, the ergodic properties and the statistical estimation of the drift parameters have not previously been studied for the extension of the double Heston model considered in this paper.

The remainder of the paper is organized as follows. Section~2 introduces the main notation and preliminary results. Section~3 classifies the double Heston process according to the asymptotic behavior of its expectation. Section~4 establishes convergence to a stationary distribution and characterizes its Fourier--Laplace transform. Section~5 proves the exponential ergodicity of the model \eqref{model01}. Sections~6 and~7 study, respectively, the maximum likelihood and conditional least squares estimators of the drift parameters and establish their asymptotic properties. Section~8 presents numerical illustrations. Finally, the appendix contains the proof of a technical lemma and recalls several limit theorems used throughout the paper.

\section{Preliminaries}
Let $\N,\;\N^*,\; \R, \;\mathbb{R}_{+},\; \mathbb{R}_{++}, \; \mathbb{R}_{-}, \; \mathbb{R}_{--}$ and $\mathbb{C}$
denote the sets of non-negative integers, positive integers, real numbers, non-negative real numbers, positive real numbers, non-positive real numbers, negative real numbers and complex numbers, respectively and let $\mathcal{D}=\R_{+}^2\times\R$, for $n \in \N^*.$ For $x,y\in\R$, we will use the notation $x\wedge y:=\min(x,y)$ and $x\vee y:=\max(x,y)$. 
For all $z\in\mathbb{C}$, \texttt{Re}($z$) denotes the real part of $z$ and \texttt{Im}($z$) denotes the imaginary part of $z$. Let us denote, for $p,q\in\N^*$,  by $\mathcal{M}_{p,q}$ the set of $p\times q$ real matrices, $\mathcal{M}_p$ the set of $p\times p$ real matrices, $\mathbf{I}_p$ the identity matrix in $\mathcal{M}_{p}$, $\textbf{0}_{p,q}$ the null matrix in $\mathcal{M}_{p,q}$, $\textbf{0}_{p}:=\textbf{0}_{p,1}$ and $\mathbf{1}_p\in\mathcal{M}_{p,1}$ is the 1-vector of size $p$. For $m,n\in\N^*$ and $x\in \R^{m+n}$, we write $I=\{1,\ldots,m\}$, $J=\{m+1,m+n\}$, $\mathbf{x}_I=(x_i)_{i\in I}$ and $\mathbf{x}_J=(x_j)_{j\in J}$. Throughout this paper, we use the notation 
$$A=\begin{bmatrix}
	A_{II}&A_{IJ}\\
	A_{JI}&A_{JJ}
\end{bmatrix}$$
for $A\in\mathcal{M}_{m+n}$ where $A_{II}=(a_{i,j})_{i,j\in{I}}$, $A_{IJ}=(a_{i,j})_{i\in{I},j\in{J}}$, $A_{JI}=(a_{i,j})_{i\in{J},j\in{I}}$ and $A_{JJ}=(a_{i,j})_{i,j\in{J}}$. 


For all matrices $A_1,\ldots,A_k$, $\mathrm{diag}(A_1,\ldots,A_k)^{\top}$ denotes the block matrix containing the matrices $A_1,\ldots,A_k$ on its diagonal and for every matrix $\textbf{M}$,  $\norm{\textbf{M}}_1:=\max\limits_{j}\displaystyle\sum_i \vert \textbf{M}_{ij}\vert$ and $\norm{\textbf{M}}_2:=\sqrt{\max\text{eig}(\textbf{M}^{\top}\textbf{M})}$, with $\textbf{M}^{\top}$ the transpose matrix of $\textbf{M}$. In particular, for every vector $\textbf{x}=(x_1,\ldots,x_p)^{\top}\in\R^p$, we have $\norm{\textbf{x}}_1:=\vert {x}_1\vert+\ldots+\vert x_p\vert $ and $\norm{\mathbf{x}}_2:=\sqrt{ {x}_1^2+\ldots+x_p^2}$.
We denote by $\otimes$ the Kronecker product defined, for all matrices $\mathbf{A}\in\mathcal{M}_{p,q}$ and $\mathbf{B}$ with suitable dimensions,
$$
\mathbf{A}\otimes\mathbf{B}=\begin{bmatrix}
	a_{11}\mathbf{B}&  \cdots & a_{1q}\mathbf{B}\\
	\vdots&  \ddots &\vdots\\
	a_{p1}\mathbf{B}&  \cdots & a_{pq}\mathbf{B}
\end{bmatrix}.
$$We define the vec-operator applied on a matrix $\mathbf{A}$ denoted by $\text{vec}(\mathbf{A})$, which stacks its columns into a column vector, for more details, see, e.g.,~\cite{Horn_matrix} and~\cite{matrix_cookbook}.

We use the notations $H_{\mathbf{x}}(f)$ and $\nabla_{\mathbf{x}}f$ for the Hessian matrix and the gradient column vector of the function $f$ with respect to ${\mathbf{x}}$. We denote by $\mathcal{C}^2(\mathcal{D},\R)$ the set of twice continuously differentiable real-valued functions on $\mathcal{D}$ and by $\mathcal{C}_c^2(\mathcal{D},\R)$ its subset of functions with compact support. We will denote the convergence in probability, in distribution and almost surely  by $\stackrel{\mathbb{P}}\longrightarrow$, $\stackrel{\mathcal{D}}{\longrightarrow}$, $\stackrel{a.s.}{\longrightarrow}$ respectively.
Let $(\Omega,\mathcal{F}, \mathbb{P})$ be a probability space endowed with the augmented filtration $(\mathcal{F}_t)_{t \in \R_{+}}$ corresponding to $(B_t,W_t)_{t\in\R_+}$ and a given initial value $Z_0=(Y_0,X_0)^{\top}$ being independent of $(B_t,W_t)_{t\in\R_+}$ such that $\mathbb{P}(Y_0\in\R^2_{++})=1$. We suppose that $(\mathcal{F}_t)_{t\in\R_+}$ satisfies the usual conditions, i.e., the filtration 
$(\mathcal{F}_{t})_{t \in \R_{+}}$ is right-continuous and $\mathcal{F}_{0}$ contains all the $\mathbb{P}$-null sets in $\mathcal{F}$. For all $t\in\R_+$, we use the notation  $\mathcal{F}_t^Y:=\sigma(Y_s;\,0\leq s\leq t)$ for the $\sigma$-algebra generated by $(Y_s)_{s\in[0,t]}$. 
\section{Classification result}
In this section, we classify the double Heston model according to the asymptotic behavior of its expected value. This classification provides insight into the model’s long-term stability and mean-reversion properties, which are essential for practical applications such as option pricing and risk management. We begin by presenting an integral representation of the model to facilitate this analysis. For all $t\in\R_+$, thanks to Itô's formula, the SDE $\eqref{model01}$ implies
			\begin{equation}
				Y^{(1)}_t=e^{-b_{11}t}Y_0^{(1)}+a_{1}\displaystyle\int_0^t e^{-b_{11
					}(t-s)} \mathrm{~d}s+\sigma_{11} \displaystyle\int_0^t e^{-b_{11}(t-s)} \sqrt{Y_s^{(1)}}\mathrm{~d}B_s^1,
				\label{Y1 expression}
			\end{equation}
		\begin{equation}
			Y_t^{(2)}=e^{-b_{22}t}Y_0^{(2)}+a_2 \displaystyle\int_0^t e^{-b_{22
				}(t-s)} \mathrm{~d}s-b_{21}\displaystyle\int_0^t e^{-b_{22}(t-s)}Y_s^{(1)}\mathrm{~d}s+\sigma_{12} \displaystyle\int_0^t e^{-b_{22}(t-s)} \sqrt{Y_s^{(2)}}\mathrm{~d}B_s^2
			\label{Y2 expression}
		\end{equation}
	and
{\begin{multline}
			X_t=e^{-\theta t} X_0+m\displaystyle\int_0^t e^{-\theta(t-s)} \mathrm{~d}s-\kappa_1\displaystyle\int_0^t e^{-\theta(t-s)} Y_s^{(1)}\mathrm{~d}s-\kappa_2\displaystyle\int_0^t e^{-\theta(t-s)} Y_s^{(2)}\mathrm{~d}s+\sigma_{21}\displaystyle\int_0^t e^{-\theta(t-s)} \sqrt{Y_s^{(1)}}\\\times\left(\rho_{11}\mathrm{~d}B_s^1+\sqrt{1-\rho_{11}^2}\mathrm{~d}W^1_s\right)+\sigma_{22} \displaystyle\int_0^t e^{-\theta(t-s)} \sqrt{Y_s^{(2)}}\left(\rho_{22}\mathrm{~d}B_s^2+\sqrt{1-\rho_{22}^2}\mathrm{~d}W^2_s\right).
\label{X expression}
	\end{multline}}
    These integral representations reveal the hierarchical structure of the double Heston model:
\( Y^{(1)} \) follows an autonomous CIR process, \( Y^{(2)} \) depends on \( Y^{(1)} \) through the coupling parameter \( b_{21} \),
and \( X \) depends on both variance factors through \( \kappa_1 \) and \( \kappa_2 \).
The exponential weights \( e^{-b_{11}(t - s)} \), \( e^{-b_{22}(t - s)} \), and \( e^{-\theta (t - s)} \)
reflect the mean-reverting nature of each component, with the decay rates determined by the
parameters \( b_{11} \), \( b_{22} \), and \( \theta \), respectively.

\subsection*{Asymptotic Expectation Analysis}

Next, we introduce a proposition about the asymptotic behavior of $\mathbb{E}(Z_t)$ as $t$ goes to infinity, where $(Z_t)_{t\in\R_+}$ is the unique strong solution of \eqref{model01} (For further details, see the section \textit{Existence and Uniqueness of Affine Processes} in~\cite{Filipovic}). 
\begin{proposition}
Let $(Z_t)_{t \in \mathbb{R}_+}$ be the unique strong solution of $\eqref{model01}$ satisfying 
$\mathbb{P}(Y_0 \in \mathbb{R}^2_{++}) = 1$ and 
$\mathbb{E}(\norm{Z_0}_1) < \infty$. Then, for all $t \in \mathbb{R}_+$, 
$Z_t$ is integrable. Moreover:
\begin{itemize}
    \item If $b_{11}, b_{22}, \theta \in \mathbb{R}_{++}$: $\mathbb{E}(Z_t)$ converges as $t \to \infty$.
    \item If one of $b_{11}, b_{22}, \theta$ equals zero and the rest are non-negative: $\mathbb{E}(Z_t)$ exhibits polynomial growth.
    \item If one of $b_{11}, b_{22}, \theta$ is negative: $\mathbb{E}(Z_t)$ exhibits exponential growth.
\end{itemize}
    \label{classification}
\end{proposition}
\begin{proof}
For square integrability, observe that for all $i \in \{1,2\}$, the process
\[
\left(\int_0^t \sqrt{Y^{(i)}_s}\, e^{-(t-s)\gamma_i}\, d\bar{B}^i_s \right)_{t \in \mathbb{R}_+}
\]
is a square-integrable martingale, where $\gamma_i \in \{b_{ii}, \theta\}$ and $\bar{B}^i_s \in \{B^i_s, W^i_s\}$.
This follows from the Itô isometry and the fact that
$\mathbb{E}\left(\int_0^t Y^{(i)}_s e^{-2(t-s)\gamma_i} ds \right) < \infty$
under the integrability assumption on $Y_0$.
 Hence, by taking the expectation on both sides of equations \eqref{Y1 expression} and applying the Fubini–Tonelli theorem, one easily checks that
\begin{equation}
\begin{split}
\mathbb{E}\left(Y_{t}^{(1)}\right)&= \mathbb{E}\left(Y_{0}^{(1)}\right)e^{-b_{11} t}+a_1 \int_{0}^{t} e^{-b_{11} u} \mathrm{~d} u= \begin{cases}\frac{a_1}{b_{11}}+\left(\mathbb{E}\left(Y_{0}^{(1)}\right)-\frac{a_1}{b_{11}}\right) e^{-b_{11} t}, &b_{11} \neq 0,\\
\mathbb{E}\left(Y_{0}^{(1)}\right)+a_1 t, &b_{11}=0,\end{cases}.
\end{split}
\label{E(Y^1_t)}
\end{equation}
Then, there exists a positive function $v_1$ defined on $\R_{++}$ such that $v_1(t)\mathbb{E}\left(Y_t^{(1)}\right)$ converges. Namely, $v_1(t)=1$, if $b_{11}\in\R_{++}$, $v_1(t)=\frac{1}{t}$, if $b_{11}=0$ and $v_1(t)=e^{b_{11}t}$, if $b_{11}\in\R_{--}$. 

Now, concerning the second component, let us compute the speed of convergence of $\mathbb{E}(Y_t^{(2)})$, a positive function $v_2(t)$ such that $v_2(t)\mathbb{E}(Y_t^{(2)})$ converges. Thanks to relation \eqref{Y2 expression}, we obtain
\begin{align}
\mathbb{E}\left(Y_{t}^{(2)}\right)&= \mathbb{E}\left(Y_{0}^{(2)}\right)e^{-b_{22} t}+a_2 \int_{0}^{t} e^{-b_{22} u} \mathrm{~d} u-b_{21}\int_0^te^{-b_{22}(t-s)}\mathbb{E}\left(Y_{s}^{(1)}\right)\mathrm{d}s.
\label{E(Y^2_t)}
\end{align}
We distinguish three cases: 

$\bullet$ $b_{22}\in\R_{++}$: by Kronecker's lemma, we obtain 
\begin{equation}
    \displaystyle\int_0^t e^{b_{22}s}\mathbb{E}\left(Y_s^{(1)}\right)\mathrm{d}s\sim \displaystyle\int_0^t \dfrac{e^{b_{22}s}}{v_1(s)}\mathrm{d}s,
\label{equiv Y^2 with Kronecker Lemma}
\end{equation}
Consequently, by combining \eqref{E(Y^2_t)} with \eqref{equiv Y^2 with Kronecker Lemma}, we distinguish three cases: if $b_{11}\in\R_{++}$, then $v_1(t)=1$ and $\mathbb{E}(Y_t^{(2)})$ converges, if $b_{11}=0$, then $v_1(t)=\dfrac{1}{t}$ and $\dfrac{1}{t} \mathbb{E}(Y_t^{(2)})$ converges, and if $b_{11}\in\R_{--}$, then $v_1(t)=e^{b_{11}t}$ and $e^{b_{11}t} \mathbb{E}(Y_t^{(2)})$ converges.

$\bullet$ $b_{22}=0$: the relation \eqref{E(Y^2_t)} becomes
\begin{align}
\mathbb{E}\left(Y_{t}^{(2)}\right)&= \mathbb{E}\left(Y_{0}^{(2)}\right)+a_2 t-b_{21}\int_0^t\mathbb{E}\left(Y_{s}^{(1)}\right)\mathrm{d}s
\end{align}
and we distinguish three cases: if $b_{11}\in\R_{++}$, then Cesaro lemma yields $\dfrac{1}{t}\mathbb{E}(Y_t^{(2)})$ converges, if $b_{11}\in\R_-$, by applying Kronecker's Lemma, we have $$\displaystyle\int_0^t \mathbb{E}(Y_s^{(1)})\mathrm{d}s\sim\displaystyle\int_0^t \dfrac{1}{v_1(s)}\mathrm{d}s$$ and we obtain if $b_{11}=0$, then $\dfrac{1}{t^{2}}\mathbb{E}(Y_t^{(2)})$ converges and if $b_{11}\in\R_{--}$, then $e^{b_{11}t}\mathbb{E}(Y_t^{(2)})$ converges.

$\bullet$ $b_{22}\in\R_{--}$: the relation \eqref{E(Y^2_t)} gives
\begin{align}
e^{b_{22}t}\mathbb{E}\left(Y_{t}^{(2)}\right)&= \mathbb{E}\left(Y_{0}^{(2)}\right)-\dfrac{a_2}{b_{22}}\left(1-e^{b_{22}t}\right)-b_{21}\int_0^te^{b_{22}s}\mathbb{E}\left(Y_{s}^{(1)}\right)\mathrm{d}s.
\end{align}
Hence, if $b_{11}\in\R_{+}$, then $e^{b_{22}s}\mathbb{E}(Y_s^{(1)})$ is integrable and $e^{b_{22}t}\mathbb{E}(Y_t^{(2)})$ converges and if $b_{11}\in\R_{--}$, then three cases appear: if $b_{11}=b_{22}$, then $\dfrac{e^{b_{22}t}}{t}\mathbb{E}(Y_{t}^{(2)})$ converges, if $b_{11}>b_{22}$, then $e^{b_{22}t}\mathbb{E}(Y_{t}^{(2)})$ converges and if $b_{11}<b_{22}$, then $e^{b_{11}t}\mathbb{E}(Y_{t}^{(2)})$ converges.

Now, concerning the third component, let us compute the rate of convergence of $\mathbb{E}(X_t)$, a positive function $v_3(t)$ such that $v_3(t)\mathbb{E}(X_t)$ converges. By relation \eqref{X expression}, we get
\begin{equation}
\mathbb{E}\left(X_t\right)=e^{-\theta t} \mathbb{E}\left(X_0\right)+m\displaystyle\int_0^t e^{-\theta(t-s)} \mathrm{~d}s-\kappa_1\displaystyle\int_0^t e^{-\theta(t-s)} \mathbb{E}\left(Y_s^{(1)}\right)\mathrm{~d}s-\kappa_2\displaystyle\int_0^t e^{-\theta(t-s)} \mathbb{E}\left(Y_s^{(2)}\right)\mathrm{~d}s.
\label{E(X_t)}
\end{equation}
Similarly to the discussion on the second component, we distinguish three cases:

$\bullet$ $\theta\in\R_{++}$: we use Kronecker's Lemma to obtain an equivalent to the third and the fourth terms on the right-hand side of relation \eqref{E(X_t)} to distinguish many cases. By simple calculations, as done for the second component, we obtain: if $v_1(t)=v_2(t)=1$, then $v_3(t)=1$, if $(v_1(t),v_2(t))\in\left\lbrace\left(1,\dfrac{1}{t}\right),\left(\dfrac{1}{t},1\right)\right\rbrace$, then $v_3(t)=\dfrac{1}{t^2}$, if $v_1(t)=v_2(t)=\dfrac{1}{t}$, then $v_3(t)=\dfrac{1}{t^3}$ and if $v_1(t)$ or $v_2(t)$ are exponential, then $v_3(t)=v_1(t)\vee v_2(t)$.

$\bullet$ $\theta=0$: Relation \eqref{E(X_t)} becomes
\begin{align}
\mathbb{E}\left(X_t\right)=\mathbb{E}\left(X_0\right)+m t-\kappa_1\displaystyle\int_0^t  \mathbb{E}\left(Y_s^{(1)}\right)\mathrm{~d}s-\kappa_2\displaystyle\int_0^t  \mathbb{E}\left(Y_s^{(2)}\right)\mathrm{~d}s.
\end{align}
If $v_1(t)=v_2(t)=1$, then $v_3(t)=\dfrac{1}{t}$, if $(v_1(t),v_2(t))\in\left\lbrace\left(1,\dfrac{1}{t}\right),\left(\dfrac{1}{t},1\right)\right\rbrace$, then $v_3(t)=\dfrac{1}{t^2}$, if $v_1(t)=v_2(t)=\dfrac{1}{t}$, then $v_3(t)=\dfrac{1}{t^3}$ and if $v_1(t)$ or $v_2(t)$ are exponential, then $v_3(t)=v_1(t)\vee v_2(t)$ and it is exponentially explosive.

$\bullet$ $\theta\in\R_{--}$: the relation \eqref{E(X_t)} gives 
\begin{align}
e^{\theta t}\mathbb{E}\left(X_t\right)= \mathbb{E}\left(X_0\right)-\dfrac{m}{\theta}\left(1-e^{\theta t}\right)-\kappa_1\displaystyle\int_0^t e^{\theta s} \mathbb{E}\left(Y_s^{(1)}\right)\mathrm{~d}s-\kappa_2\displaystyle\int_0^t e^{\theta s} \mathbb{E}\left(Y_s^{(2)}\right)\mathrm{~d}s.
\end{align}
As in the analysis done for the second component when $b_{22}\in\R_{--}$, we prove that $v_3(t)$ is exponentially explosive.
\end{proof}
Now, based on the asymptotic behavior of $\mathbb{E}(Z_t)$ as $t\to\infty$, we introduce in what follows the classification of the double Heston model defined by the SDE $\eqref{model01}$.
\begin{definition}[Classification of the double Heston model]
Let $(Z_t)_{t \in \mathbb{R}_+}$ be the unique strong solution of \eqref{model01} with
$\mathbb{P}(Y_0 \in \mathbb{R}^2_{++}) = 1$ and $\mathbb{E}(\norm{Z_0}_1) < \infty$.
Then, $Z$ is called
\begin{itemize}
    \item \textbf{Subcritical:} if $b_{11}, b_{22}, \theta \in \mathbb{R}_{++}$ 
    (i.e. $\mathbb{E}(Z_t)$ converges as $t\to\infty$);
    \item \textbf{Critical:} if one of $b_{11}, b_{22}, \theta$ equals zero and the rest are non-negative 
    (i.e. $\mathbb{E}(Z_t)$ grows polynomially);
    \item \textbf{Supercritical:} if one of $b_{11}, b_{22}, \theta$ is negative 
    (i.e. $\mathbb{E}(Z_t)$ grows exponentially).
\end{itemize}
\end{definition}
\begin{remarque}
The classification is based on the asymptotic behavior of the mean. 
In the \textbf{subcritical} case, the process exhibits mean-reverting behavior. 
In the \textbf{critical} case, the process exhibits polynomial explosion without stabilization. 
In the \textbf{supercritical} case, exponential explosion occurs, making parameter estimation more challenging.
\end{remarque}

For the remainder of the paper we restrict our attention to the subcritical regime $b_{11}, b_{22}, \theta \in \mathbb{R}_{++}$, since this is the only regime in which ergodic statistical inference can be developed.
 
\section{Stationarity}

Building on the previous classification, this section establishes the existence of a stationary distribution for the extended ergodic double Heston process. This result provides the probabilistic foundation for the ergodicity analysis and the asymptotic study of the estimators developed in the subsequent sections.\\

Proposition \ref{stationarity theorem AD21} specializes the general existence theorem for invariant distributions of affine processes established by Jin et al.~\cite{Jin2} to the present ergodic extension of the double Heston model. While the existence of the invariant distribution follows from the general affine-process framework established by Jin et al.~\cite{Jin2}, our contribution lies in the explicit derivation of the Fourier–Laplace transform and the associated Riccati system adapted to the present affine diffusion, which will be essential in the ergodicity analysis developed in the next section.

\begin{proposition}
    Let us consider the affine model $\eqref{model01}$ with
    $a\in\R^2_{+}$, $b_{11},b_{22}\in\R_{++}$,
    $b_{21}\in\R_-$, $m\in\R$, $\kappa\in\R^2$,
    $\theta\in\R_{++}$,
    $Y^{(1)}_0=y^{(1)}_0\in\R_+$,
    $Y^{(2)}_0=y^{(2)}_0\in\R_+$ and $X_0=0$. Then
    \begin{enumerate}[label=\arabic*)]
        \item
        $Z_t=(Y_t,X_t)^\top
        \stackrel{\mathcal{D}}{\longrightarrow}
        Z_{\infty}=(Y_\infty,X_\infty)^\top$
        as $t\to\infty$, and the distribution of $Z_\infty$
        is given by
        \begin{equation}
            \mathbb{E}\left(e^{\nu^{\top} Z_\infty}\right)
            =\mathbb{E}\left(
                e^{-\lambda^{\top} Y_\infty+i\mu X_\infty}
            \right)
            =\exp\left(
                a^{\top}\displaystyle\int_0^{\infty}
                K_s\left(-\lambda,\mu\right)\mathrm{d}s
                +\dfrac{i\mu m}{\theta}
            \right),
            \label{loi limite}
        \end{equation}
        for $\nu=(-\lambda,i\mu)^\top
        \in\mathcal{U}_1\times\mathcal{U}_2$, where
        $\mathcal{U}_1:=
        \lbrace z\in\mathbb{C}^2:
        \texttt{Re}(z)\in\R_-^2\rbrace$
        and
        $\mathcal{U}_2:=
        \lbrace z\in\mathbb{C}:
        \texttt{Re}(z)=0\rbrace$.
        The function
        $t\mapsto K_t(u):=
        (K_t^{(1)}(u),K_t^{(2)}(u))^\top$,
        with $u=(u_1,u_2,u_3)^\top
        \in\mathcal{U}_1\times\R$, is the unique global
        solution of the following Riccati system:
        \begin{equation}
            \begin{cases}
                \begin{aligned}
                    \dfrac{\partial K_t^{(1)}}{\partial t}(u)
                    &=
                    \frac{{\sigma_{11}}^2}{2}
                    (K_t^{(1)}(u))^2
                    -b_{11}K_t^{(1)}(u)
                    -b_{21}K_t^{(2)}(u)\\
                    &\quad
                    +i u_3\sigma_{11}\sigma_{21}\rho_{11}
                    e^{-\theta t}K_t^{(1)}(u)
                    -i u_3\kappa_1 e^{-\theta t}
                    -\dfrac{1}{2}u_3^2{\sigma_{21}}^2
                    e^{-2\theta t},\\
                    \dfrac{\partial K_t^{(2)}}{\partial t}(u)
                    &=
                    \frac{{\sigma_{12}}^2}{2}
                    (K_t^{(2)}(u))^2
                    -b_{22}K_t^{(2)}(u)\\
                    &\quad
                    +i u_3\sigma_{12}\sigma_{22}\rho_{22}
                    e^{-\theta t}K_t^{(2)}(u)
                    -i u_3\kappa_2 e^{-\theta t}
                    -\dfrac{1}{2}u_3^2{\sigma_{22}}^2
                    e^{-2\theta t},\\
                    K_0(u)&=(u_1,u_2)^\top.
                \end{aligned}
            \end{cases}
            \label{EDK}
        \end{equation}
        \label{i Stationarity}

        \item
        In addition, if $Z_0$ is a random initial value
        independent of $(B_t,W_t)_{t\in\R_+}$ which has
        the same distribution as $Z_\infty$ given by
        $\eqref{loi limite}$, then
        $(Z_t)_{t\in\R_+}$ is strictly stationary.
    \end{enumerate}
    \label{stationarity theorem AD21}
\end{proposition}
			\begin{proof}
    \begin{enumerate}[label=\arabic*)]
        \item The proof of the first assertion follows four main steps:\\
        \textbf{Step 1}:
        The first step is devoted to deriving the Fourier--Laplace
        transform of the process $Z=(Y,X)^\top$.
        In contrast to the reduction used in the proof of
        Theorem 3.1 in~\cite{Dahbi}, we work directly with $Z$
        to account for the correlations between the Brownian
        motions driving $Y$ and $X$.
        We shall obtain
        \begin{equation}
            \mathbb{E}\left(
                e^{-\lambda^{\top}Y_t+i\mu X_t}
            \right)
            =
            \exp\left(
                a^{\top}\displaystyle\int_0^t
                K_s(-\lambda,\mu)\mathrm{~d}s
                +y_0^{(1)}K_t^{(1)}(-\lambda,\mu)
                +y_0^{(2)}K_t^{(2)}(-\lambda,\mu)
                +\frac{i\mu m}{\theta}
                \left(1-e^{-\theta t}\right)
            \right).
            \label{(Y X) to (Y U V)}
        \end{equation}
        Indeed, the process $Z=(Y,X)^\top$ satisfies
        the following SDE:
        \begin{equation*}
            \begin{aligned}
                \begin{cases}
                    \mathrm{d}Y_t
                    &=(a-bY_t)\mathrm{d}t
                    +\mathrm{diag}(\sigma_1)S(Y_t)
                    \mathrm{~d}B_t,\\
                    \mathrm{d}X_t
                    &=(m-\kappa^{\top}Y_t-\theta X_t)
                    \mathrm{d}t
                    +\sigma_2^{\top}S(Y_t)
                    \left(\rho\mathrm{~d}B_t
                    +\bar{\rho}\mathrm{~d}W_t\right),\\
                    Z_0
                    &=(y_0^{(1)},y_0^{(2)},0)^\top.
                \end{cases}
            \end{aligned}
        \end{equation*}
        This process is affine by
        \cite[Theorem 3.2]{Filipovic}, which provides
        sufficient conditions on the drift and diffusion
        coefficients to guarantee that a diffusion process
        is affine. According to Definition 2.1
        in~\cite{Filipovic}, the process $Z$ satisfies,
        for all $T\in\R_+$, $0\leq t\leq T$ and
        $u=(u_1,u_2,u_3)^\top
        \in\mathcal{U}:=\mathcal{U}_1\times\mathcal{U}_2$,
        the following relation:
        \begin{equation}
            \begin{aligned}
                \mathbb{E}\left(
                    e^{u^{\top}Z_T}\mid\mathcal{F}_t
                \right)
                &=
                \mathbb{E}\left(
                    e^{(u_1,u_2)Y_T+u_3X_T}
                    \mid\mathcal{F}_t
                \right)\\
                &=
                \exp\left(
                    \phi(T-t,u)
                    +\psi(T-t,u)^{\top}Z_t
                \right),
            \end{aligned}
            \label{Laplace Y,U,V}
        \end{equation}
        where
        $\psi(t,u)=(\psi_1(t,u),\psi_2(t,u),
        \psi_3(t,u))^\top$ is the solution to
        \begin{equation}
            \begin{cases}
                \begin{aligned}
                    \dfrac{\partial\psi_1}{\partial t}(t,u)
                    &=
                    \frac{{\sigma_{11}}^2}{2}
                    \psi_1^2(t,u)
                    -b_{11}\psi_1(t,u)
                    -b_{21}\psi_2(t,u)\\
                    &\quad
                    +\sigma_{11}\sigma_{21}\rho_{11}
                    \psi_1(t,u)\psi_3(t,u)
                    -\kappa_1\psi_3(t,u)
                    +\frac{{\sigma_{21}}^2}{2}
                    \psi_3^2(t,u),\\[5pt]
                    \dfrac{\partial\psi_2}{\partial t}(t,u)
                    &=
                    \frac{{\sigma_{12}}^2}{2}
                    \psi_2^2(t,u)
                    -b_{22}\psi_2(t,u)\\
                    &\quad
                    +\sigma_{12}\sigma_{22}\rho_{22}
                    \psi_2(t,u)\psi_3(t,u)
                    -\kappa_2\psi_3(t,u)
                    +\frac{{\sigma_{22}}^2}{2}
                    \psi_3^2(t,u),\\[5pt]
                    \dfrac{\partial\psi_3}{\partial t}(t,u)
                    &=-\theta\psi_3(t,u),\\[5pt]
                    \psi(0,u)&=u,
                \end{aligned}
            \end{cases}
            \label{psi10}
        \end{equation}
        and
        \begin{equation}
            \begin{aligned}
                \phi(t,u)
                &=
                a_1\displaystyle\int_0^t
                \psi_1(s,u)\mathrm{~d}s
                +a_2\displaystyle\int_0^t
                \psi_2(s,u)\mathrm{~d}s
                +m\displaystyle\int_0^t
                \psi_3(s,u)\mathrm{~d}s.
            \end{aligned}
            \label{phi}
        \end{equation}
        From equation $\eqref{psi10}$, we deduce that
        $\psi_3(t,u)=u_3e^{-\theta t}$.
        Next, by combining $\eqref{Laplace Y,U,V}$
        with $\eqref{phi}$, $\mathbb{E}\left(e^{u^{\top}Z_t}\right)$ is equal to
        \begin{equation}
            \exp\left(
                a_1\displaystyle\int_0^t
                \psi_1(s,u)\mathrm{~d}s
                +a_2\displaystyle\int_0^t
                \psi_2(s,u)\mathrm{~d}s
                +\frac{mu_3}{\theta}
                \left(1-e^{-\theta t}\right)
                +y_0^{(1)}\psi_1(t,u)
                +y_0^{(2)}\psi_2(t,u)
            \right).
            \label{transfo}
        \end{equation}
        Taking $u=(-\lambda,i\mu)^\top$ in
        $\eqref{psi10}$ and comparing with $\eqref{EDK}$,
        we obtain
        $\psi_1(t,(-\lambda,i\mu)^\top)
        =K_t^{(1)}(-\lambda,\mu)$ and
        $\psi_2(t,(-\lambda,i\mu)^\top)
        =K_t^{(2)}(-\lambda,\mu)$.
        Hence, $\eqref{transfo}$ yields
        $\eqref{(Y X) to (Y U V)}$.
        
        \textbf{Step 2}:
Now, we prove that if $(u_1,u_2)\in\mathcal{U}_1$,
we have
$\left(\psi_1(t,u),\psi_2(t,u)\right)\in\mathcal{U}_1$,
for all $t\in\R_+$ and $u_3\in\mathcal{U}_2$.
Since $Y_t\in\R_+^2$ almost surely and
$\texttt{Re}(u_3)=0$, we deduce that
\begin{equation}
    \begin{aligned}
        \left\vert
        \mathbb{E}\left(
            e^{(u_1,u_2)Y_t+u_3X_t}
        \right)
        \right\vert
        \leq
        \mathbb{E}\left(
            e^{\texttt{Re}((u_1,u_2))Y_t}
        \right)
        \leq 1,
    \end{aligned}
    \label{E<1}
\end{equation}
for all $t\in\R_+$ and $u\in\mathcal{U}$.
Consequently, combining equations $\eqref{transfo}$
and $\eqref{E<1}$, and noting that
$\texttt{Re}\left(
\frac{mu_3}{\theta}(1-e^{-\theta t})
\right)=0$, we obtain
\begin{equation*}
    \exp\left(
        a_1\,\texttt{Re}\left(
            \int_0^t\psi_1(s,u)\mathrm{~d}s
        \right)
        +a_2\,\texttt{Re}\left(
            \int_0^t\psi_2(s,u)\mathrm{~d}s
        \right)
        +y_0^{(1)}\,\texttt{Re}\left(\psi_1(t,u)\right)
        +y_0^{(2)}\,\texttt{Re}\left(\psi_2(t,u)\right)
    \right)\leq 1.
\end{equation*}
Now, the Riccati system $\eqref{psi10}$ does not
depend on the parameters $a_1,a_2,y_0^{(1)}$
and $y_0^{(2)}$. By uniqueness, its solutions
$\psi_1$ and $\psi_2$ do not depend on these
parameters either. Since the preceding inequality
holds for every admissible choice of these parameters,
we may first choose
$a_1=a_2=y_0^{(2)}=0$ and $y_0^{(1)}\in\R_{++}$,
which yields
$\texttt{Re}\left(\psi_1(t,u)\right)\leq 0$.
Similarly, choosing
$a_1=a_2=y_0^{(1)}=0$ and $y_0^{(2)}\in\R_{++}$
yields
$\texttt{Re}\left(\psi_2(t,u)\right)\leq 0$.
Consequently, for all $t\in\R_+$ and
$u_3\in\mathcal{U}_2$, if
$(u_1,u_2)\in\mathcal{U}_1$, we have
$\left(\psi_1(t,u),\psi_2(t,u)\right)
\in\mathcal{U}_1$.

	\textbf{Step 3}:
Now, for $\mathcal{U}'=\mathcal{U}_1\times\R$,
we focus on constructing an upper bound for the modulus
of the functions $K_t^{(1)}(u)$ and $K_t^{(2)}(u)$.
By $\eqref{psi10}$ and $\eqref{EDK}$, we have
\begin{equation}
    \begin{aligned}
        \dfrac{\partial K_t^{(1)}}{\partial t}(u)
        &=
        \frac{{\sigma_{11}}^2}{2}
        (K_t^{(1)}(u))^2
        -b_{11}K_t^{(1)}(u)
        -b_{21}K_t^{(2)}(u)\\
        &\quad
        +iu_3\sigma_{11}\sigma_{21}\rho_{11}
        e^{-\theta t}K_t^{(1)}(u)
        -iu_3\kappa_1e^{-\theta t}
        -\frac{1}{2}u_3^2{\sigma_{21}}^2e^{-2\theta t},
    \end{aligned}
    \label{K tilde EDO }
\end{equation}
and
\begin{equation}
    \begin{aligned}
        \dfrac{\partial K_t^{(2)}}{\partial t}(u)
        &=
        \frac{{\sigma_{12}}^2}{2}
        (K_t^{(2)}(u))^2
        -b_{22}K_t^{(2)}(u)\\
        &\quad
        +iu_3\sigma_{12}\sigma_{22}\rho_{22}
        e^{-\theta t}K_t^{(2)}(u)
        -iu_3\kappa_2e^{-\theta t}
        -\frac{1}{2}u_3^2{\sigma_{22}}^2e^{-2\theta t}.
    \end{aligned}
    \label{K tilde 2 EDO }
\end{equation}
Furthermore, $\eqref{transfo}$ yields
\begin{equation}
    \mathbb{E}\left(e^{(u_1,u_2)Y_t+iu_3X_t}\right)
    =
    \exp\left(
        a^{\top}\displaystyle\int_0^t
        K_s(u)\mathrm{~d}s
        +y_0^{\top}K_t(u)
        +\frac{iu_3m}{\theta}
        \left(1-e^{-\theta t}\right)
    \right).
    \label{Lap en K}
\end{equation}
For $j\in\{1,2\}$, let us introduce the functions
$v^{(j)}$ and $w^{(j)}$ defined by
$$
    v_t^{(j)}(u)=-\texttt{Re}(K_t^{(j)}(u))
    \quad\text{and}\quad
    w_t^{(j)}(u)=\texttt{Im}(K_t^{(j)}(u)).
$$
These functions satisfy
\begin{equation}
    \begin{cases}
        \begin{aligned}
            \dfrac{\partial v_t^{(1)}}{\partial t}(u)
            &=
            -\frac{{\sigma_{11}}^2}{2}
            (v_t^{(1)}(u))^2
            +\frac{{\sigma_{11}}^2}{2}
            (w_t^{(1)}(u))^2
            -b_{11}v_t^{(1)}(u)
            -b_{21}v_t^{(2)}(u)\\
            &\quad
            +u_3\sigma_{11}\sigma_{21}\rho_{11}
            e^{-\theta t}w_t^{(1)}(u)
            +\frac{1}{2}u_3^2{\sigma_{21}}^2
            e^{-2\theta t},\\[5pt]
            \dfrac{\partial v_t^{(2)}}{\partial t}(u)
            &=
            -\frac{{\sigma_{12}}^2}{2}
            (v_t^{(2)}(u))^2
            +\frac{{\sigma_{12}}^2}{2}
            (w_t^{(2)}(u))^2
            -b_{22}v_t^{(2)}(u)\\
            &\quad
            +u_3\sigma_{12}\sigma_{22}\rho_{22}
            e^{-\theta t}w_t^{(2)}(u)
            +\frac{1}{2}u_3^2{\sigma_{22}}^2
            e^{-2\theta t},\\[5pt]
            \dfrac{\partial w_t^{(1)}}{\partial t}(u)
            &=
            -{\sigma_{11}}^2v_t^{(1)}(u)w_t^{(1)}(u)
            -b_{11}w_t^{(1)}(u)
            -b_{21}w_t^{(2)}(u)\\
            &\quad
            -u_3\sigma_{11}\sigma_{21}\rho_{11}
            e^{-\theta t}v_t^{(1)}(u)
            -u_3\kappa_1e^{-\theta t},\\[5pt]
            \dfrac{\partial w_t^{(2)}}{\partial t}(u)
            &=
            -{\sigma_{12}}^2v_t^{(2)}(u)w_t^{(2)}(u)
            -b_{22}w_t^{(2)}(u)\\
            &\quad
            -u_3\sigma_{12}\sigma_{22}\rho_{22}
            e^{-\theta t}v_t^{(2)}(u)
            -u_3\kappa_2e^{-\theta t},\\[5pt]
            v_0^{(j)}(u)&=-\texttt{Re}(u_j),\\[5pt]
            w_0^{(j)}(u)&=\texttt{Im}(u_j).
        \end{aligned}
    \end{cases}
    \label{vw}
\end{equation}
By Step 2, $v_t^{(j)}(u)\geq0$ for
$j\in\{1,2\}$, $t\in\R_+$ and $u\in\mathcal{U}'$.
To estimate the modulus, we note that
\begin{equation*}
        \texttt{Re}\left(
            \overline{K_t^{(j)}(u)}
            (K_t^{(j)}(u))^2
        \right)
        =
        -v_t^{(j)}(u)|K_t^{(j)}(u)|^2
        \leq0\quad \text{and}\quad
        \texttt{Re}\left(
            \overline{K_t^{(j)}(u)}
            iu_3\sigma_{1j}\sigma_{2j}\rho_{jj}
            e^{-\theta t}K_t^{(j)}(u)
        \right)
        =0.
\end{equation*}
Consequently, equation $\eqref{K tilde 2 EDO }$
gives, almost everywhere in $t$,
\begin{equation*}
    \dfrac{\partial}{\partial t}|K_t^{(2)}(u)|
    \leq
    -b_{22}|K_t^{(2)}(u)|
    +|\kappa_2u_3|e^{-\theta t}
    +\frac{1}{2}u_3^2{\sigma_{22}}^2e^{-2\theta t}.
\end{equation*}

Set
$$
    C_2:=\theta\wedge\frac{b_{11}}{2}
    \wedge\frac{b_{22}}{2},
    \qquad
    C_4(u):=|\kappa_2u_3|
    +\frac{1}{2}u_3^2{\sigma_{22}}^2.
$$
Integration of the preceding
differential inequality yields
\begin{equation*}
    \begin{aligned}
        |K_t^{(2)}(u)|
        \leq
        |u_2|e^{-b_{22}t}
        +C_4(u)\int_0^t
        e^{-b_{22}(t-s)}e^{-C_2s}\mathrm{~d}s&=
        |u_2|e^{-b_{22}t}
        +C_4(u)
        \frac{e^{-C_2t}-e^{-b_{22}t}}{b_{22}-C_2}\\
        &\leq
        \left(
            |u_2|+\frac{C_4(u)}{b_{22}-C_2}
        \right)e^{-C_2t}.
    \end{aligned}
\end{equation*}
Thus, setting
$C_1(u):=|u_2|+\frac{C_4(u)}{b_{22}-C_2}$,
we obtain
\begin{equation}
    |K_t^{(2)}(u)|\leq C_1(u)e^{-C_2t},
    \label{K vers 0}
\end{equation}
for all $t\in\R_+$ and $u\in\mathcal{U}'$.

Similarly, equation $\eqref{K tilde EDO }$
and $b_{21}\leq0$ give, almost everywhere in $t$,
\begin{equation*}
    \begin{aligned}
        \dfrac{\partial}{\partial t}|K_t^{(1)}(u)|
        &\leq
        -b_{11}|K_t^{(1)}(u)|
        -b_{21}|K_t^{(2)}(u)|+|\kappa_1u_3|e^{-\theta t}
        +\frac{1}{2}u_3^2{\sigma_{21}}^2e^{-2\theta t}\\
        &\leq
        -b_{11}|K_t^{(1)}(u)|
        +C_6(u)e^{-C_2t},
    \end{aligned}
\end{equation*}
where
$$
    C_6(u):=-b_{21}C_1(u)
    +|\kappa_1u_3|
    +\frac{1}{2}u_3^2{\sigma_{21}}^2.
$$
Integrating this differential inequality, we obtain
\begin{equation*}
    \begin{aligned}
        |K_t^{(1)}(u)|
        &\leq
        |u_1|e^{-b_{11}t}
        +C_6(u)
        \frac{e^{-C_2t}-e^{-b_{11}t}}{b_{11}-C_2}\\
        &\leq
        \left(
            |u_1|+\frac{C_6(u)}{b_{11}-C_2}
        \right)e^{-C_2t}.
    \end{aligned}
\end{equation*}
Finally, setting
$C_9(u):=|u_1|+\frac{C_6(u)}{b_{11}-C_2}$,
we deduce that
\begin{equation}
    |K_t^{(1)}(u)|\leq C_9(u)e^{-C_2t},
    \label{K 2 vers 0}
\end{equation}
for all $t\in\R_+$ and $u\in\mathcal{U}'$.			
                            
                 \textbf{Step 4}:
By Lévy's continuity theorem and relation $\eqref{Lap en K}$,
in order to prove the convergence in distribution
$Z_t\stackrel{\mathcal{D}}{\longrightarrow}Z_\infty$
as $t\to\infty$, it is enough to establish the following
two assertions:
\begin{enumerate}[label=\roman*)]
    \item
    For all $u\in\mathcal{U}'$ and $y_0\in\R_+^2$,
    we have
    \begin{equation*}
        \begin{aligned}
            \lim_{t\to\infty}
            \left[
                y_0^{\top}K_t(u)
                +a^{\top}\displaystyle\int_0^t
                K_s(u)\mathrm{~d}s
                +\frac{iu_3m}{\theta}
                \left(1-e^{-\theta t}\right)
            \right]
            =
            a^{\top}\displaystyle\int_0^\infty
            K_s(u)\mathrm{~d}s
            +\frac{iu_3m}{\theta}.
        \end{aligned}
    \end{equation*}
    \label{lim g_infty}

    \item
    The function
    $u\mapsto a^{\top}\displaystyle\int_0^\infty
    K_s(u)\mathrm{~d}s+\frac{iu_3m}{\theta}$
    is continuous on $\mathcal{U}'$ and vanishes at
    $u=(0,0,0)^\top$.
\end{enumerate}
For the first assertion, equations $\eqref{K vers 0}$
and $\eqref{K 2 vers 0}$ imply
$$
    \lim_{t\to\infty}y_0^{\top}K_t(u)=0
    \quad\text{and}\quad
    \int_0^\infty|K_s^{(j)}(u)|\mathrm{~d}s<\infty,
    \qquad j\in\{1,2\}.
$$
Consequently,
$$
    \lim_{t\to\infty}
    \int_0^t K_s(u)\mathrm{~d}s
    =
    \int_0^\infty K_s(u)\mathrm{~d}s.
$$
Together with $e^{-\theta t}\to0$, this proves
\ref{lim g_infty}.

For the second assertion, let $(u^{(n)})_{n\in\N}$
be a sequence in $\mathcal{U}'$ such that
$u^{(n)}\to u\in\mathcal{U}'$.
The right-hand sides of $\eqref{K tilde EDO }$
and $\eqref{K tilde 2 EDO }$ are continuous in
$t$ and $u_3$, and locally Lipschitz in the real
and imaginary parts of $K^{(1)}$ and $K^{(2)}$.
Since $K_0(u)=(u_1,u_2)^\top$, continuous dependence
of solutions on their initial values and parameters,
together with the bounds obtained in Step 3, yields
$$
    K_s(u^{(n)})\longrightarrow K_s(u)
    \quad\text{for every }s\in\R_+.
$$
Moreover, equations $\eqref{K vers 0}$
and $\eqref{K 2 vers 0}$ give
\begin{equation*}
    |K_s^{(2)}(u^{(n)})|
    \leq C_1(u^{(n)})e^{-C_2s},
    \qquad n\in\N,\ s\in\R_+,
\end{equation*}
and
\begin{equation*}
    |K_s^{(1)}(u^{(n)})|
    \leq C_9(u^{(n)})e^{-C_2s},
    \qquad n\in\N,\ s\in\R_+.
\end{equation*}
The functions $C_1$ and $C_9$ defined in Step 3
are continuous and bounded on bounded subsets of
$\mathcal{U}'$. Since $(u^{(n)})_{n\in\N}$ is bounded,
these inequalities provide an integrable dominating
function independent of $n$. Therefore, Lebesgue's
dominated convergence theorem implies
$$
    \lim_{n\to\infty}
    \int_0^\infty K_s(u^{(n)})\mathrm{~d}s
    =
    \int_0^\infty K_s(u)\mathrm{~d}s.
$$
This proves the required continuity.
Furthermore, uniqueness in $\eqref{EDK}$ implies
$K_s(0,0,0)=0$ for all $s\in\R_+$, so the limiting
exponent vanishes at the origin.

Hence, relation $\eqref{Lap en K}$ yields
\begin{equation*}
    \lim_{t\to\infty}
    \mathbb{E}\left(e^{(u_1,u_2)Y_t+iu_3X_t}\right)
    =
    \exp\left(
        a^{\top}\displaystyle\int_0^\infty
        K_s(u)\mathrm{~d}s
        +\frac{iu_3m}{\theta}
    \right),
    \qquad u\in\mathcal{U}'.
\end{equation*}
Restricting this identity to
$u_1,u_2\in i\R$ and $u_3\in\R$, we obtain
pointwise convergence of the characteristic functions
of $Z_t$ to a function continuous at the origin
and equal to $1$ there. Lévy's continuity theorem
therefore gives
$Z_t\stackrel{\mathcal{D}}{\longrightarrow}Z_\infty$.

Since $Z_t$ takes values in the closed set
$\R_+^2\times\R$, so does $Z_\infty$ almost surely.
For $-\lambda\in\mathcal{U}_1$ and $\mu\in\R$,
the function
$(y,x)\mapsto e^{-\lambda^{\top}y+i\mu x}$
is bounded and continuous on this state space.
Thus, convergence in distribution and the preceding
transform limit imply
\begin{equation*}
    \mathbb{E}\left(
        e^{-\lambda^{\top}Y_\infty+i\mu X_\infty}
    \right)
    =
    \exp\left(
        a^{\top}\displaystyle\int_0^\infty
        K_s(-\lambda,\mu)\mathrm{~d}s
        +\frac{i\mu m}{\theta}
    \right).
\end{equation*}
This proves $\eqref{loi limite}$ and completes
the proof of \ref{i Stationarity}.
			
        \item
        In order to prove strict stationarity, that is,
        translation invariance of the finite-dimensional
        distributions, we first show that, if $Z_0$ has
        the same distribution as $Z_\infty$ and is independent
        of $(B_t,W_t)_{t\in\R_+}$, then $Z_t$ has the same
        distribution as $Z_\infty$ for every $t\in\R_+$.
        It is enough to check that, for all $t\in\R_+$,
        $\lambda\in\R_+^2$ and $\mu\in\R$,
        \begin{equation*}
                \mathbb{E}\left(
                    \exp\left(
                        K_t(-\lambda,\mu)^{\top}Y_\infty
                        +i\mu e^{-\theta t}X_\infty
                        +g_t(-\lambda,\mu)
                    \right)
                \right)=
                \exp\left(
                    a^{\top}\displaystyle\int_0^\infty
                    K_s(-\lambda,\mu)\mathrm{~d}s
                    +\frac{i\mu m}{\theta}
                \right),
        \end{equation*}
        where
        $$
            g_t(-\lambda,\mu)
            :=
            a^{\top}\displaystyle\int_0^t
            K_s(-\lambda,\mu)\mathrm{~d}s
            +\frac{i\mu m}{\theta}
            \left(1-e^{-\theta t}\right).
        $$
        Indeed, the affine transform formula
        $\eqref{Laplace Y,U,V}$, applied with
        $u=(-\lambda,i\mu)^\top$, gives
        \begin{equation*}
            \mathbb{E}\left(
                e^{-\lambda^{\top}Y_t+i\mu X_t}
                \mid Z_0
            \right)
            =
            \exp\left(
                K_t(-\lambda,\mu)^{\top}Y_0
                +i\mu e^{-\theta t}X_0
                +g_t(-\lambda,\mu)
            \right).
        \end{equation*}
        Since $K_t(-\lambda,\mu)\in\mathcal{U}_1$,
        equation $\eqref{loi limite}$ yields
        \begin{equation*}
            \begin{aligned}
                &\mathbb{E}\left(
                    \exp\left(
                        K_t(-\lambda,\mu)^{\top}Y_\infty
                        +i\mu e^{-\theta t}X_\infty
                        +g_t(-\lambda,\mu)
                    \right)
                \right)\\
                &=
                \exp\left(
                    a^{\top}\displaystyle\int_0^\infty
                    K_s\left(
                        K_t(-\lambda,\mu),
                        \mu e^{-\theta t}
                    \right)\mathrm{~d}s
                    +\frac{i\mu m}{\theta}e^{-\theta t}
                    +g_t(-\lambda,\mu)
                \right)\\
                &=
                \exp\left(
                    a^{\top}\left(
                        \displaystyle\int_0^\infty
                        K_s\left(
                            K_t(-\lambda,\mu),
                            \mu e^{-\theta t}
                        \right)\mathrm{~d}s
                        +\displaystyle\int_0^t
                        K_s(-\lambda,\mu)\mathrm{~d}s
                    \right)
                    +\frac{i\mu m}{\theta}
                \right).
            \end{aligned}
        \end{equation*}
        Hence, it remains to check that
        \begin{equation*}
            \displaystyle\int_t^\infty
            K_s(-\lambda,\mu)\mathrm{~d}s
            =
            \displaystyle\int_0^\infty
            K_s\left(
                K_t(-\lambda,\mu),
                \mu e^{-\theta t}
            \right)\mathrm{~d}s.
        \end{equation*}
        This follows from the identity
        \begin{equation}
            K_s\left(
                K_t(-\lambda,\mu),
                \mu e^{-\theta t}
            \right)
            =
            K_{t+s}(-\lambda,\mu),
            \qquad t,s\in\R_+.
            \label{equality}
        \end{equation}
        Indeed, for each fixed $t\in\R_+$, the functions
        $$
            s\mapsto
            K_s\left(
                K_t(-\lambda,\mu),
                \mu e^{-\theta t}
            \right)
            \quad\text{and}\quad
            s\mapsto K_{t+s}(-\lambda,\mu)
        $$
        satisfy the same coupled Riccati system.
        In the shifted system, the time-dependent terms
        agree because
        $$
            \mu e^{-\theta(t+s)}
            =
            \left(\mu e^{-\theta t}\right)e^{-\theta s},
            \qquad
            \mu^2e^{-2\theta(t+s)}
            =
            \left(\mu e^{-\theta t}\right)^2e^{-2\theta s}.
        $$
        Both functions have the initial value
        $K_t(-\lambda,\mu)$ at $s=0$.
        Uniqueness of the solution of $\eqref{EDK}$
        therefore proves $\eqref{equality}$.
        Integrating this identity over $s\in\R_+$
        gives the required equality of integrals.

        Consequently, $Z_t$ has the same Fourier--Laplace
        transform, and hence the same distribution,
        as $Z_\infty$ for every $t\in\R_+$.
        Finally, $(Z_t)_{t\in\R_+}$ is a time-homogeneous
        Markov process. For any $h\in\R_+$ and
        $0\leq t_1<\cdots<t_k$, the joint distribution
        of $(Z_{t_1+h},\ldots,Z_{t_k+h})$ is determined
        by the distribution of $Z_{t_1+h}$ and the
        transition kernels over the successive time
        intervals $t_{j+1}-t_j$.
        These are the same as for
        $(Z_{t_1},\ldots,Z_{t_k})$.
        Thus, all finite-dimensional distributions are
        invariant under time translation, and
        $(Z_t)_{t\in\R_+}$ is strictly stationary.
    \end{enumerate}
\end{proof}

		\section{Exponential ergodicity}

        This section strengthens the stability analysis by proving exponential ergodicity of the double Heston model. Exponential ergodicity quantifies the speed at which the process converges to its stationary distribution, ensuring rapid stabilization of its long-run behavior; specifically at an exponential rate. Exponential ergodicity is the statistical backbone of our paper. The theorem that follows rigorously applies the Foster–Lyapunov criteria to establish the exponential ergodicity of the double Heston process.
		\begin{theoreme} Let us consider the double Heston model $\eqref{model01}$
    with $a_1\in\R_{++}$,
    $a_2\in(\frac{\sigma_{12}^2}{2},\infty)$,
    $b_{21}\in\R_-$, $b_{11},b_{22}\in\R_{++}$,
    $m\in\R$, $\kappa\in\R^2$ and $\theta\in\R_{++}$,
    and with initial random values
    $Z_0=(Y_0^{(1)},Y_0^{(2)},X_0)^\top$
    independent of $(B_t,W_t)_{t\in\R_+}$
    satisfying $\mathbb{P}(Y_0\in\R_{++}^2)=1$.
    Then there exist $\delta,B,\alpha,\beta\in\R_{++}$
    such that
    \begin{equation}
        \underset{\vert g\vert\leq V+1}{\sup}
        \left\vert
            \mathbb{E}\left(g(Z_t)\mid Z_0=z_0\right)
            -\mathbb{E}\left(g(Z_\infty)\right)
        \right\vert
        \leq B(V(z_0)+1)e^{-\delta t},
        \label{ergodicity1}
    \end{equation}
    for all $t\in\R_+$ and
    $z_0=(y_0,x_0)\in\mathcal{D}$, where the supremum
    is taken over all Borel measurable functions
    $g:\mathcal{D}\to\R$ satisfying $\vert g\vert\leq V+1$,
    and
    $$
        V(z):=y_1^2+\alpha y_2^2+\beta x^2,
        \quad z=(y_1,y_2,x)\in\mathcal{D}.
    $$
    Here, $Z_\infty=(Y_\infty,X_\infty)^\top$
    has the distribution given by $\eqref{loi limite}$.

    In addition, for every Borel measurable function
    $f:\R_+^2\times\R\to\R$ such that
    $\mathbb{E}\vert f(Z_\infty)\vert<\infty$, we have
    \begin{equation}
        \mathbb{P}\left(
            \underset{T\to\infty}{\lim}
            \dfrac{1}{T}\displaystyle\int_0^T
            f(Z_s)\mathrm{~d}s
            =\mathbb{E}f(Z_\infty)
        \right)=1.
        \label{ergodic}
    \end{equation}
    \label{ergodicity theorem}
\end{theoreme}

\begin{proof}
			By Foster–Lyapunov criteria, see \cite[Theorem 6.1]{Foster-Lyapunov}, in order to prove the exponential ergodicity given by $\eqref{ergodicity1}$, it is enough to check the following assertions
			\begin{enumerate}[label=(\roman*)]
				\item $(Z_t)_{t\in\R_+}$ is a Borel right process.
				\item For the skeleton chain $(Z_n)_{n\in\N}$ all compact sets are petite.
				\item There exist $c\in\R_{++}$ and $d\in\R$ such that the inequality
			\end{enumerate}
			\begin{equation*}
				(\mathcal{A}_n\, V)(z)\leq -c\,V(z)+d,\quad z\in\mathcal{O}_n,
			\end{equation*}
			holds for all $n\in\N$, where $\mathcal{O}_n:=\lbrace z\in\mathcal{D}:\norm{z}_1<n\rbrace$, for each $n\in\N$, $\mathcal{A}_n$ denotes the extended generator of the process $Z_{t,n}=(Y_{t,n},X_{t,n})_{t\in\R_{+}}$ given by
			\begin{equation*}
				Z_{t,n}:=\begin{cases}
					Z_t,\quad&t<T_n,\\
					(0,0,n),\quad&t\geq T_n,
				\end{cases}
			\end{equation*}
			where the stopping time $T_n$ is defined by $T_n:=\inf\lbrace t\in \R_+ : Z_t \in \mathcal{D}\backslash \mathcal{O}_n\rbrace$.
			
			Next, we are going to prove the above assertions.
						\begin{enumerate}[label=(\roman*)]
				\item $(Z_t)_{t\in\R_{+}}$ is a Borel right process. In fact, by Theorem $3.5$ of~\cite{Keller}, all affine processes are Feller processes and since $(Z_t)_{t\in\R_{+}}$ has continuous sample paths almost surely, it is automatically a strong Markov process, see \cite[Theorem 1 page 56]{Chung}. The Feller property implies that $(Z_t)_{t\in\R_{+}}$ is a Borel right process. 
				\item Since $(Z_n)_{n\in\N}$ is Feller, it is sufficient, by Proposition 6.2.8 in Meyn and Tweedie~\cite{Meync'}, to check that it is irreducible w.r.t. Lebesgue measure. Note that the irreducibility of $(Z_k)_{k\in\N}$ holds if  the conditional distribution of $Z_1=(Y_1,X_1)^{{\top}}$ given $Z_0=(Y_0,X_0)^{{\top}}$ is absolutely continuous w.r.t. the Lebesgue
measure on $\mathcal{D}$ such that the conditional density function $f_{Z_1\vert Z_0}$ is positive on $\mathcal{D}$. Indeed, the Lebesgue measure on $\mathcal{D}$ is $\sigma$-finite, and
				if $B$ is a Borel set in $\mathcal{D}$ with positive Lebesgue measure, then
				\begin{equation*}
					\begin{aligned}
						\mathbb{E}\left(\displaystyle\sum_{n=0}^{\infty}\mathds{1}_B(Z_n)\displaystyle\vert Z_0=z_0\right)&\geq\mathbb{P}(Z_1\in B\vert Z_0=z_0)=\displaystyle\int_B f_{Z_1\vert Z_0}(z\vert z_0)\mathrm{~d}z,
					\end{aligned}
				\end{equation*}
				for all $z_0\in\mathcal{D}$. The existence of this conditional density with the required property can be checked as follows. By taking $t=1$ in the expression $\eqref{X expression}$ and considering only its random part $$\tilde{X}_1=-\kappa^{\top}\displaystyle\int_0^1 e^{\theta s}Y_s\mathrm{~d}s+ {\sigma_2}^{\top}\displaystyle\int_0^1 e^{\theta s}S(Y_s)(\rho \mathrm{~d}B_s+\bar{\rho}\mathrm{~d}W_s),$$
    then it is sufficient to prove that $\left(Y_1^{(1)},Y_1^{(2)},\tilde{X}_1\right)^{\top}$ is absolutely continuous with respect to the Lebesgue measure on $\R^3$ with a positive density on $\mathcal{D}$. 
				Note that, conditionally on $(B_t)_{t\in[0,1]}$,
the random variable $\tilde{X}_1$ is normally
distributed with mean
$$
m_Y=-\kappa^\top\displaystyle\int_0^1
e^{\theta s}Y_s\mathrm{~d}s
+\sigma_2^\top\displaystyle\int_0^1
e^{\theta s}S(Y_s)\rho\mathrm{~d}B_s
$$
and variance
$$
C_Y=\sigma_{21}^2\bar{\rho}_{11}^{\,2}
\displaystyle\int_0^1e^{2\theta s}Y_s^{(1)}
\mathrm{~d}s
+\sigma_{22}^2\bar{\rho}_{22}^{\,2}
\displaystyle\int_0^1e^{2\theta s}Y_s^{(2)}
\mathrm{~d}s.
$$
Indeed, $Y$ is measurable with respect to the
driving $B$-path, whereas $W$ is independent of $B$.
The assumption
$\bar{\rho}_{11}^{\,2}+\bar{\rho}_{22}^{\,2}>0$
implies $C_Y>0$ almost surely.
    Consequently, for all $(y_1,y_2,x)\in\mathcal{D}$, we obtain
				\begin{equation*}
					\mathbb{P}\left(Y_1^{(1)}<y_1,\,Y_1^{(2)}<y_2,\,\tilde{X}_1<x \right)=\mathbb{E}\left(\mathds{1}_{\lbrace Y_1^{(1)}<y_1\rbrace} \mathds{1}_{\lbrace Y_1^{(2)}<y_2\rbrace}\displaystyle\int_{-\infty}^x p(u)\mathrm{d}u\right),
					\end{equation*}
				where $p(u)=\dfrac{1}{\sqrt{2\pi C_Y}} \exp\left(-\dfrac{1}{2C_Y}\left(u-m_Y\right)^2\right)$, for all $u\in\R$. Then, we get
				\begin{equation*}
    \mathbb{P}\left(Y_1^{(1)}<y_1,\,Y_1^{(2)}<y_2,\,\tilde{X}_1<x \right)=\displaystyle\int_0^{y_1}\int_0^{y_2} \int_{-\infty}^x \mathbb{E}\left(p(u) \vert Y_1=(z_1,z_2)\right)f_{Y_1}(z)\mathrm{~d}z_1\mathrm{~d}z_2\mathrm{~d}u
      \end{equation*}
				where $f_{Y_1}$ denotes the density function of $Y_1=(Y_1^{(1)},Y_1^{(2)})$, for a fixed initial state $Z_0=z_0\in\mathcal{D}$,
with $Y_0=y_0\in\R_+^2$.
    
    Next, we need the following result proved in the extended arXiv version of~\cite{Barczy_stationarity_ergodicty}: if $\xi$ and $\eta$ are random
variables such that $P(\xi > 0) = 1$, $E(\xi) < \infty$, $P(\eta > 0) = 1$, and $\eta$ is absolutely continuous with a
density function $f_{\eta}$ having the property $f_{\eta}(x) \in\R_{++}$ Lebesgue a.e. $x \in\R_{++}$, then $E(\xi \mid \eta = y)\in\R_{++}$
Lebesgue a.e. $y\in\R_{++}$.
On the one hand, by the Cauchy--Schwarz inequality,
we get
\begin{equation*}
\mathbb{E}(p(u))
\leq\frac{1}{\sqrt{2\pi}}
\mathbb{E}\left(\frac{1}{\sqrt{C_Y}}\right)
\leq
\begin{cases}
\displaystyle
\frac{1}{\sqrt{2\pi}\sigma_{21}|\bar{\rho}_{11}|}
\sqrt{\mathbb{E}\left(
\frac{1}{\int_0^1Y_s^{(1)}\mathrm{~d}s}
\right)},
&\bar{\rho}_{11}\neq0,\\[12pt]
\displaystyle
\frac{1}{\sqrt{2\pi}\sigma_{22}|\bar{\rho}_{22}|}
\sqrt{\mathbb{E}\left(
\frac{1}{\int_0^1Y_s^{(2)}\mathrm{~d}s}
\right)},
&\bar{\rho}_{11}=0.
\end{cases}
\end{equation*}
The first bound is finite by Lemma
$\eqref{finite Expectation of the inverse of CIR integral}$
in the appendix.
The second bound is finite by comparison with the CIR
process obtained by setting $b_{21}=0$, with the same
initial value and driving Brownian motion, and by
applying the same lemma.
Thus, $0<\mathbb{E}(p(u))<\infty$ for all $u\in\R$.
    Next, we prove that the probability density $f_{Y_1}$ exists and it is positive. For all $t\in\R_+$, we denote by $\tilde{Y}_t^{(2)}$ the conditional random variable $Y_t^{(2)}\vert\left({Y}_s^{(1)}\right)_{s\in[0,t]}$. Then, similarly as above, we get
    		\begin{equation*}
    \mathbb{P}\left(Y_1^{(1)}<y_1,\,Y_1^{(2)}<y_2 \right)=\displaystyle\int_0^{y_1}\int_0^{y_2} \mathbb{E}\left(\tilde{p}(z_2) \vert Y_1^{(1)}=z_1\right)f_{Y_1^{(1)}}(z_1)\mathrm{~d}z_1\mathrm{~d}z_2,
      \end{equation*}
where $\tilde{p}$ is the probability density of $\tilde{Y}_1^{(2)}$. Conditionally on the first-component path, the second
component is an extended CIR process whose time and
state change leads to a squared Bessel process with
time-dependent dimension; see Lemma 2.4
in~\cite{Shirakawa} and~\cite{Peng}. Furthermore the expectation of its density can be deduced through the  inverse Fourier transform as follows $2\pi\mathbb{E}(\tilde{p}(u))
=
\mathbb{E}\left(
\displaystyle\int_\R e^{-i\omega u}
\mathbb{E}\left(
e^{i\omega Y_1^{(2)}}
\mid (Y_s^{(1)})_{s\in[0,1]}
\right)\mathrm{~d}\omega
\right).$ which is equal to
\begin{align*}
&\mathbb{E}\left(\int_\R e^{-i\omega u}\exp\left(i\omega\dfrac{y_0^{(2)}e^{-b_{22}}}{1-i\omega\frac{\sigma_{12}^2}{2b_{22}}\left(1-e^{-b_{22}}\right)}\right)\exp\left(i\omega\int_0^1\dfrac{\left(a_2-b_{21}Y_s^{(1)}\right)e^{-b_{22}(1-s)}}{1-i\omega\frac{ \sigma_{12}^2}{2b_{22}}\left(1-e^{-b_{22}(1-s)}\right)}\mathrm{d}s\right)\mathrm{d}\omega\right)\\
    &=\int_\R e^{-i\omega u}\exp\left(i\omega\dfrac{y_0^{(2)}e^{-b_{22}}}{1+\omega^2\frac{\sigma_{12}^4}{4b_{22}^2}\left(1-e^{-b_{22}}\right)^2}\right)\exp\left(-\omega^2\dfrac{\frac{\sigma_{12}^2 y_0^{(2)}}{2b_{22}}e^{-b_{22}}\left(1-e^{-b_{22}}\right)}{1+\omega^2\frac{\sigma_{12}^4}{4b_{22}^2}\left(1-e^{-b_{22}}\right)^2}\right)\\
    &\qquad\times\mathbb{E}\left(\exp\left(i\omega\int_0^1\dfrac{\left(a_2-b_{21}Y_s^{(1)}\right)e^{-b_{22}(1-s)}}{1+\omega^2\frac{\sigma_{12}^4}{4b_{22}^2}\left(1-e^{-b_{22}(1-s)}\right)^2}\mathrm{d}s\right)\right.\\
    &\qquad\qquad\left.\times\exp\left(-\omega^2\int_0^1\dfrac{\frac{\sigma_{12}^2}{2b_{22}}\left(a_2-b_{21}Y_s^{(1)}\right)e^{-b_{22}(1-s)}\left(1-e^{-b_{22}(1-s)}\right)}{1+\omega^2\frac{\sigma_{12}^4}{4b_{22}^2}\left(1-e^{-b_{22}(1-s)}\right)^2}\mathrm{d}s\right)\right)\mathrm{d}\omega.
\end{align*}
Then, since $-b_{21}Y_s^{(1)}\in\R_{+}$, it is easy to deduce, by applying the modulus to the above equation, that $\mathbb{E}\left(\tilde{p}(u)\right)$ is bounded by
\begin{align*}
\int_\R
  \exp\left(-\omega^2\int_0^1\dfrac{\frac{\sigma_{12}^2}{2b_{22}}a_2e^{-b_{22}(1-s)}\left(1-e^{-b_{22}(1-s)}\right)}{1+\omega^2\frac{\sigma_{12}^4}{4b_{22}^2}\left(1-e^{-b_{22}(1-s)}\right)^2}\mathrm{d}s\right)\mathrm{d}\omega\leq \int_\R
  \dfrac{1}{\left(1+\omega^2\frac{\sigma_{12}^4}{4b_{22}^2}\left(1-e^{-b_{22}}\right)^2\right)^{\frac{a_2}{\sigma_{12}^2}}}\mathrm{d}\omega,
\end{align*}
which is clearly finite under the assumption $\frac{2a_2}{\sigma_{12}^2}>1$. Consequently, the density $f_{Y_1}$ is positive, which completes the proof of this assertion.
				\item Let us recall that since the sample paths of $(Y,X)^{\top}$ are almost surely continuous, for each $n\in\N$, the extended generator $\mathcal{A}_n$ has a form as given in \cite[Theorem 1.2]{Friesen}. In fact, for all functions $f\in\mathcal{C}_c^2(\mathcal{D},\R)$, $\mathcal{A}_nf = \mathcal{A}f$ on $O_n$, where $\mathcal{A}$ denotes the (non-extended) generator of the process $(Y,X)$, which is given, for all $z=(y_1,y_2,x)$, by
				\begin{equation*}
					\begin{aligned}
	&\mathcal{A}_nf(y_1,y_2,x)=\frac{{\sigma_{11}}^2}{2}y_1 \dfrac{\partial^2 f}{\partial y_1^2}(z)+\sigma_{11} \sigma_{21} \rho_{11} y_1 \,\dfrac{\partial^2 f}{\partial y_1\partial x}(z)+\frac{{\sigma_{12}}^2}{2}y_2\, \dfrac{\partial^2 f}{\partial y_2^2}(z)\\&+\sigma_{12} \sigma_{22} \rho_{22} y_2 \,\dfrac{\partial^2 f}{\partial y_2\partial x}(z)+\frac{1}{2} ({\sigma_{21}}^2 y_1+{\sigma_{22}}^2 y_2)\,\dfrac{\partial^2 f}{\partial x^2}(z)+(a_1-b_{11}y_1)\,\dfrac{\partial f}{\partial y_1}(z)\\&+(a_2-b_{21}y_1-b_{22}y_2)\,\dfrac{\partial f}{\partial y_2}(z)+(m-\kappa_1 y_1-\kappa_2 y_2-\theta x)\,\dfrac{\partial f}{\partial x}(z),
					\end{aligned}
				\end{equation*}
				for all $(y_1,y_2,x)\in \mathcal{O}_n,\, n\in\N$ and $f\in \mathcal{C}^2(\R_+^2\times\R,\R)$.
				Thus, one easily checks that
				\begin{equation*}
				\begin{aligned}
					\mathcal{A}_n V(y_1,y_2,x)=&({\sigma_{11}}^2+\beta{\sigma_{21}}^2+2a_1)y_1+(\alpha{\sigma_{12}}^2+\beta{\sigma_{22}}^2+2\alpha a_2)y_2-2b_{11}y_1^2\\&-2\alpha b_{22}y_2^2-2\alpha b_{21}y_1 y_2+2\beta m x-2\beta\kappa_1y_1 x-2\beta \kappa_2 y_2 x -2\beta\theta x^2.
				\end{aligned}
				\end{equation*}
The next step is to complete the squares in
$\mathcal{A}_nV+cV$, successively in $x$, $y_2$
and $y_1$.
Choose
$c\in(0,2\min\lbrace\theta,b_{22},b_{11}\rbrace)$
and set $c_1=\beta(2\theta-c)$,
$c_2=\alpha(2b_{22}-c)
-\frac{\beta\kappa_2^2}{2\theta-c}$ and $$c_3=2b_{11}-c
-\frac{\beta\kappa_1^2}{2\theta-c}
-\frac{1}{c_2}
\left(
\alpha b_{21}
-\frac{\beta\kappa_1\kappa_2}{2\theta-c}
\right)^2.$$ Define $C_1(y_1,y_2)=\beta(m-\kappa_1y_1-\kappa_2y_2)$, $C_2(y_1)=
\alpha\left(a_2+\frac{\sigma_{12}^2}{2}\right)
+\frac{\beta\sigma_{22}^2}{2}
-\frac{\beta m\kappa_2}{2\theta-c}
+\left(
-\alpha b_{21}
+\frac{\beta\kappa_1\kappa_2}{2\theta-c}
\right)y_1$ and $$C_3=
a_1+\frac{\sigma_{11}^2}{2}
+\frac{\beta\sigma_{21}^2}{2}
-\frac{\beta m\kappa_1}{2\theta-c}
+\frac{C_2(0)}{c_2}
\left(
-\alpha b_{21}
+\frac{\beta\kappa_1\kappa_2}{2\theta-c}
\right).
$$
Then
\begin{equation*}
(\mathcal{A}_nV)(y_1,y_2,x)+cV(y_1,y_2,x)=-c_1\left(x-\frac{C_1(y_1,y_2)}{c_1}\right)^2 -c_2\left(y_2-\frac{C_2(y_1)}{c_2}\right)^2
-c_3\left(y_1-\frac{C_3}{c_3}\right)^2+d,
\end{equation*}
where $d=\frac{\beta m^2}{2\theta-c}
+\frac{C_2(0)^2}{c_2}
+\frac{C_3^2}{c_3}$. First choose $\alpha>0$ sufficiently small that $2b_{11}-c-\frac{\alpha b_{21}^2}{2b_{22}-c}>0$. For this fixed $\alpha$, as $\beta\downarrow0$,
$$
c_2\longrightarrow\alpha(2b_{22}-c)>0,
\qquad
c_3\longrightarrow
2b_{11}-c-\frac{\alpha b_{21}^2}{2b_{22}-c}>0.
$$
Thus, choosing $\beta\in(0,\alpha)$ sufficiently
small ensures $c_1,c_2,c_3>0$, and consequently
$$
(\mathcal{A}_nV)(z)\leq-cV(z)+d,
\qquad z\in\mathcal{O}_n,
$$
uniformly in $n$.			\end{enumerate}
		\end{proof}
 
	\section{Maximum likelihood estimation}
    With the long-term stability of the model established, we now turn to estimating the unknown drift parameters. In the statistical part of this paper, we will consider a global matrix representation of the model \eqref{model01} given by 
\begin{equation}
 \mathrm{d}Z_t=\Lambda(Z_t)\tau \mathrm{d}t+\sigma \tilde{S}(Z_t^1,Z_t^2)\tilde{\rho}\,\mathrm{d}\tilde{B}_t,\quad  t \in \mathbb{R}_{+},
	\label{model0 Z}
\end{equation} 
where $Z_t=(Z_t^1,Z_t^2,Z_t^3)^{{\top}}=(Y_t,X_t)^{{\top}}$, for all $t \in \mathbb{R}_+$, the matrix $\Lambda(Z_t)$ is given by
\[
\Lambda(Z_t) = \mathrm{diag}\big(\Lambda_1(Z_t), \Lambda_2(Z_t), \Lambda_3(Z_t)\big),
\]
where
\[
\Lambda_1(Z_t) = (1, -Z_t^1), \quad 
\Lambda_2(Z_t) = (1, -Z_t^1, -Z_t^2), \quad 
\Lambda_3(Z_t) = (1, -Z_t^1, -Z_t^2, -Z_t^3).
\]
or equivalently, $\Lambda(Z_t)$ can be expressed as the $3 \times 9$ matrix
\[
\Lambda(Z_t) =
\begin{pmatrix}
1 & -Z_t^1 & 0 & 0 & 0 & 0 & 0 & 0 & 0 \\
0 & 0 & 1 & -Z_t^1 & -Z_t^2 & 0 & 0 & 0 & 0 \\
0 & 0 & 0 & 0 & 0 & 1 & -Z_t^1 & -Z_t^2 & -Z_t^3
\end{pmatrix},
\]
the parameter vector $\tau$ is the nine-dimensional vector given by \begin{equation*}
\tau=\left(a_1,b_{11},a_2,b_{21},b_{22},m,\kappa_1,\kappa_2,\theta\right)^{{\top}},
\end{equation*}
$\sigma=\begin{bmatrix}
    \mathrm{diag}(\sigma_1)&\mathbf{0}_{22}\\
    \mathbf{0}_{2}^{\top}&\sigma_2^{{\top}}
\end{bmatrix}$, $\tilde{S}(Z_t^1,Z_t^2)=\mathbf{I}_2\otimes S(Z_t^1,Z_t^2)$, $\tilde{\rho}=\begin{bmatrix}
    \mathbf{I}_2&\mathbf{0}_{22}\\
    \rho&\bar{\rho}
\end{bmatrix}$ and $\tilde{B}_t=(B_t^1,B_t^2,W_t^1,W_t^2)^{{\top}}$.\\

This section introduces the Maximum Likelihood Estimator (MLE), which is one of the most statistically efficient approaches available for continuous observations. We are going to study the MLE of the drift parameter vector $\tau$ in the ergodic case using the above obtained results.
	The following theorem is about the construction of the Maximum Likelihood estimator based on continuous-time observations $(Z_t)_{t\in[0,T]}$.
    \begin{theoreme}
		Let $Z$ be the unique strong solution of $\eqref{model0 Z}$ with initial random values independent of $(\tilde{B}_t)_{t\in\R_{+}}$ satisfying $\mathbb{P}(Y_0\in\R_{++}^2)=1$. If $a_1\geq\frac{\sigma_{11}^2}{2}$ and $a_2\geq\frac{\sigma_{12}^2}{2}$, then, for all $T\in\R_{++}$, the MLE $\hat{\tau}_T$ of $\tau$ is given by
		\begin{equation}
			\hat{\tau}_T= \left(\displaystyle\int_0^T \Lambda(Z_s)^{\top} \mathrm{R}(Z_s)^{-1} \Lambda(Z_s)\mathrm{~d}s\right)^{-1}\displaystyle\int_0^T  \Lambda(Z_s)^{\top} \mathrm{R}(Z_s)^{-1} \mathrm{~d}Z_s,
			\label{MLE expression}
		\end{equation}
	where, for all $t\in\R_+$, $$\mathrm{R}(Z_t)=\begin{bmatrix}
			{\sigma_{11}}^2 Y_t^{(1)}&0&\rho_{11} \sigma_{11} \sigma_{21} Y_t^{(1)}\\
			0&{\sigma_{12}}^2Y_t^{(2)}&\rho_{22}\sigma_{12} \sigma_{22} Y_t^{(2)}\\
			\rho_{11} \sigma_{11} \sigma_{21} Y_t^{(1)}&\rho_{22}\sigma_{12} \sigma_{22} Y_t^{(2)}&{\sigma_{21}}^2 Y_t^{(1)}+{\sigma_{22}}^2 Y_t^{(2)}
		\end{bmatrix}.$$
		\label{MLE estimator}
	\end{theoreme}
	\begin{proof}
		Let $(\tilde{Z}_t)_{t\in\R_+}$ be the process defined, for all $t\in\R_+$, by the following SDE
		\begin{equation*}
			\mathrm{d}\tilde{Z}_t=\mathbf{q}(\tilde{Z}_t)\mathrm{~d}\tilde{B}_t,
		\end{equation*}
		where, for all $t\in\R_+$, $\mathbf{q}(\tilde{Z}_t)\in\mathcal{M}_{3,4}$ is given by 
		\begin{equation}\mathbf{q}(\tilde{Z}_t)=\begin{bmatrix}
			\sigma_{11}\sqrt{\tilde{Y}_t^{(1)}}&0&0&0\\
			0&\sigma_{12}\sqrt{\tilde{Y}_t^{(2)}}&0&0\\
			\rho_{11}\sigma_{21}\sqrt{\tilde{Y}_t^{(1)}}&			\rho_{22}\sigma_{22}\sqrt{\tilde{Y}_t^{(2)}}&
			\bar{\rho}_{11}\sigma_{21}\sqrt{\tilde{Y}_t^{(1)}}
			&\bar{\rho}_{22}\sigma_{22}\sqrt{\tilde{Y}_t^{(2)}}			\end{bmatrix}.\label{rho tilde}
   \end{equation}
		Note that 
$\mathbf{q}(Z_t)\mathbf{q}^{\top}(Z_t) = \mathrm{R}(Z_t)$. By definition, the likelihood ratio associated with the process $(Z_t)_{t\in\R_+}$ is given, under the assumptions of Liptser and Shiryayev \cite[relation (7.138)]{Lipster}, by
\begin{equation*}
    \frac{\mathrm{d}\mathbb{P}_{Z}}{\mathrm{d}\mathbb{P}_{\tilde{Z}}}(t)
    = \exp\Bigg(
        \int_0^t \tau^{\top} \Lambda(Z_s)^{\top} \mathrm{R}(Z_s)^{-1} \,\mathrm{d}Z_s
        - \frac{1}{2} \int_0^t \tau^{\top} \Lambda(Z_s)^{\top} \mathrm{R}(Z_s)^{-1} \Lambda(Z_s) \tau \,\mathrm{d}s
    \Bigg).
\end{equation*}

It remains to verify that the assumptions of \cite[relation (7.138)]{Lipster} are satisfied so that this expression is well-defined. First, note that the second integral is finite almost surely if 
\[
\int_0^t \big(\mathrm{R}(Z_s)^{-1}\big)_{ij} \,\mathrm{d}s < \infty, \quad \forall i,j\in\{1,2,3\}.
\]
A straightforward calculation shows that the components of the inverse matrix take the form
\begin{equation*}
    \left(\mathrm{R}(Z_s)^{-1}\right)_{ij}
    = \frac{C_1}{\sigma_{21}^2 \bar{\rho}_{11}^2 Y_s^{(1)} + \sigma_{22}^2 \bar{\rho}_{22}^2 Y_s^{(2)}}
      + \frac{C_2 (Y_s^{(1)})^2 + C_3 (Y_s^{(2)})^2}{\sigma_{21}^2 \bar{\rho}_{11}^2 (Y_s^{(1)})^2 Y_s^{(2)} + \sigma_{22}^2 \bar{\rho}_{22}^2 Y_s^{(1)} (Y_s^{(2)})^2},
\end{equation*}
for some constants $C_1, C_2, C_3$. Therefore, there exist constants $C_4, C_5$ such that
\begin{equation*}
    \int_0^t \left(\mathrm{R}(Z_s)^{-1}\right)_{ij} \,\mathrm{d}s
    \le \int_0^t \frac{C_4}{Y_s^{(1)}} \,\mathrm{d}s + \int_0^t \frac{C_5}{Y_s^{(2)}} \,\mathrm{d}s,
\end{equation*}
which is almost surely finite under the condition $a_i \ge \frac{\sigma_{1i}^2}{2}$ for $i=1,2$ (see, e.g., Proposition 4 in~\cite{Alaya1} for $i=1$, and for $i=2$ combine the same proposition with the comparison theorem (see, e.g., \cite[2.18 Proposition]{Karatzas}), taking $b_{21}=0$). 

Hence, the likelihood ratio is well-defined for all $t\in\R_+$.

		Denote by $L(\tau)$ the log-likelihood of $\tau$ defined, for all $t\in\R_+$, as follows		\begin{equation}
			L(\tau)=\displaystyle\int_0^t \tau^{\top} \Lambda(Z_s)^{\top} \mathrm{R}(Z_s)^{-1} \mathrm{~d}Z_s -\dfrac{1}{2} \displaystyle\int_0^t \tau^{\top}\Lambda(Z_s)^{\top} \mathrm{R}(Z_s)^{-1} \Lambda(Z_s)\tau\mathrm{~d}s.
			\label{Log-Likelihood}
		\end{equation}
		The above equation $\eqref{Log-Likelihood}$ yields 
		$$\triangledown_{\tau}L(\tau)=\displaystyle\int_0^t  \Lambda(Z_s)^{\top} \mathrm{R}(Z_s)^{-1} \mathrm{~d}Z_s - \displaystyle\int_0^t \Lambda(Z_s)^{\top} \mathrm{R}(Z_s)^{-1} \Lambda(Z_s)\mathrm{~d}s\,\tau.$$
		Furthermore, we have
		$$H_{\tau}L(\tau)=-\displaystyle\int_0^t \Lambda(Z_s)^{\top} \mathrm{R}(Z_s)^{-1} \Lambda(Z_s)\mathrm{~d}s,
		$$
		which is negative definite. Then, the maximum likelihood estimator is the solution of $\triangledown_{\tau}L(\tau)=\mathbf{0}_{9}$ given by 
		\begin{equation*}
			\hat{\tau}_t= \left(\displaystyle\int_0^t \Lambda(Z_s)^{\top} \mathrm{R}(Z_s)^{-1} \Lambda(Z_s)\mathrm{~d}s\right)^{-1}\displaystyle\int_0^t  \Lambda(Z_s)^{\top} \mathrm{R}(Z_s)^{-1} \mathrm{~d}Z_s.
		\end{equation*}
	\end{proof}
	\subsection{Consistency of the MLE}
    The next step is to rigorously prove that the MLE is consistent. Consistency is a fundamental property that guarantees that as the observation period $T$ tends to infinity, the estimated parameter vector converges almost surely to the true underlying parameter vector.
	\begin{theoreme}
		Let $Z$ be the unique strong solution of $\eqref{model0 Z}$ with initial random value $Z_0$ independent of $(\tilde{B}_t)_{t\in\R_{+}}$, having the same distribution as $Z_\infty$ defined by \eqref{loi limite} and satisfying $\mathbb{P}(Y_0\in\R_{++}^2)=1$. If $a_1>\frac{\sigma_{11}^2}{2}$, $a_2>\frac{\sigma_{12}^2}{2}$ and $b_{11},b_{22},\theta\in\R_{++}$, then, for $T\in\R_+$, the MLE $\hat{\tau}_T$ of $\tau$ given by $\eqref{MLE expression}$ is strongly consistent, i.e. 
		$\mathbb{P}\left(\underset{T\to\infty}{\lim} \hat{\tau}_T=\tau\right)=1.$
		\label{consistency Theorem}
	\end{theoreme}
	\begin{proof}
	The SDE \eqref{model0 Z} implies
    \begin{equation*}
			\tau=\left(\displaystyle\int_0^T \Lambda(Z_s)^{\top} \mathrm{R}^{-1}(Z_s)\Lambda(Z_s)\mathrm{~d}s\right)^{-1}\left(\displaystyle\int_0^T \Lambda(Z_s)^{\top} \mathrm{R}^{-1}(Z_s)\mathrm{~d}Z_s-\displaystyle\int_0^T\Lambda(Z_s)^{\top} \mathrm{R}^{-1}(Z_s)\mathbf{q}(Z_s)\mathrm{~d}\tilde{B}_s\right).
		\end{equation*}
		Then, by combining the above equation with the MLE expression \eqref{MLE expression}, we get
		\begin{equation}
			\begin{aligned}
				\hat{\tau}_T-\tau =\langle M \rangle_T^{-1} M_T,             
			\end{aligned}
			\label{error}
		\end{equation}
		where $M$ is a Brownian martingale given by
		\begin{equation*}
			M_T=\displaystyle\int_0^T \Lambda(Z_s)^{\top} \mathrm{R}^{-1}(Z_s)\mathbf{q}(Z_s)\mathrm{~d}\tilde{B}_s
		\end{equation*}
		and its quadratic variation 
		\begin{equation*}
			\langle M\rangle_T =\displaystyle\int_0^T \Lambda(Z_s)^{\top} \mathrm{R}^{-1}(Z_s)\Lambda(Z_s)\mathrm{~d}s.
		\end{equation*}
The square-integrable martingale property follows from the
strict stationarity of $Z$ and the integrability established
below, which imply that
$\mathbb{E}\left(\langle M^{(p)}\rangle_T\right)
=T\,\mathbb{E}\left(\langle M^{(p)}\rangle_1\right)<\infty$,
for all $T\in\R_+$ and $p\in\{1,\ldots,9\}$. Note that, for all $t\in\R_+$, by simple computations we can write an explicit expression of $M_T:=(M_T^{(1)},\ldots,M_T^{(9)})^\top$ and deduce that all its components can be written as an affine combination of the following martingales \\ 
		\begin{equation*}
				\tilde{M}_T^{(1)}=\displaystyle\int_0^T\dfrac{ c_1\sqrt{Y_s^{(i)}}+c_2\frac{Y_s^{(j)}}{\sqrt{Y_s^{(i)}}}}{\xi(Y_s)}\mathrm{~d}\tilde{B}_s^k,\quad \tilde{M}_T^{(2)}=\displaystyle\int_0^T \dfrac{c_3 (Y_s^{(i)})^{\frac{3}{2}}+c_4Y_s^{(j)} \sqrt{Y_s^{(i)}}}{\xi(Y_s)}\mathrm{~d}\tilde{B}_s^k
		\end{equation*}
	\begin{equation*}
	\begin{aligned}
		\tilde{M}_T^{(3)}=\displaystyle\int_0^T\dfrac{c_5Y_s^{(i)}\sqrt{Y_s^{(j)}}+c_6\frac{(Y_s^{(i)})^{\frac{3}{2}}}{\sqrt{Y_s^{(j)}}}}{\xi(Y_s)}\mathrm{~d}\tilde{B}_s^k \quad\text{and}\quad \tilde{M}_T^{(4)}=\displaystyle\int_0^T \dfrac{\sqrt{Y_s^{(i)}}X_s}{\xi(Y_s)}\mathrm{~d}W_s^i,
	\end{aligned}
\end{equation*}
 for all $i,j\in\{1,2\}$ such that $i\neq j$, $k\in\{1,2,3,4\}$ and some constants $c_p\in\R$, $p\in\{1,\ldots,5\}$ where $\xi(Y_s)=\bar{\rho}_{11}^2 {\sigma_{21}}^2 Y_s^{(1)}+\bar{\rho}_{22}^2 {\sigma_{22}}^2 Y_s^{(2)}$. Consequently, we can compute as well the quadratic variation $V_T^{(p)}=\langle M^{(p)}\rangle_T$, for all $p\in\{1,\ldots,9\}$ and deduce that it can be written as a linear combination of
		\begin{equation*}
			\tilde{V}^{(1)}_{T,(\alpha,\beta)}=\displaystyle\int_0^T \dfrac{\left(Y_s^{(i)}\right)^\alpha\left(Y_s^{(j)}\right)^\beta}{\xi^2(Y_s)}\mathrm{d}s, 
   	\end{equation*}
    for all $(\alpha,\beta)\in\{(1,0),(-1,2),(1,2),(3,0),(4,1)\}$,
		\begin{equation*}
		\tilde{V}^{(2)}_T=\displaystyle\int_0^T \dfrac{(Y_s^{(1)})^3}{Y_s^{(2)}\xi^2(Y_s)}\mathrm{d}s, \quad\tilde{V}^{(3)}_T=\displaystyle\int_0^T \dfrac{\left(Y_s^{(1)}\right)^{\frac{5}{2}}}{\xi^2(Y_s)}\mathrm{d}s,\quad \tilde{V}^{(4)}_T=\displaystyle\int_0^T \dfrac{1}{\xi(Y_s)}\mathrm{d}s
	\end{equation*}
 and
	$\tilde{V}^{(5)}_T=\displaystyle\int_0^T \dfrac{(Z_s^{(k)})^2}{\xi(Y_s)}\mathrm{d}s$, for all $i,j\in\{1,2\}$ such that $i\neq j$ and $k\in\{1,2,3\}$.
		Hence, using the ergodicity theorem \ref{ergodicity theorem}, we deduce that
					\begin{equation*}
			\dfrac{1}{T}\tilde{V}^{(1)}_{T,(\alpha,\beta)}\stackrel{a.s.}{\longrightarrow}\mathbb{E}\left( \dfrac{\left(Y_\infty^{(i)}\right)^\alpha\left(Y_\infty^{(j)}\right)^\beta}{\xi^2(Y_\infty)}\right), 
   	\end{equation*}
    for all $(\alpha,\beta)\in\{(1,0),(-1,2),(1,2),(3,0),(4,1)\}$,
		\begin{equation*}
		\dfrac{1}{T}\tilde{V}^{(2)}_T\stackrel{a.s.}{\longrightarrow}\mathbb{E}\left( \dfrac{(Y_\infty^{(1)})^3}{Y_\infty^{(2)}\xi^2(Y_\infty)}\right),\quad\dfrac{1}{T}\tilde{V}^{(3)}_T\stackrel{a.s.}{\longrightarrow}\mathbb{E}\left( \dfrac{\left(Y_\infty^{(1)}\right)^{\frac{5}{2}}}{\xi^2(Y_\infty)}\right),
  \end{equation*}
	\begin{equation*}
	\dfrac{1}{T}\tilde{V}^{(4)}_T\stackrel{a.s.}{\longrightarrow}\mathbb{E}\left( \dfrac{1}{\xi(Y_\infty)}\right)\ \text{ and }\ \dfrac{1}{T}\tilde{V}^{(5)}_T\stackrel{a.s.}{\longrightarrow}\mathbb{E}\left( \dfrac{(Z_\infty^{(k)})^2}{\xi(Y_\infty)}\right),\end{equation*}
as $T\to\infty$. In fact, for $i,j\in\lbrace 1,2\rbrace$ with $i\neq j$, we have
		\begin{equation*}
\mathbb{E}\left( \dfrac{Y_{\infty}^{(i)}}{\xi^2(Y_\infty)} \right)\leq\dfrac{1}{\bar{\rho}_{ii}^4 {\sigma_{2i}}^4}\mathbb{E}\left( \dfrac{1}{Y_{\infty}^{(i)}} \right) ,\quad \mathbb{E}\left( \dfrac{(Y_{\infty}^{(i)})^2}{Y_{\infty}^{(j)}\xi^2(Y_\infty)}\right)\leq\dfrac{1}{\bar{\rho}_{ii}^4 {\sigma_{2i}}^4}\mathbb{E}\left( \dfrac{1}{Y_{\infty}^{(j)}}\right) ,
		\end{equation*}
     \begin{equation*}
\mathbb{E}\left( \dfrac{(Y_{\infty}^{(i)})^3}{\xi^2(Y_\infty)}\right)\leq         \frac{1}{\bar{\rho}_{ii}^4 {\sigma_{2i}}^4}\mathbb{E}\left( Y_{\infty}^{(i)}\right),\quad \mathbb{E}\left( \dfrac{Y_{\infty}^{(i)}(Y_{\infty}^{(j)})^2}{\xi^2(Y_\infty)}\right)\leq\dfrac{1}{\bar{\rho}_{jj}^4 {\sigma_{2j}}^4} \mathbb{E}\left(Y_{\infty}^{(i)}\right),
     \end{equation*}
\begin{equation*}
	\mathbb{E}\left( \dfrac{(Y_{\infty}^{(i)})^4Y_{\infty}^{(j)}}{\xi^2(Y_\infty)}\right)\leq \dfrac{1}{\bar{\rho}_{ii}^4 {\sigma_{2i}}^4}	\mathbb{E}\left((Y_{\infty}^{(i)})^2Y_{\infty}^{(j)}\right),\quad \mathbb{E}\left( \dfrac{(Y_{\infty}^{(1)})^3}{Y_{\infty}^{(2)}\xi^2(Y_\infty)}\right)\leq \dfrac{1}{\bar{\rho}_{11}^4 {\sigma_{21}}^4} \mathbb{E}\left( \dfrac{Y_{\infty}^{(1)}}{Y_{\infty}^{(2)}}\right),
\end{equation*}

\begin{equation*}
	\mathbb{E}\left( \dfrac{(Y_{\infty}^{(1)})^{\frac{5}{2}}}{\xi^2(Y_\infty)}\right)\leq \dfrac{1}{\bar{\rho}_{11}^4 {\sigma_{21}}^4} \mathbb{E}\left(\sqrt{Y_{\infty}^{(1)}}\right),\quad \mathbb{E}\left( \dfrac{1}{\xi(Y_\infty)}\right)\leq \dfrac{1}{\bar{\rho}_{11}^2 {\sigma_{21}}^2} \mathbb{E}\left( \dfrac{1}{ Y_{\infty}^{(1)}}\right) ,
\end{equation*}

\begin{equation*}
	\mathbb{E}\left( \dfrac{(Y_{\infty}^{(i)})^2}{\xi(Y_\infty)}\right)\leq \dfrac{1}{\bar{\rho}_{ii}^2 {\sigma_{2i}}^2} \mathbb{E}\left(Y_{\infty}^{(i)}\right)
\end{equation*}
and \begin{equation*}
	\mathbb{E}\left( \dfrac{X_{\infty}^2}{\xi(Y_\infty)}\right)\leq \dfrac{1}{\bar{\rho}_{11}^2 {\sigma_{21}}^2} \mathbb{E}\left(X_{\infty}^{2\gamma}\right)^{\frac{1}{\gamma}}\mathbb{E}\left(\dfrac{1}{(Y_{\infty}^{(1)})^\zeta}\right)^{\frac{1}{\zeta}},
\end{equation*}
	for all $\gamma\in\R_{++}$ and $1<\zeta<
\min\left(\frac{2a_1}{\sigma_{11}^2},
          \frac{2a_2}{\sigma_{12}^2}\right)$ such that $\frac{1}{\gamma}+\frac{1}{\zeta}=1$. On the one hand, for all $r \geq 0$, 
the moment $\mathbb{E}\big[\vert Z_{\infty}^i\vert^r\big]$ is finite. 
This result follows by induction on the order $r$ using the same arguments as, e.g., \cite[Theorem B.2]{Boylog} or \cite[Proposition A.2]{Dahbi2}. On the other hand, we have $\mathbb{E}\left(\dfrac{1}{(Z_{\infty}^1)^\zeta}\right)<\infty$ by \cite[Proposition 1]{Alaya} and $\mathbb{E}\left(\dfrac{1}{(Z_{\infty}^2)^\zeta}\right)<\infty$ by the comparison theorem; in fact, by considering the CIR process $\mathcal{Y}_t^{(2)}$ satisfying the SDE associated to $Y_t^{(2)}$ with $b_{21}=0$, we can easily check that $\dfrac{1}{T}\displaystyle\int_0^T\dfrac{1}{(Y_s^{(2)})^\zeta}\mathrm{d}s\leq\dfrac{1}{T}\displaystyle\int_0^T\dfrac{1}{(\mathcal{Y}_s^{(2)})^\zeta}\mathrm{d}s\stackrel{a.s.}{\longrightarrow}\mathbb{E}\left( \dfrac{1}{(\mathcal{Y}_{\infty}^{(2)})^\zeta} \right)$ which is finite. Consequently, for $D_T=\mathrm{diag}(V_T^{-1})$, where $V_T=(V_T^{(1)},\ldots,V_T^{(9)})$, we get
		\begin{equation*}
			\hat{\tau}_T-\tau=(D_T\langle M\rangle_T)^{-1}D_T M_T.
		\end{equation*}
		Hence, by Theorem \ref{Lipster Shiryaev theorem} in the appendix, we get $D_T M_T \stackrel{a.s.}{\longrightarrow} \mathbf{0}_{9}$, as $T\to\infty$. It remains to prove that $D_T \langle M \rangle_T = (T D_T)\big(\frac{1}{T}\langle M \rangle_T\big)$ converges almost surely to an invertible limit matrix as $T \to \infty$.\\
On the one hand, we have already established that $T D_T$ converges almost surely to an invertible matrix. On the other hand, for any nonzero vector $y \in \mathbb{R}^9$, we have 
\begin{equation*}
    y^{\top} \lim_{T \to \infty} \frac{1}{T}\langle M \rangle_T y 
    = \mathbb{E}\!\left( y^{\top}\Lambda(Z_\infty)^{\top} \mathrm{R}^{-1}(Z_\infty) \Lambda(Z_\infty) y \right).
\end{equation*}
To show that this limit is positive, we note that for each nonzero $y$, the vector $\Lambda(Z_\infty) y$ is a nontrivial affine function of the components of $Z_\infty = (Y_\infty^{(1)}, Y_\infty^{(2)}, X_\infty)^{\top}$. Since the stationary distribution of $(Z_t)_{t \ge 0}$ admits a density with respect to the Lebesgue measure, $\Lambda(Z_\infty)y = 0$ can only occur on a set of measure zero. Hence, with probability one, $\Lambda(Z_\infty)y \ne 0$.\\
Moreover, the matrix $\mathrm{R}(Z_\infty)$ is positive definite because for all nonzero $z \in \mathbb{R}^3$, 
\[
z^{\top} \mathrm{R}(Z_\infty) z = z^{\top}\mathbf{q}(Z_\infty)\mathbf{q}(Z_\infty)^{\top} z 
    = \|\mathbf{q}(Z_\infty)^{\top} z\|^2 > 0,
\]
as $\mathbf{q}(Z_\infty)^{\top} z$ is a nontrivial affine combination of $(\sqrt{Y_\infty^{(1)}}, \sqrt{Y_\infty^{(2)}})$, which are strictly positive random variables. Consequently, $\mathrm{R}^{-1}(Z_\infty)$ is also positive definite. Therefore, for all nonzero $y \in \mathbb{R}^9$, 
\[
y^{\top}\Lambda(Z_\infty)^{\top}\mathrm{R}^{-1}(Z_\infty)\Lambda(Z_\infty)y > 0
\quad \text{almost surely,}
\]
and taking expectations preserves strict positivity. Hence, $
\lim_{T \to \infty} \frac{1}{T}\langle M \rangle_T$
is a positive definite (and therefore invertible) matrix. Finally, since both factors $T D_T$ and $\frac{1}{T}\langle M \rangle_T$ converge almost surely to invertible limits, their product $D_T\langle M \rangle_T$ converges almost surely to an invertible matrix, which completes the proof of the Theorem.

	\end{proof}
	\subsection{Asymptotic behavior of the MLE}
    While consistency guarantees convergence, we need to know the convergence rate and distribution to perform proper inference. The following theorem establishes that the MLE for the parameter vector $\tau$ is asymptotically normal. We prove it converges to a Gaussian distribution centered at zero, with a deterministic covariance matrix.
	\begin{theoreme}
	Let $Z$ be the unique strong solution of $\eqref{model0 Z}$ with initial random values independent of $(\tilde{B}_t)_{t\in\R_{+}}$ satisfying $\mathbb{P}(Y_0\in\R_{++}^2)=1$. If $a_1>\frac{\sigma_{11}^2}{2},a_2>\frac{\sigma_{12}^2}{2}$ and $b_{11},b_{22},\theta\in\R_{++}$, then the MLE $\hat{\tau}_T$ of $\tau$ given by $\eqref{MLE expression}$ is asymptotically normal, namely
		\begin{equation*}
			\sqrt{T}(\hat{\tau}_T -\tau)\stackrel{\mathcal{D}}{\longrightarrow}\mathcal{N}_9(0,\mathcal{V}),\quad \text{as }T\to\infty,
		\end{equation*}
		where $\mathcal{V}$ is the inverse matrix of $\mathbb{E}\left(\Lambda(Z_\infty)^{\top}\mathrm{R}^{-1}(Z_\infty)\Lambda(Z_\infty)\right)$.
	\end{theoreme}
	\begin{proof}
		Using the equation $\eqref{error}$, one can write
		\begin{equation*}
			\begin{aligned}
				\sqrt{T}(\hat{\tau}_T -\tau)= \left(\frac{1}{T}\langle M\rangle_T\right)^{-1}\frac{1}{\sqrt{T}}M_T.
			\end{aligned}
		\end{equation*}
		On one hand, we have 
		\begin{equation*}
			\left(\frac{1}{T}\langle M\rangle_T\right)^{-1}\stackrel{a.s.}{\longrightarrow} \left( \mathbb E\!\left(\Lambda(Z_\infty)^\top R(Z_\infty)^{-1}\Lambda(Z_\infty)\right)\right)^{-1}=: \mathcal{V},\quad \text{as }T\to\infty.
		\end{equation*}
		On the other hand, $M$ is a Brownian martingale with quadratic variation $\langle M\rangle$. Hence,
		by Theorem \ref{CLT Van Zanten} (see the appendix), we deduce that
		\begin{equation*}
			\frac{1}{\sqrt{T}}M_T\stackrel{\mathcal{D}}{\longrightarrow}\mathcal{N}_9(0,\mathcal{V}^{-1}),\quad \text{as }T\to\infty.
		\end{equation*}
  This completes the proof.
	\end{proof}
\section{Conditional least squares estimation}
In this section, we will use a second method of estimation, which is the Conditional Least Squares Estimator (CLSE) method. Its advantage is that it is based on some discretized observations, which is a more realistic scenario for financial market data than the continuous path required by the MLE. We begin this section with the construction of the CLSE. We consider the diffusion process $Z$ unique strong solution of the SDE $\eqref{model0 Z}$. The following discussion is about the construction of a CLSE for the drift parameter $\tau=\left(a_1,b_{11},a_2,b_{21},b_{22},m,\kappa_1,\kappa_2,\theta\right)^{{\top}}$ based on a continuous-time observation of $(Z_t)_{t\in[0,T]}$, for some $T\in\R_{++}$. At first, we consider the CLSE $\check{\tau}_{T,N}$ based on the following discretized process $(Y_{\frac{i}{N}},X_{\frac{i}{N}})_{i\in\lbrace 0,1,\ldots,\lfloor NT \rfloor\rbrace}$, for $N\in\N^*$ and it is obtained by solving the extremum problem
\begin{equation}
	\begin{aligned}
	\check{\tau}_{T,N}:=&\underset{\tau\in \R^{9}}{\arg\min}\displaystyle\sum_{i=1}^{ \lfloor NT \rfloor}\left[\left(Y_{\frac{i}{N}}^{(1)}-\mathbb{E}\left( Y_{\frac{i}{N}}^{(1)}\vert \mathcal{F}_{\frac{i-1}{N}}\right)\right)^2+\left(Y_{\frac{i}{N}}^{(2)}-\mathbb{E}\left(Y_{\frac{i}{N}}^{(2)}\vert \mathcal{F}_{\frac{i-1}{N}}\right)\right)^2+\left(X_{\frac{i}{N}}-\mathbb{E}\left( X_{\frac{i}{N}}\vert \mathcal{F}_{\frac{i-1}{N}}\right)\right)^2\right].
	\end{aligned}
\label{extremum problem}
\end{equation}
Thanks to relations \eqref{Y1 expression}, \eqref{Y2 expression} and \eqref{X expression}, for all $i\in \N^*$, we get
$\mathbb{E}\left( Y_{\frac{i}{N}}^{(1)}\vert \mathcal{F}_{\frac{i-1}{N}}\right)=e^{-{\frac{b_{11}}{N}}}Y_{\frac{i-1}{N}}^{(1)}+a_1\displaystyle\int_0^{\frac{1}{N}} e^{-b_{11}u}\mathrm{d}u,
$
\begin{align*}
    \mathbb{E}\left( Y_{\frac{i}{N}}^{(2)}\vert \mathcal{F}_{\frac{i-1}{N}}\right)=&e^{-\frac{b_{22}}{N}}Y^{(2)}_{\frac{i-1}{N}}+a_2\displaystyle\int_0^{\frac{1}{N}} e^{-b_{22}u}\mathrm{d}u-b_{21}\displaystyle\int_0^{\frac{1}{N}} e^{b_{22} (u-\frac{1}{N})}e^{-b_{11} u}\, \mathrm{d}u\,Y^{(1)}_\frac{i-1}{N}\\
  &-b_{21}a_1\displaystyle\int_0^{\frac{1}{N}} e^{b_{22} (u-\frac{1}{N})}\displaystyle\int_0^{u} e^{-b_{11}(u-v)}\mathrm{d}v\,\mathrm{d}u,
\end{align*}

and 
\begin{align*}
\mathbb{E}\left( X_{\frac{i}{N}}\vert \mathcal{F}_{\frac{i-1}{N}}\right)=&e^{-{\frac{1}{N}}\theta}X_{\frac{i-1}{N}}+ m\displaystyle\int_0^{\frac{1}{N}} e^{-\theta u}  \mathrm{d}u-\kappa^{{\top}}\displaystyle\int_0^{\frac{1}{N}} e^{\theta (u-{\frac{1}{N}})} e^{-b u}\,\mathrm{d}u\,Y_{\frac{i-1}{N}}\\&-\kappa^{{\top}}\displaystyle\int_0^{\frac{1}{N}} e^{\theta (u-{\frac{1}{N}})}\int_0^{u} e^{-b(u-v)}\mathrm{d}v\,\mathrm{d}u\,a.
\end{align*}
Consequently, by combining the relation \eqref{extremum problem} with the above explicit expressions of conditional expectations, we get 
\begin{align*}
		\check{\tau}_{T,N}=\underset{\tau\in \R^{9}}{\arg\min}\displaystyle\sum_{i=1}^{ \lfloor NT \rfloor}&\left[\left(Y_{\frac{i}{N}}^{(1)}-Y_{\frac{i-1}{N}}^{(1)}-\left( \tilde{a}_1- \tilde{b}_{11}Y_{\frac{i-1}{N}}^{(1)} \right)\right)^2+(Y_{\frac{i}{N}}^{(2)}-Y_{\frac{i-1}{N}}^{(2)}-\left( \tilde{a}_2- \tilde{b}_{21}Y_{\frac{i-1}{N}}^{(1)}-\tilde{b}_{22}Y_{\frac{i-1}{N}}^{(2)} \right)\right)^2
		\\&\left.+\left(X_{\frac{i}{N}}-X_{\frac{i-1}{N}}-\left( \tilde{m} - \tilde{\kappa}^{{\top}}Y_{\frac{i-1}{N}} -\tilde{\theta} X_{\frac{i-1}{N}}  \right)\right)^2\right],
\end{align*}
with $\tilde{a}_1:=a_1\displaystyle\int_0^{\frac{1}{N}} e^{-b_{11}u}\mathrm{d}u$, $\tilde{b}_{11}:=1-e^{-\frac{b_{11}}{N}}$, $\tilde{a}_2:=a_2\displaystyle\int_0^{\frac{1}{N}} e^{-b_{22}u}\mathrm{d}u -b_{21}a_1\displaystyle\int_0^{\frac{1}{N}} e^{b_{22} (u-\frac{1}{N})}\displaystyle\int_0^{u} e^{-b_{11}(u-v)}\mathrm{d}v\,\mathrm{d}u$, $\tilde{b}_{21}:=b_{21}\displaystyle\int_0^{\frac{1}{N}} e^{b_{22} (u-\frac{1}{N})}e^{-b_{11} u}\, \mathrm{d}u$, $\tilde{b}_{22}:=1-e^{-\frac{b_{22}}{N}}$, $\tilde{m}:=m\displaystyle\int_0^{\frac{1}{N}} e^{-\theta u}\,\mathrm{d}u-\kappa^{{\top}}\displaystyle\int_0^{\frac{1}{N}} e^{\theta (u-{\frac{1}{N}})}\int_0^{u} e^{-b(u-v)}\mathrm{d}v\,\mathrm{d}u\,a$, $\tilde{\kappa}^{{\top}}:=\kappa^{{\top}}\displaystyle\int_0^{\frac{1}{N}} e^{\theta (u-{\frac{1}{N}})} e^{-b u}\,\mathrm{d}u$ and $\tilde{\theta}:=1-e^{-\frac{\theta}{N}}$. Let the function $g_N$ be defined, for all $N\in\N^*$, by 
\begin{equation} g_N(a_1,b_{11},a_2,b_{21},b_{22},m,\kappa_1,\kappa_2,\theta)=(\tilde{a}_1,\tilde{b}_{11},\tilde{a}_2,\tilde{b}_{21},\tilde{b}_{22},\tilde{m},\tilde{\kappa}_1,\tilde{\kappa}_2,\tilde{\theta})
\label{the function g}
\end{equation}
 and let us introduce the vector $(\tilde{a}_1^{T,N},\tilde{b}_{11}^{T,N},\tilde{a}_2^{T,N},\tilde{b}_{21}^{T,N},\tilde{b}_{22}^{T,N},\tilde{m}^{T,N},\tilde{\kappa}_1^{T,N},\tilde{\kappa}_2^{T,N},\tilde{\theta}^{T,N})^{{\top}}$ as the CLSE associated to the parameter vector $(\tilde{a}_1,\tilde{b}_{11},\tilde{a}_2,\tilde{b}_{21},\tilde{b}_{22},\tilde{m},\tilde{\kappa}_1,\tilde{\kappa}_2,\tilde{\theta})^{{\top}}$.
 Then, it can be observed that we have
\begin{equation}
	(\tilde{a}_1^{T,N},\tilde{b}_{11}^{T,N})=\underset{(\tilde{a}_1,\tilde{b}_{11})\in \R^2}{\arg\min}\displaystyle\sum_{i=1}^{ \lfloor NT \rfloor}\left(Y_{\frac{i}{N}}^{(1)}-Y_{\frac{i-1}{N}}^{(1)}-\left( \tilde{a}_1- \tilde{b}_{11}Y_{\frac{i-1}{N}}^{(1)} \right)\right)^2,
	\label{c d}
\end{equation}
\begin{equation}
(\tilde{a}_2^{T,N},\tilde{b}_{21}^{T,N},\tilde{b}_{22}^{T,N})=\underset{(\tilde{a}_2,\tilde{b}_{21},\tilde{b}_{22})\in \R^3}{\arg\min}\displaystyle\sum_{i=1}^{ \lfloor NT \rfloor}\left(Y_{\frac{i}{N}}^{(2)}-Y_{\frac{i-1}{N}}^{(2)}-\left( \tilde{a}_2- \tilde{b}_{21}Y_{\frac{i-1}{N}}^{(1)}-\tilde{b}_{22}Y_{\frac{i-1}{N}}^{(2)} \right)\right)^2
\label{c' d'}
\end{equation}
and
\begin{equation}
	\begin{aligned}
		(\tilde{m}^{T,N},\tilde{\kappa}_1^{T,N},\tilde{\kappa}_2^{T,N},\tilde{\theta}^{T,N})=\underset{(\tilde{m},\tilde{\kappa}_1,\tilde{\kappa}_2,\tilde{\theta})\in \R^{4}}{\arg\min}\displaystyle\sum_{i=1}^{ \lfloor NT \rfloor}\left(X_{\frac{i}{N}}-X_{\frac{i-1}{N}}-\left( \tilde{m} - \tilde{\kappa}^{{\top}}Y_{\frac{i-1}{N}} -\tilde{\theta} X_{\frac{i-1}{N}}  \right)\right)^2.
	\end{aligned}
	\label{delta varepsilon zeta}
\end{equation}

Hence, in order to solve first the extremum problem $\eqref{c d}$, it is enough to solve at first the system
\begin{equation}\label{derivative 1 equals zero}
	\begin{cases}
		\displaystyle\sum_{i=1}^{ \lfloor NT \rfloor} Y_{\frac{i}{N}}^{(1)}-Y_{\frac{i-1}{N}}^{(1)}-\left( \tilde{a}_1- \tilde{b}_{11}Y_{\frac{i-1}{N}}^{(1)} \right)=0,\\
		\displaystyle\sum_{i=1}^{ \lfloor NT \rfloor} \left(Y_{\frac{i}{N}}^{(1)}-Y_{\frac{i-1}{N}}^{(1)}-\left( \tilde{a}_1- \tilde{b}_{11}Y_{\frac{i-1}{N}}^{(1)} \right)\right)Y^{(1)}_{\frac{i-1}{N}}=0,
	\end{cases}
\end{equation}
which can be written as follows
$\Gamma^{(1)}_{T,N}\begin{bmatrix}
		\tilde{a}_1\\\tilde{b}_{11}
	\end{bmatrix}=\phi^{(1)}_{T,N},
$
with
\begin{equation*}
	\Gamma^{(1)}_{T,N}=\begin{bmatrix}
		\lfloor NT \rfloor& -\displaystyle\sum_{i=1}^{ \lfloor NT \rfloor}Y_{\frac{i-1}{N}}\\
		-\displaystyle\sum_{i=1}^{ \lfloor NT \rfloor}Y_{\frac{i-1}{N}}& \displaystyle\sum_{i=1}^{ \lfloor NT \rfloor}Y_{\frac{i-1}{N}}^2
	\end{bmatrix}\quad\text{and}\quad \phi^{(1)}_{T,N}=\begin{bmatrix}
		Y_{\frac{\lfloor NT \rfloor}{N}}-Y_{0}\\
		-\displaystyle\sum_{i=1}^{ \lfloor NT \rfloor} \left(Y_{\frac{i}{N}}-Y_{\frac{i-1}{N}}\right)Y_{\frac{i-1}{N}}
	\end{bmatrix}.
\end{equation*}
Then, provided that $\lfloor NT \rfloor \displaystyle\sum_{i=1}^{ \lfloor NT \rfloor}Y_{\frac{i-1}{N}}^2 -\left(\displaystyle\sum_{i=1}^{ \lfloor NT \rfloor}Y_{\frac{i-1}{N}}\right)^2>0$, we deduce that
\begin{equation*}
\begin{bmatrix}
	\tilde{a}_1^{T,N}\\\tilde{b}_{11}^{T,N}
\end{bmatrix}=\left(\Gamma^{(1)}_{T,N}\right)^{-1}\phi^{(1)}_{T,N}.	
\end{equation*}
For the second extremum problem $\eqref{c' d'}$, it is sufficient to solve the following system
\begin{equation}\label{derivative 1 prime equals 0}
	\begin{cases}
		\displaystyle\sum_{i=1}^{ \lfloor NT \rfloor} (Y_{\frac{i}{N}}^{(2)}-Y_{\frac{i-1}{N}}^{(2)}-\left( \tilde{a}_2- \tilde{b}_{21}Y_{\frac{i-1}{N}}^{(1)}-\tilde{b}_{22}Y_{\frac{i-1}{N}}^{(2)} \right) =0\\
		\displaystyle\sum_{i=1}^{ \lfloor NT \rfloor}\left((Y_{\frac{i}{N}}^{(2)}-Y_{\frac{i-1}{N}}^{(2)}-\left( \tilde{a}_2- \tilde{b}_{21}Y_{\frac{i-1}{N}}^{(1)}-\tilde{b}_{22}Y_{\frac{i-1}{N}}^{(2)} \right)\right)Y_{\frac{i-1}{N}}^{(1)}=0\\
		\displaystyle\sum_{i=1}^{ \lfloor NT \rfloor}\left( (Y_{\frac{i}{N}}^{(2)}-Y_{\frac{i-1}{N}}^{(2)}-\left( \tilde{a}_2- \tilde{b}_{21}Y_{\frac{i-1}{N}}^{(1)}-\tilde{b}_{22}Y_{\frac{i-1}{N}}^{(2)} \right)\right)Y^{(2)}_{\frac{i-1}{N}}=0.
	\end{cases}
\end{equation}
The above system can be written in the following matrix form
$
	\Gamma^{(2)}_{T,N}\begin{bmatrix}
		\tilde{a}_2\\
		\tilde{b}_{21}\\
		\tilde{b}_{22}
	\end{bmatrix} = \phi^{(2)}_{T,N},
$
with
\begin{equation*}
	\Gamma^{(2)}_{T,N}=\begin{bmatrix}
		\lfloor NT \rfloor & -\displaystyle\sum_{i=1}^{ \lfloor NT \rfloor}Y^{(1)}_{\frac{i-1}{N}}  &-\displaystyle\sum_{i=1}^{ \lfloor NT \rfloor}Y^{(2)}_{\frac{i-1}{N}}\\
		-\displaystyle\sum_{i=1}^{ \lfloor NT \rfloor}Y^{(1)}_{\frac{i-1}{N}} & \displaystyle\sum_{i=1}^{ \lfloor NT \rfloor}\left(Y^{(1)}_{\frac{i-1}{N}}\right)^2 & \displaystyle\sum_{i=1}^{ \lfloor NT \rfloor} Y^{(1)}_{\frac{i-1}{N}}Y^{(2)}_{\frac{i-1}{N}}\\
		-\displaystyle\sum_{i=1}^{ \lfloor NT \rfloor}Y^{(2)}_{\frac{i-1}{N}} &  \displaystyle\sum_{i=1}^{ \lfloor NT \rfloor} Y^{(1)}_{\frac{i-1}{N}}Y^{(2)}_{\frac{i-1}{N}} & \displaystyle\sum_{i=1}^{ \lfloor NT \rfloor} \left(Y^{(2)}_{\frac{i-1}{N}}\right)^2
	\end{bmatrix}\quad
	\text{and}\quad \phi^{(2)}_{T,N}=\begin{bmatrix}
		Y^{(2)}_{\frac{\lfloor NT \rfloor}{N}}-Y^{(2)}_{0}\\
		-\displaystyle\sum_{i=1}^{ \lfloor NT \rfloor} Y^{(1)}_{\frac{i-1}{N}}\left(Y^{(2)}_{\frac{i}{N}}-Y^{(2)}_{\frac{i-1}{N}}\right)\\
		-\displaystyle\sum_{i=1}^{ \lfloor NT \rfloor} Y^{(2)}_{\frac{i-1}{N}}\left(Y^{(2)}_{\frac{i}{N}}-Y^{(2)}_{\frac{i-1}{N}}\right)
	\end{bmatrix}.
\end{equation*}
Then, provided that $\Gamma^{(2)}_{T,N}$ is invertible, we deduce that
\begin{equation*}
	\begin{bmatrix}
		\tilde{a}_2^{T,N}\\
		\tilde{b}_{21}^{T,N}\\
		\tilde{b}_{22}^{T,N}
	\end{bmatrix}=\left(\Gamma^{(2)}_{T,N}\right)^{-1}\phi^{(2)}_{T,N}.
\end{equation*}
Next, in order to solve the third extremum problem $\eqref{delta varepsilon zeta}$, it is enough to solve the system
\begin{equation}\label{derivative 2 equals 0}
	\begin{cases}
		\displaystyle\sum_{i=1}^{ \lfloor NT \rfloor} X_{\frac{i}{N}}-X_{\frac{i-1}{N}}-\left( \tilde{m} - \tilde{\kappa}^{{\top}}Y_{\frac{i-1}{N}} -\tilde{\theta} X_{\frac{i-1}{N}}  \right) =0\\
		\displaystyle\sum_{i=1}^{ \lfloor NT \rfloor}\left(X_{\frac{i}{N}}-X_{\frac{i-1}{N}}-\left( \tilde{m} - \tilde{\kappa}^{{\top}}Y_{\frac{i-1}{N}} -\tilde{\theta} X_{\frac{i-1}{N}}  \right)\right)Y_{\frac{i-1}{N}}=\mathbf{0}_2\\
		\displaystyle\sum_{i=1}^{ \lfloor NT \rfloor}\left(X_{\frac{i}{N}}-X_{\frac{i-1}{N}}-\left( \tilde{m} - \tilde{\kappa}^{{\top}}Y_{\frac{i-1}{N}} -\tilde{\theta} X_{\frac{i-1}{N}}\right) \right)X_{\frac{i-1}{N}}=0.
	\end{cases}
\end{equation}
The above system can be written in the following matrix form
$
	\Gamma^{(3)}_{T,N}\begin{bmatrix}
		\tilde{m}\\
		\tilde{\kappa}\\
		\tilde{\theta}
	\end{bmatrix} = \phi^{(3)}_{T,N},
$
with
\begin{equation*}
	\Gamma^{(3)}_{T,N}=\begin{bmatrix}
		\lfloor NT \rfloor & -\displaystyle\sum_{i=1}^{ \lfloor NT \rfloor}Y_{\frac{i-1}{N}}^{{\top}}  &-\displaystyle\sum_{i=1}^{ \lfloor NT \rfloor}X_{\frac{i-1}{N}}\\
		-\displaystyle\sum_{i=1}^{ \lfloor NT \rfloor}Y_{\frac{i-1}{N}} & \displaystyle\sum_{i=1}^{ \lfloor NT \rfloor}Y_{\frac{i-1}{N}}Y_{\frac{i-1}{N}}^{{\top}} & \displaystyle\sum_{i=1}^{ \lfloor NT \rfloor} X_{\frac{i-1}{N}}Y_{\frac{i-1}{N}}\\
		-\displaystyle\sum_{i=1}^{ \lfloor NT \rfloor}X_{\frac{i-1}{N}} &  \displaystyle\sum_{i=1}^{ \lfloor NT \rfloor} X_{\frac{i-1}{N}}Y_{\frac{i-1}{N}}^{{\top}} & \displaystyle\sum_{i=1}^{ \lfloor NT \rfloor} X_{\frac{i-1}{N}}^2
	\end{bmatrix}
	\quad \text{and}\quad \phi^{(3)}_{T,N}=\begin{bmatrix}
		X_{\frac{\lfloor NT \rfloor}{N}}-X_{0}\\
		-\displaystyle\sum_{i=1}^{ \lfloor NT \rfloor} \left(X_{\frac{i}{N}}-X_{\frac{i-1}{N}}\right)Y_{\frac{i-1}{N}}\\
		-\displaystyle\sum_{i=1}^{ \lfloor NT \rfloor} X_{\frac{i-1}{N}}\left(X_{\frac{i}{N}}-X_{\frac{i-1}{N}}\right)
	\end{bmatrix}.
\end{equation*}
Then, provided that $\Gamma^{(3)}_{T,N}$ is invertible, we deduce that
\begin{equation*}
	\begin{bmatrix}
		\tilde{m}^{T,N}\\
		\tilde{\kappa}^{T,N}\\
		\tilde{\theta}^{T,N}
\end{bmatrix}=\left(\Gamma^{(3)}_{T,N}\right)^{-1}\phi^{(3)}_{T,N}.
\end{equation*}
Consequently, it remains to prove the invertibility of the matrices $\Gamma^{(1)}_{T,N},\ \Gamma^{(2)}_{T,N}$ and $\Gamma^{(3)}_{T,N}$. In our case, it is enough to show that they are almost surely positive definite. Let $x\in\R^2\setminus\lbrace \mathbf{0}_2 \rbrace$, $y\in\R^3\setminus\lbrace \mathbf{0}_3 \rbrace$ and $z\in\R^{4}\setminus\lbrace \mathbf{0}_{4} \rbrace$. On the one hand, we have
\begin{equation}
	\begin{bmatrix}
		x_1\\
		x_2
	\end{bmatrix}^{{\top}}\Gamma^{(1)}_{T,N} \begin{bmatrix}
		x_1\\
		x_2
	\end{bmatrix} =\displaystyle\sum_{i=1}^{ \lfloor NT \rfloor}\begin{bmatrix}
		x_1\\
		x_2
	\end{bmatrix}^{{\top}}
	\begin{bmatrix}
		1\\
		- Y^{(1)}_{\frac{i-1}{N}}
	\end{bmatrix}
	\begin{bmatrix}
		1\\
		-Y^{(1)}_{\frac{i-1}{N}}
	\end{bmatrix}^{{\top}}
	\begin{bmatrix}
		x_1\\
		x_2
	\end{bmatrix}= \displaystyle\sum_{i=1}^{ \lfloor NT \rfloor} (x_1-x_2 Y^{(1)}_{\frac{i-1}{N}})^2\geq 0.
\label{invertibility 1}
\end{equation}
On the other hand,
$	\begin{bmatrix}
		x_1\\
		x_2
	\end{bmatrix}^{{\top}}\Gamma^{(1)}_{T,N} \begin{bmatrix}
		x_1\\
		x_2
	\end{bmatrix} =0\ \Leftrightarrow \ x_1-x_2 Y_{\frac{i-1}{N}}=0,$ for all $ i\in\lbrace 1,2,\ldots,\lfloor NT\rfloor\rbrace, 
$ which is impossible since $x\neq 0_2$ and, for each $i\in \lbrace 1,2,\ldots,\lfloor NT\rfloor\rbrace $, the distribution of $Y^{(1)}_{\frac{i-1}{N}}$ is absolutely continuous because its conditional distribution given $Y^{(1)}_0$ is absolutely continuous.
Likewise, we have
\begin{equation}
\begin{bmatrix}
		y_1\\
		y_2\\
  y_3
	\end{bmatrix}^{{\top}}\Gamma^{(2)}_{T,N} \begin{bmatrix}
		y_1\\
		y_2\\
  y_3
	\end{bmatrix} = \displaystyle\sum_{i=1}^{ \lfloor NT \rfloor} (y_1-y_2 Y^{(1)}_{\frac{i-1}{N}}-y_3 Y^{(2)}_{\frac{i-1}{N}})^2\geq 0.
\label{invertibility 2}
\end{equation}
In addition,
$y^\top\Gamma_{T,N}^{(2)}y=0$
if and only if
$y_1-y_2Y_{(i-1)/N}^{(1)}-y_3Y_{(i-1)/N}^{(2)}=0$
for all $i\in\{1,\ldots,\lfloor NT\rfloor\}$,
which is also impossible since $y\neq \mathbf{0}_{3}$ and, for each $i\in \lbrace 1,2,\ldots,\lfloor NT\rfloor\rbrace $, the distribution of $Y_{\frac{i-1}{N}}$ is absolutely continuous because its conditional distribution given $Y_0$ is absolutely continuous. Furthermore, by similar arguments, one easily checks that $z^{{\top}}\Gamma^{(3)}_{T,N}z>0$. Now, we define the approximate CLSE of $\tau$ by 
\begin{equation*}
	\check{\tau}_{T,N}^{\mathrm{approx}}:=N(\tilde{a}_1^{T,N},\tilde{b}_{11}^{T,N},\tilde{a}_2^{T,N},\tilde{b}_{21}^{T,N},\tilde{b}_{22}^{T,N},\tilde{m}^{T,N},\tilde{\kappa}_1^{T,N},\tilde{\kappa}_2^{T,N},\tilde{\theta}^{T,N})^{{\top}}.
\end{equation*}
Thanks to the almost sure continuity of  $(Y_t,X_t)_{t\in \R_{+}}$ and by Proposition $4.44$ in Jacod and Shirayev~\cite{Jacod}, we obtain
\begin{equation}
	\dfrac{1}{N}\Gamma^{(1)}_{T,N}\stackrel{a.s.}{\longrightarrow}\begin{bmatrix}
		T & -\displaystyle\int_0^T Y^{(1)}_s \mathrm{d}s\\[8pt]
		-\displaystyle\int_0^T Y^{(1)}_s \mathrm{d}s & \displaystyle\int_0^T \left(Y^{(1)}_s\right)^2 \mathrm{d}s
	\end{bmatrix}=:G_T^{(1)},\ 	\phi^{(1)}_{T,N}\stackrel{\mathbb{P}}{\longrightarrow}\begin{bmatrix}
	Y^{(1)}_T-Y^{(1)}_{0}\\
	-\displaystyle\int_0^T Y^{(1)}_s \mathrm{d}Y^{(1)}_s\end{bmatrix}:=f_T^{(1)},
\label{fT1 GT1}
\end{equation}
\begin{equation}
	\dfrac{1}{N}\Gamma^{(2)}_{T,N}\stackrel{a.s.}{\longrightarrow}\begin{bmatrix}
		T & -\displaystyle\int_0^T Y_s^{{\top}} \mathrm{d}s\\[8pt]
		-\displaystyle\int_0^T Y_s \mathrm{d}s & \displaystyle\int_0^T Y_s Y_s^{{\top}} \mathrm{d}s	\end{bmatrix}=:G_T^{(2)},\  \phi^{(2)}_{T,N}\stackrel{\mathbb{P}}{\longrightarrow}\begin{bmatrix}
	Y^{(2)}_{T}-Y^{(2)}_{0}\\[2pt]
	-\displaystyle\int_0^T Y_{s} \mathrm{d}Y^{(2)}_{s}
\end{bmatrix}:=f_T^{(2)},
\label{fT1' GT1'}
\end{equation}
and
\begin{equation}
	\dfrac{1}{N}\Gamma^{(3)}_{T,N}\stackrel{a.s.}{\longrightarrow}\begin{bmatrix}
		T & -\displaystyle\int_0^T Z_s^{{\top}} \mathrm{d}s\\[8pt]
		-\displaystyle\int_0^T Z_s \mathrm{d}s & \displaystyle\int_0^T Z_s Z_s^{{\top}}  \mathrm{d}s 
	\end{bmatrix}=:G_T^{(3)},\ \phi^{(3)}_{T,N}\stackrel{\mathbb{P}}{\longrightarrow}\begin{bmatrix}
	X_{T}-X_{0}\\[2pt]
	-\displaystyle\int_0^T Z_{s} \mathrm{d}X_{s}
\end{bmatrix}:=f_T^{(3)},
\label{fT2 GT2}
\end{equation}
as $N\to\infty$. Hence, provided that $G_T^{(1)}$, $G_T^{(2)}$ and $G_T^{(3)}$  are invertible and using relations \eqref{fT1 GT1}, \eqref{fT1' GT1'} and \eqref{fT2 GT2}, Slutsky's theorem yields
\begin{equation}
\check{\tau}_{T,N}^{\mathrm{approx}}\stackrel{\mathbb{P}}{\longrightarrow}G_T^{-1}f_T=:\check{\tau}_T,\quad \text{as } N\to\infty,
	\label{tau approx}
\end{equation}
where $G_T=\mathrm{diag}\left(G_T^{(1)},G_T^{(2)},G_T^{(3)}\right)$ and $f_T=\begin{bmatrix}
    f_T^{(1)}\\f_T^{(2)}\\f_T^{(3)}
\end{bmatrix}$. Hence, by the same arguments used in the proof of the invertibility of the matrices $\Gamma^{(1)}_{T,N}$, $\Gamma^{(2)}_{T,N}$ and $\Gamma^{(3)}_{T,N}$ with replacing the sum from $0$ to $\lfloor NT \rfloor$ by the integral on $[0,T]$, one easily checks that $G_T$ is invertible.  
Note that this approximate CLSE is constructed through the regular part of a first order Taylor approximation of the function $g_N$ at zero point. In the following lemma we prove that $\check{\tau}_T$ is also the limit of the CLSE $\check{\tau}_{T,N}$ introduced in \eqref{extremum problem}. Hence, $\check{\tau}_T$ is called the CLSE based on continuous-time observations.  
\begin{lemme}
	Let us consider the affine model $\eqref{model0 Z}$, then for each $T\in\R_{++}$, we have $$\check{\tau}_{T,N}:=(\check{a}_1^{T,N},\check{b}_{11}^{T,N},\check{a}_2^{T,N},\check{b}_{21}^{T,N},\check{b}_{22}^{T,N},\check{m}^{T,N},\check{\kappa}_1^{T,N},\check{\kappa}_2^{T,N},\check{\theta}^{T,N})\stackrel{\mathbb{P}}{\longrightarrow}\check{\tau}_T,\quad \text{as }N\to\infty.$$ 
	\label{lemma hat tau T N to hat tau}
\end{lemme}
\begin{proof}
Thanks to relation \eqref{tau approx}, one easily checks that  
\begin{equation}
(\tilde{a}_1^{T,N},\tilde{b}_{11}^{T,N},\tilde{a}_2^{T,N},\tilde{b}_{21}^{T,N},\tilde{b}_{22}^{T,N},\tilde{m}^{T,N},\tilde{\kappa}_1^{T,N},\tilde{\kappa}_2^{T,N},\tilde{\theta}^{T,N})^\top=\frac{1}{N}\check{\tau}_{T,N}^{\mathrm{approx}}\stackrel{\mathbb{P}}{\longrightarrow} \mathbf{0}_9,\quad \text{as } N\to\infty.
	\label{tau hat T N goes to zero}
\end{equation}

In addition, it is easy to check that the function $g_N:\R^9\to \R\times (-\infty,1)\times \R^2 \times (-\infty,1)\times \R^3\times (-\infty,1)$ defined by \eqref{the function g} admits an inverse function in a neighborhood of the origin, which implies that $$g_N^{-1}(\tilde{a}_1^{T,N},\tilde{b}_{11}^{T,N},\tilde{a}_2^{T,N},\tilde{b}_{21}^{T,N},\tilde{b}_{22}^{T,N},\tilde{m}^{T,N},\tilde{\kappa}_1^{T,N},\tilde{\kappa}_2^{T,N},\tilde{\theta}^{T,N})=\check{\tau}_{T,N},$$ with probability tending to one as $N\to \infty$. Furthermore, we can deduce the following expressions $$\check{b}_{11}^{T,N}=-N\log(1-\tilde{b}_{11}^{T,N}),\quad\check{b}_{22}^{T,N}=-N\log(1-\tilde{b}_{22}^{T,N}),\quad\check{\theta}_{T,N}=-N\log(1-\tilde{\theta}_{T,N}),$$
$$\check{a}_1^{T,N}=\left(\int_0^{\frac{1}{N}}e^{-\check{b}_{11}^{T,N}u}\mathrm{d}u\right)^{-1}\tilde{a}_1^{T,N},\quad \check{b}_{21}^{T,N}=\left(\displaystyle\int_0^{\frac{1}{N}} e^{\check{b}_{22}^{T,N} (u-{\frac{1}{N}})} e^{-\check{b}_{11}^{T,N} u}\mathrm{d}u \right)^{-1}\tilde{b}_{21}^{T,N},$$
$$\check{a}_2^{T,N}=\left(\displaystyle\int_0^{\frac{1}{N}} e^{-\check{b}_{22}^{T,N} u} \mathrm{d}u \right)^{-1}\left(\tilde{a}_2^{T,N}+\check{a}_1^{T,N}\displaystyle\int_0^{\frac{1}{N}} e^{(u-{\frac{1}{N}})\check{b}_{22}^{T,N}}\check{b}_{21}^{T,N}\left(\displaystyle\int_0^{u} e^{-\check{b}_{11}^{T,N}(u-v)}\mathrm{d}v\right)\mathrm{d}u\right),$$
$$\left(\check{\kappa}^{T,N}\right)^{{\top}}=\left(\tilde{\kappa}^{T,N}\right)^{{\top}}\left(\displaystyle\int_0^{\frac{1}{N}} e^{\check{\theta}^{T,N} (u-{\frac{1}{N}})} e^{-\check{b}^{T,N} u}\mathrm{d}u \right)^{-1}$$ and $$\check{m}^{T,N}=\left(\displaystyle\int_0^{\frac{1}{N}} e^{-\check{\theta}^{T,N} u} \mathrm{d}u \right)^{-1}\left(\tilde{m}^{T,N}+\left(\check{\kappa}^{T,N}\right)^{{\top}}\displaystyle\int_0^{\frac{1}{N}} e^{\check{\theta}_{T,N}(u-{\frac{1}{N}})}\displaystyle\int_0^{u} e^{-\check{b}^{T,N}(u-v)}\mathrm{d}v\,\mathrm{d}u\right)\check{a}^{T,N}.$$
Consequently, following the same arguments as in the proof of Lemma~3.1 in~\cite{Dahbi2} (see also Lemma~3.3 in~\cite{Boylog1}), and using first-order Taylor expansions at zero of the functions $x \mapsto \log(1 - x)$ and $x \mapsto \frac{1 - e^{-x}}{x}$, we readily obtain the desired result.
\end{proof}
To study the asymptotic properties of the continuous-time CLSE,
we first derive a representation of the estimation error. Using the SDE \eqref{model0 Z}, we can easily deduce the following equalities $f_T^{(1)}=G_T^{(1)}\begin{bmatrix}
		a_1\\
		b_{11}
	\end{bmatrix}+h_T^{(1)}$, 
 $f_T^{(2)}=G_T^{(2)}\begin{bmatrix}
		a_2\\
		b_{21}\\
		b_{22}
	\end{bmatrix}+h_T^{(2)}\ \text{and}\  f_T^{(3)}=G_T^{(3)}\begin{bmatrix}
		m\\
		\kappa\\
		\theta
	\end{bmatrix}+h_T^{(3)},$ 
where $h_T^{(1)}=\begin{bmatrix}
	\sigma_{11}\displaystyle\int_0^T\sqrt{Y^{(1)}_s}\mathrm{d}B_s^1\\
	-\sigma_{11}\displaystyle\int_0^TY^{(1)}_s \sqrt{Y^{(1)}_s}\mathrm{d}B_s^1
\end{bmatrix}$,\\  $$h_T^{(2)}=\begin{bmatrix} \sigma_{12}\displaystyle\int_0^T\sqrt{Y^{(2)}_s}\mathrm{d}B_s^2\\[8pt]
	-\sigma_{12}\displaystyle\int_0^TY_s\sqrt{Y^{(2)}_s}\mathrm{d}B_s^2
\end{bmatrix}\ \text{and}\ h_T^{(3)}=\begin{bmatrix} \displaystyle\int_0^T \sigma_2^{{\top}}S(Y_s)(\rho \mathrm{~d}B_s+\bar{\rho}\mathrm{~d}W_s)\\[8pt]
	- \displaystyle\int_0^T Z_s\sigma_2^{{\top}}S(Y_s)(\rho \mathrm{~d}B_s+\bar{\rho}\mathrm{~d}W_s)
\end{bmatrix}.$$

Therefore, we get 
\begin{equation}
	\check{\tau}_T-\tau=G_T^{-1}h_T,\quad\text{where } h_T:=\begin{bmatrix}
h_T^{(1)}\\
h_T^{(2)}\\
h_T^{(3)}
\end{bmatrix}.
	\label{err}
\end{equation}
\subsection{Consistency of the CLSE}
With the estimator explicitly defined, the first step in validating its usefulness is to confirm its consistency. This fundamental property ensures that, as $T \to \infty$, the CLSE $\check{\tau}_T$ will converge to the true underlying parameter vector $\tau$.
\begin{theoreme}\label{consistency CLSE b>0}
	Let us consider the affine diffusion model $\eqref{model0 Z}$ with initial random values independent of $(B^1_t,B_t^2,W_t)_{t\in\R_{+}}$ satisfying $\mathbb{P}(Y_0^{(1)}\in\R_{++})=1$ and $\mathbb{P}(Y_0^{(2)}\in\R_{++})=1$. If $a_1,b_{11},b_{22},\theta\in\R_{++}$ and $a_2\in(\frac{\sigma_{12}^2}{2},\infty)$, then the $CLSE$ of $\tau$ is strongly consistent, i.e., $\check{\tau}_T \stackrel{a.s.}{\longrightarrow} \tau$, as $T\to \infty$. 
\end{theoreme}
\begin{proof}
	By relation$\eqref{err}$, we have
$
		\check{\tau}_T-\tau=(T^{-1} G_T)^{-1}(T^{-1}h_T)
$. Furthermore,	using Theorem 4.1 in~\cite{Dahbi}, we get
	\begin{equation*}
		T^{-1}G_T \stackrel{a.s.}{\longrightarrow} \mathbb{E}(G_{\infty}),\quad \text{as } T\to \infty,
	\end{equation*}
	where 
	\begin{equation}
		G_{\infty}:=\mathrm{diag}\left(
			G_{\infty}^{(1)}, G_{\infty}^{(2)},G_{\infty}^{(3)}\right),
            \label{Ginfty}
	\end{equation}
	with
	\begin{equation*}
		G_{\infty}^{(1)}:=\begin{bmatrix}
			1 & - Y^{(1)}_{\infty} \\
			- Y^{(1)}_{\infty} & \left(Y^{(1)}_{\infty}\right)^2
		\end{bmatrix},\quad
		G_{\infty}^{(2)}:=\begin{bmatrix}
			1 & - Y_{\infty}^{{\top}} \\[4pt]
			- Y_{\infty} & Y_{\infty}Y_{\infty}^{{\top}}
		\end{bmatrix}\ \text{and}\ G_{\infty}^{(3)}:= \begin{bmatrix}
			1 & - Z_{\infty}^{{\top}} \\[4pt]
			- Z_{\infty} & Z_{\infty}Z_{\infty}^{{\top}}
		\end{bmatrix},
	\end{equation*}
	where $Y_\infty=(Y_\infty^{(1)},Y_\infty^{(2)})^\top$ and $Z_\infty=(Y_\infty,X_\infty)^\top$ is defined by the Fourier–Laplace transform given by relation \eqref{loi limite}. It is worth noting that $\mathbb{E}(G_{\infty})$ is finite. In addition, it is an invertible matrix since it is positive definite by the same argument used in relation \eqref{invertibility 1}. Hence, we deduce that
	\begin{equation*}
		(T^{-1}G_T)^{-1}\stackrel{a.s.}{\longrightarrow} (\mathbb{E}(G_\infty))^{-1},\quad \text{as } T\to \infty.
	\end{equation*}
	In order to finish the proof, it is enough to check that
$T^{-1}h_T \stackrel{a.s.}{\longrightarrow} \mathbf{0},$ as $T\to \infty$. Firstly, we have  $\langle h^{(1)}_T,\left.h^{(1)}_T\right.^{{\top}}\rangle$, $\langle h^{(1)}_T,\left.h^{(2)}_T\right.^{{\top}}\rangle$ $\langle h_T^{(1)},\left.h^{(3)}_{T}\right.^{{\top}}\rangle$, $\langle h^{(2)}_{T},\left.h^{(2)}_{T}\right.^{{\top}}\rangle$, $\langle h^{(2)}_{T},\left.h^{(3)}_{T}\right.^{{\top}}\rangle$ and $\langle h^{(3)}_{T},\left.h^{(3)}_{T}\right.^{{\top}}\rangle$ are given, respectively, by
$$\mathrm{H}_T^{(1,1)}=\sigma_{11}^2\displaystyle\int_{0}^T\begin{bmatrix}
		Y_s^{(1)}&-\left(Y_s^{(1)}\right)^2\\-\left(Y_s^{(1)}\right)^2&\left(Y_s^{(1)}\right)^3
	\end{bmatrix}\mathrm{d}s,\quad 
 \mathrm{H}_T^{(1,2)}=\mathbf{0}_{2,3},$$
  $$\mathrm{H}_T^{(1,3)}=\sigma_{11}\sigma_{21} \rho_{11}\displaystyle\int_{0}^T\begin{bmatrix}
	Y^{(1)}_s&-Y^{(1)}_s Z_s^{{\top}}\\
        -\left(Y^{(1)}_s\right)^2&\left(Y^{(1)}_s\right)^2 Z_s^{{\top}}
\end{bmatrix}\mathrm{d}s,\quad\mathrm{H}_T^{(2,2)}=\sigma_{12}^2\displaystyle\int_{0}^T\begin{bmatrix}
	Y^{(2)}_s&-Y^{(2)}_s Y_s^{{\top}}\\
        -Y^{(2)}_sY_s&Y^{(2)}_s Y_s Y_s^{{\top}}
\end{bmatrix}\mathrm{d}s,$$
 $$\mathrm{H}_T^{(2,3)}=\sigma_{12}\sigma_{22} \rho_{22}\displaystyle\int_{0}^T\begin{bmatrix}
	Y^{(2)}_s&-Y^{(2)}_s Z_s^{{\top}}\\
        -Y^{(2)}_sY_s&Y^{(2)}_sY_sZ_s^{{\top}}
\end{bmatrix}\mathrm{d}s$$
and
$$\mathrm{H}_T^{(3,3)}=\displaystyle\int_{0}^T\begin{bmatrix}
\sigma_{21}^2Y_s^{(1)}+\sigma_{22}^2Y_s^{(2)}&-\left(\sigma_{21}^2Y_s^{(1)}+\sigma_{22}^2Y_s^{(2)}\right)Z_s^{{\top}}\\
-\left(\sigma_{21}^2Y_s^{(1)}+\sigma_{22}^2Y_s^{(2)}\right)Z_s&\left(\sigma_{21}^2Y_s^{(1)}+\sigma_{22}^2Y_s^{(2)}\right)Z_s Z_s^{{\top}}
\end{bmatrix}\mathrm{d}s.$$ 

Secondly, by the ergodicity theorem, the quadratic variation $\mathrm{H}_T$ of $h_T$ satisfies
\begin{equation*}
\dfrac{1}{T}\mathrm{H}_T\stackrel{a.s.}{\longrightarrow} \mathbb{E}(\mathrm{H}_{\infty}),\quad\text{as }T\to\infty,
\end{equation*}  where
	\begin{equation}
		\mathrm{H}_\infty=\begin{bmatrix}
			\sigma_{11}^2\mathrm{H}_\infty^{(1,1)}&\mathbf{0}_{2,3}&\sigma_{11}\sigma_{21}\rho_{11}\mathrm{H}_\infty^{(1,3)}\\
\mathbf{0}_{3,2}&\sigma_{12}^2\mathrm{H}_\infty^{(2,2)}&\sigma_{12}\sigma_{22}\rho_{22}\mathrm{H}_\infty^{(2,3)}\\
\sigma_{11}\sigma_{21}\rho_{11}\mathrm{H}_\infty^{(3,1)}&\sigma_{12}\sigma_{22}\rho_{22}\mathrm{H}_\infty^{(3,2)}&\mathrm{H}_\infty^{(3,3)}
		\end{bmatrix},
		\label{quadratic variation of h at infinty}
	\end{equation}
	with $$\mathrm{H}_{\infty}^{(1,1)}=\begin{bmatrix}
		Y_{\infty}^{(1)}&-\left(Y_{\infty}^{(1)}\right)^2\\-\left(Y_{\infty}^{(1)}\right)^2&\left(Y_{\infty}^{(1)}\right)^3
	\end{bmatrix},\quad
\mathrm{H}_{\infty}^{(1,3)}=\begin{bmatrix}
	Y^{(1)}_{\infty}&-Y^{(1)}_{\infty} Z_{\infty}^{{\top}}\\
        -\left(Y^{(1)}_{\infty}\right)^2&\left(Y^{(1)}_{\infty}\right)^2 Z_{\infty}^{{\top}}
\end{bmatrix},$$
 $$\mathrm{H}_{\infty}^{(2,2)}=\begin{bmatrix}
	Y^{(2)}_{\infty}&-Y^{(2)}_{\infty} Y_{\infty}^{{\top}}\\
        -Y^{(2)}_{\infty}Y_{\infty}&Y^{(2)}_{\infty} Y_{\infty} Y_{\infty}^{{\top}}
\end{bmatrix},\quad\mathrm{H}_{\infty}^{(2,3)}=\begin{bmatrix}
	Y^{(2)}_{\infty}&-Y^{(2)}_{\infty} Z_{\infty}^{{\top}}\\
        -Y^{(2)}_{\infty}Y_{\infty}&Y^{(2)}_{\infty}Y_{\infty}Z_{\infty}^{{\top}}
\end{bmatrix}$$
and
$\mathrm{H}_{\infty}^{(3,3)}=\begin{bmatrix}
\sigma_{21}^2Y_{\infty}^{(1)}+\sigma_{22}^2Y_{\infty}^{(2)}&-\left(\sigma_{21}^2Y_{\infty}^{(1)}+\sigma_{22}^2Y_{\infty}^{(2)}\right)Z_{\infty}^{{\top}}\\
-\left(\sigma_{21}^2Y_{\infty}^{(1)}+\sigma_{22}^2Y_{\infty}^{(2)}\right)Z_{\infty}&\left(\sigma_{21}^2Y_{\infty}^{(1)}+\sigma_{22}^2Y_{\infty}^{(2)}\right)Z_{\infty} Z_{\infty}^{{\top}}\end{bmatrix}.$
Finally, the strong law of large numbers for continuous local martingales (see Theorem 1 in~\cite{Dahbi}) applied on $h_T$ completes the proof.
\end{proof}
\subsection{Asymptotic behavior of the CLSE}
Next, we proceed to determine the asymptotic distribution of the CLSE. The following theorem establishes the asymptotic normality of $\check{\tau}_T$.
\begin{theoreme}\label{Asymptotic Normality CLSE b>0 Continuous}
	Let us consider the affine diffusion model $\eqref{model0 Z}$ with initial random values independent of $(B^1_t,B_t^2,W_t)_{t\in\R_{+}}$ satisfying $\mathbb{P}(Y_0^{(1)}\in\R_{++})=1$ and $\mathbb{P}(Y_0^{(2)}\in\R_{++})=1$. If $a_1,b_{11},b_{22},\theta\in\R_{++}$ and $a_2\in(\frac{\sigma_{12}^2}{2},\infty)$, then the CLSE of $\tau$ is asymptotically normal, namely, 
	\begin{equation*}
		\sqrt{T}(\check{\tau}_T-\tau)\stackrel{\mathcal{D}}{\longrightarrow}\mathcal{N}_{9}({0},[\mathbb{E}(G_\infty)]^{-1}\mathbb{E}(\mathrm{H}_\infty)[\mathbb{E}(G_\infty)]^{-1})
	\end{equation*}
	where $G_{\infty}$ and $\mathrm{H}_{\infty}$ are given by relations $\eqref{Ginfty}$ and \eqref{quadratic variation of h at infinty}, respectively.
\end{theoreme}
\begin{proof}
 We have, for all $T\in\R_{++}$,
	\begin{equation}
		\sqrt{T}(\check{\tau}_T-\tau)=\left(\dfrac{1}{T}G_T\right)^{-1}\left(\dfrac{1}{\sqrt{T}}h_T\right).
		\label{proof normality}
	\end{equation}
	On one side, the ergodicity theorem implies that $\left(\dfrac{1}{T}G_T\right)^{-1}\stackrel{a.s.}{\longrightarrow}[\mathbb{E}(G_\infty)]^{-1}$, as $T\to\infty$. On the other side, by the central limit theorem given by Theorem 2 in~\cite{Dahbi}, we obtain
	\begin{equation*}
		\dfrac{1}{\sqrt{T}}h_T \stackrel{\mathcal{D}}{\longrightarrow}\mathcal{N}_{9}(0,\mathbb{E}(\mathrm{H}_{\infty})),\quad\text{as }T\to\infty.
	\end{equation*}	
	Finally, thanks to Slutsky's lemma and to the convergence $\eqref{proof normality}$, the proof is completed \end{proof}
\section{Numerical results}
The task in this section is to check the validity of some theoretical results obtained in this paper on the classification and the parametric estimation problem of the double Heston model using extensive simulations with \textit{Python}. The simulation setup involves the following steps: first, for $T\in\R_{++}$, we discretize the interval $[0, T]$ into $N$ sub-intervals of length $\Delta t$, where $N\in\N$ and we denote by $t_i=i\Delta t$, for all $i\in\lbrace 0,\cdots,N\rbrace$. Second, we generate random variables $\Delta B_t^1,\Delta B_t^2,\Delta W_t^1,\Delta W_t^2\sim \mathcal{N}(0,\Delta t)$, then we generate sample paths for the processes \(Y^{(1)}\), \(Y^{(2)}\), and \(X\) using the modified Euler-Maruyama scheme defined as follows: given $\hat{Y}_{t_0}^{(1)}\in\R_{++},\ \hat{Y}_{t_0}^{(2)}\in\R_{++}$ and $\hat{X}_{t_0}\in\R$, we consider
\begin{equation*}
    \begin{cases}
        \hat{Y}^{(1)}_{t_{i+1}} &= \hat{Y}^{(1)}_{t_i} + (a_1 - b_{11} \hat{Y}^{(1)}_{t_i}) \Delta t + \sigma_{11} \sqrt{\left\vert \hat{Y}^{(1)}_{t_i}\right\vert} \Delta B^1_{t_i}, \\
        \hat{Y}^{(2)}_{t_{i+1}} &= \hat{Y}^{(2)}_{t_i} + (a_2 - b_{21} \hat{Y}^{(1)}_{t_i} - b_{22} \hat{Y}^{(2)}_{t_i}) \Delta t + \sigma_{12} \sqrt{\left\vert \hat{Y}^{(2)}_{t_i}\right\vert} \Delta B^2_{t_i}, \\
        \hat{X}_{t_{i+1}} &= \hat{X}_{t_i} + \left( m - \kappa_1 \hat{Y}^{(1)}_{t_i} - \kappa_2 \hat{Y}^{(2)}_{t_i} - \theta \hat{X}_{t_i} \right) \Delta t \\
        & \quad + \sigma_{21} \sqrt{\left\vert \hat{Y}^{(1)}_{t_i}\right\vert} \left( \rho_{11} \Delta B^1_{t_i} + \bar{\rho}_{11} \Delta W^1_{t_i} \right) + \sigma_{22} \sqrt{\left\vert \hat{Y}^{(2)}_{t_i}\right\vert} \left( \rho_{22} \Delta B^{2}_{t_i} + \bar{\rho}_{22} \Delta W^2_{t_i} \right).
        \end{cases}
\end{equation*}
It is worth noting here that if we do not consider the absolute values of $\hat{Y}^{(1)}_{t_i}$ and $\hat{Y}^{(2)}_{t_i}$ in the scheme, then the square root is not well defined since the Gaussian increments may lead $\hat{Y}^{(1)}_{t_{i+1}}$ and $\hat{Y}^{(2)}_{t_{i+1}}$ to negative values with some positive probability. This solution was proposed by Higham and Mao~\cite{higham2005convergence} for the CIR process and they have shown the strong convergence of the scheme. In relation to this problem, many discretization schemes dedicated to the CIR process have been studied in recent years by Deelstra and Delbaen~\cite{deelstra1998convergence}, Bossy et al.~\cite{berkaoui2008euler,diop2003discretisation}, Brigo and Alfonsi~\cite{brigo2005credit,alfonsi2005discretization}, Kahl and Schurz~\cite{kahl2006balanced}, Lord et al.~\cite{lord2010comparison}, Andersen~\cite{andersen2007efficient}, and especially Alfonsi~\cite{Alfonsi2010Simulation}, where he studied higher-order schemes for the CIR process.\\

In the following, we provide a numerical illustration for the classification result given in Proposition \ref{classification}. We study first the ergodic case and we consider the following parameter set :  $b_{11}= 1$, $b_{22} = 3$, $\theta = 2$, $a_1 = 1$, $a_2 = 1$, $b_{21} = -0.5$, $m= 1$, $\kappa_1 = 0.5$, $\kappa_2 = 0.5$, $\sigma_{11} = 0.1$, $\sigma_{12} = 0.1$, $\sigma_{21} = 0.1$, $\sigma_{22} = 0.1$, $\rho_{11} = 0.8$ and $\rho_{22} = 0.8$. We fix as well the time parameters as follows $T = 50$, $\Delta t = 0.1$ and $N = \frac{T}{\Delta t} = 500$. The number of simulation runs is set to 1000. We define the initial values $Y^{(1)}_0=Y^{(2)}_0= 0.5$ and $X_0= 0$.

\begin{figure}[H]
    \centering
    \includegraphics[width=0.8\textwidth]{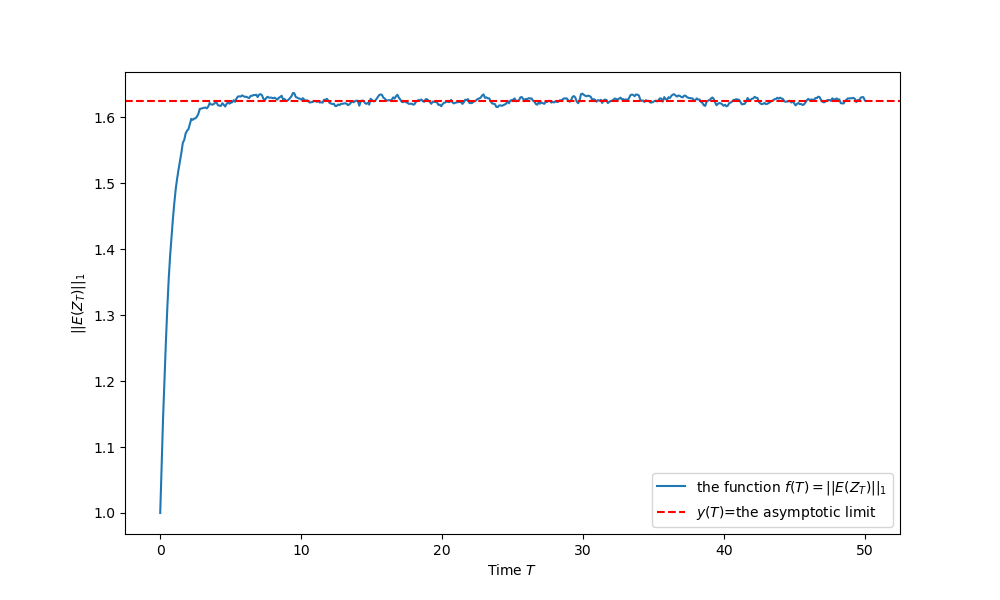}
    \caption{Asymptotic behavior of $E(Z_T)$ in the subcritical case: $(b_{11}, b_{22}, \theta) = (1, 3, 2)$.}
    \label{fig:subcritical}
\end{figure}

For the non-ergodic cases, we vary the parameters $b_{11}$, $b_{22}$ and $\theta$, while keeping the remaining parameters unchanged.

\begin{figure}[H]
    \centering
    \includegraphics[width=0.8\textwidth]{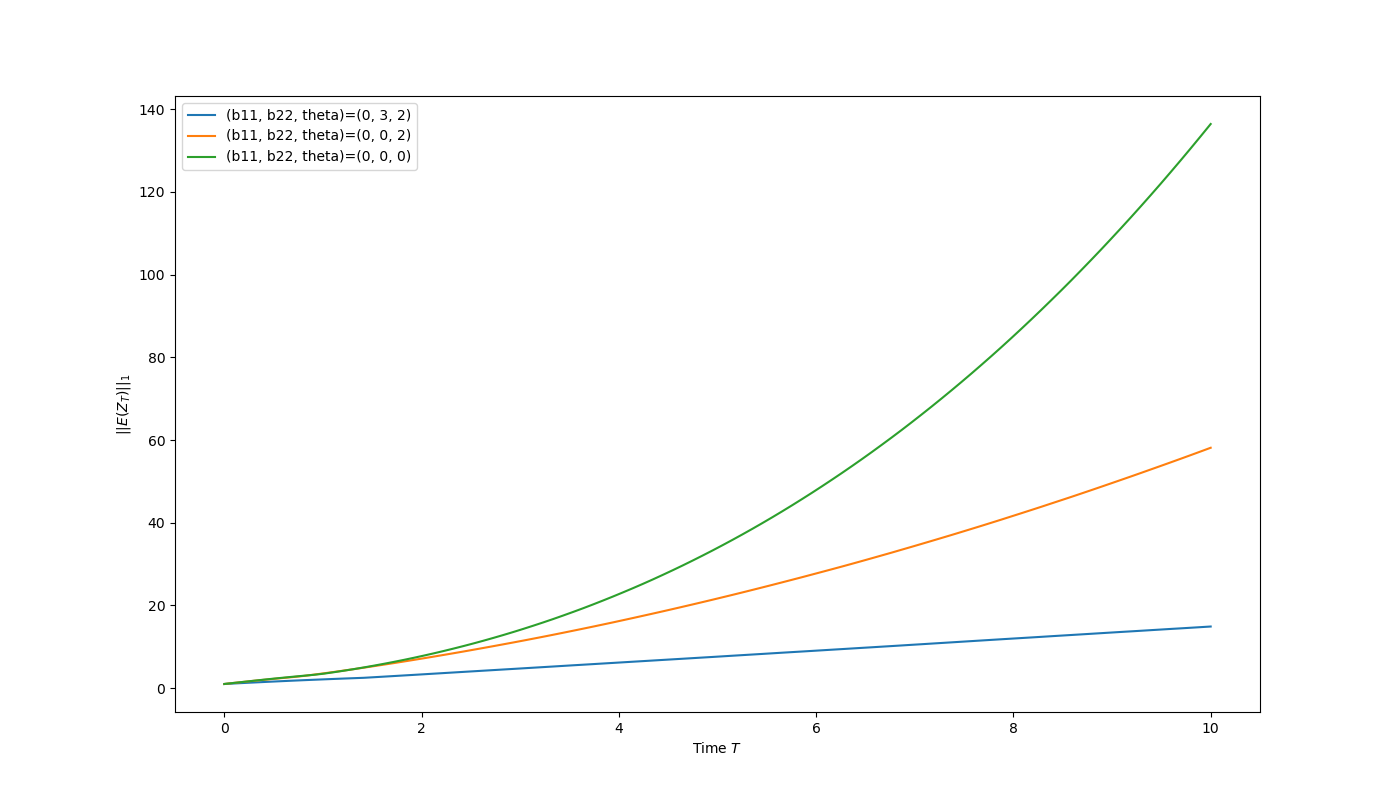}
    \caption{Asymptotic behavior of $E(Z_T)$ in the critical case for different parameter sets.}
    \label{fig:critical}
\end{figure}

In these critical cases, we remark that the expectation explodes with a speed proportional to $T$ when $(b_{11}, b_{22}, \theta)=(0, 3, 2)$, to $T^2$ when $(b_{11}, b_{22}, \theta)=(0, 0, 2)$ and to $T^3$ when $(b_{11}, b_{22}, \theta)=(0, 0, 0)$.

\begin{figure}[H]
    \centering
    \includegraphics[width=0.8\textwidth]{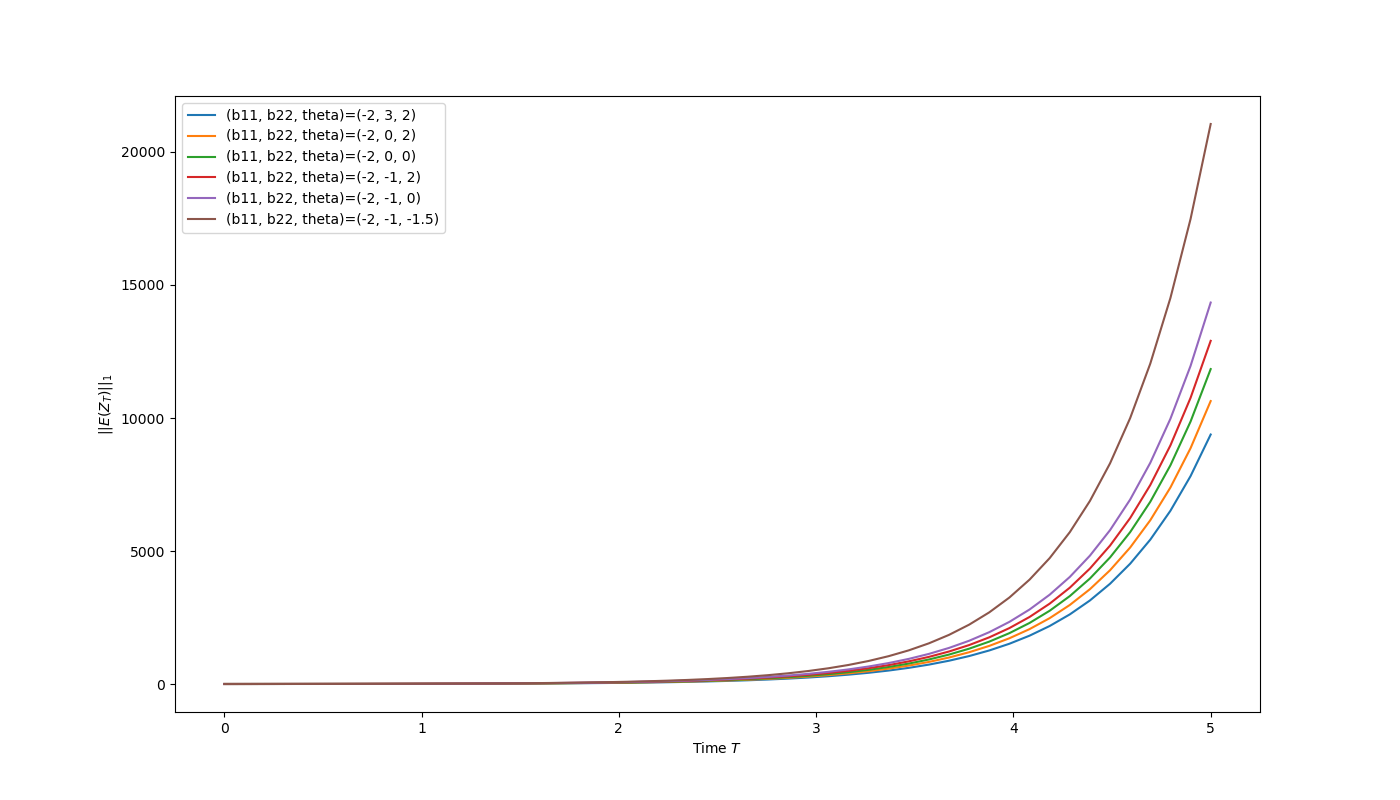}
    \caption{Asymptotic behavior of $E(Z_T)$ in the supercritical case.}
    \label{fig:supercritical_behavior}
\end{figure}

We remark here that when we take at least one negative parameter the expectation explodes with an exponential speed.\\

For the estimation problem, in order to simulate the maximum likelihood and the conditional least squares estimators, we need to compute necessary integrals and stochastic integrals using the trapezoidal rule for ordinary integrals and left-point sums for Itô integrals as follows
    \begin{equation*}
        \int_0^T \xi_s \, \mathrm{d}s \approx \sum_{i=0}^{N-1} \frac{\Delta t}{2} \left( \xi_{t_i} + \xi_{t_{i+1}} \right), \quad
        \int_0^T \xi_s \, \mathrm{d}\eta_s \approx \sum_{i=0}^{N-1} \xi_{t_i} (\eta_{t_{i+1}} - \eta_{t_i}).
    \end{equation*}
We recall that in this part we have studied just the subcritical case. In the numerical illustrations, we can test different values of the drift and the diffusion parameters. Let us consider, for example, the same vector as above: $(a_1,b_{11},a_2,b_{21},b_{22},m,\kappa_1,\kappa_2,\theta)=(1,1,1,-0.5,3,1,0.5,0.5,2)$.

\begin{figure}[H]
    \centering
    \includegraphics[width=0.8\textwidth]{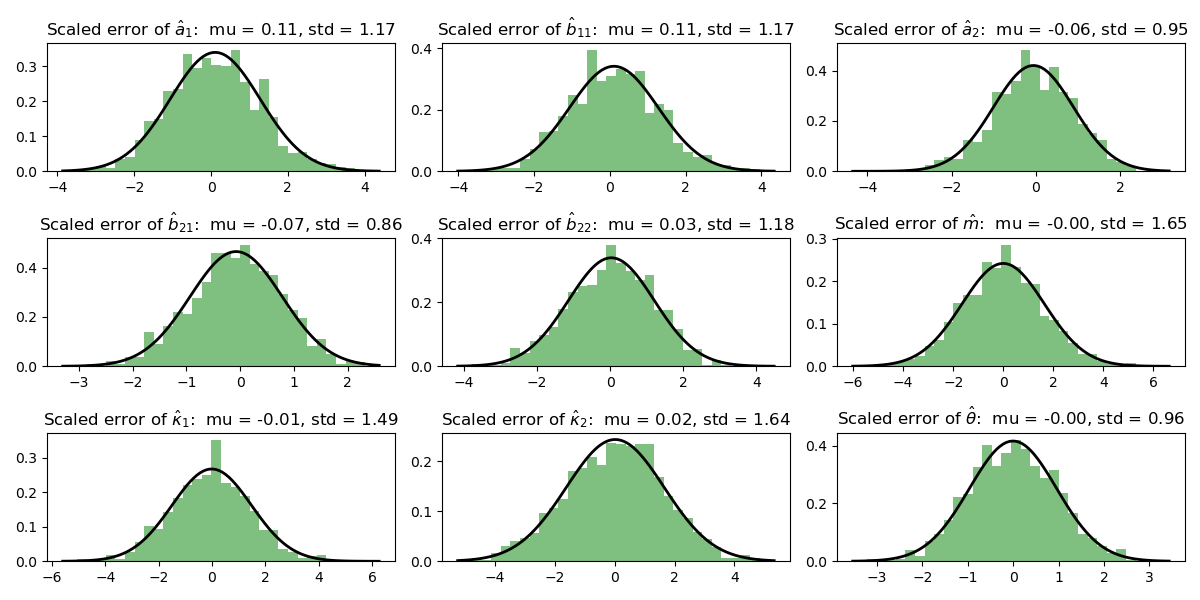}
    \caption{MLE: Distribution of the scaled error: $\sqrt{T}(\hat{\tau}_i-\tau_i)$, with $T=10^3$.}   \label{Distribution of the scaled error MLE}
\end{figure}

\begin{figure}[H]
    \centering
    \includegraphics[width=0.85\textwidth]{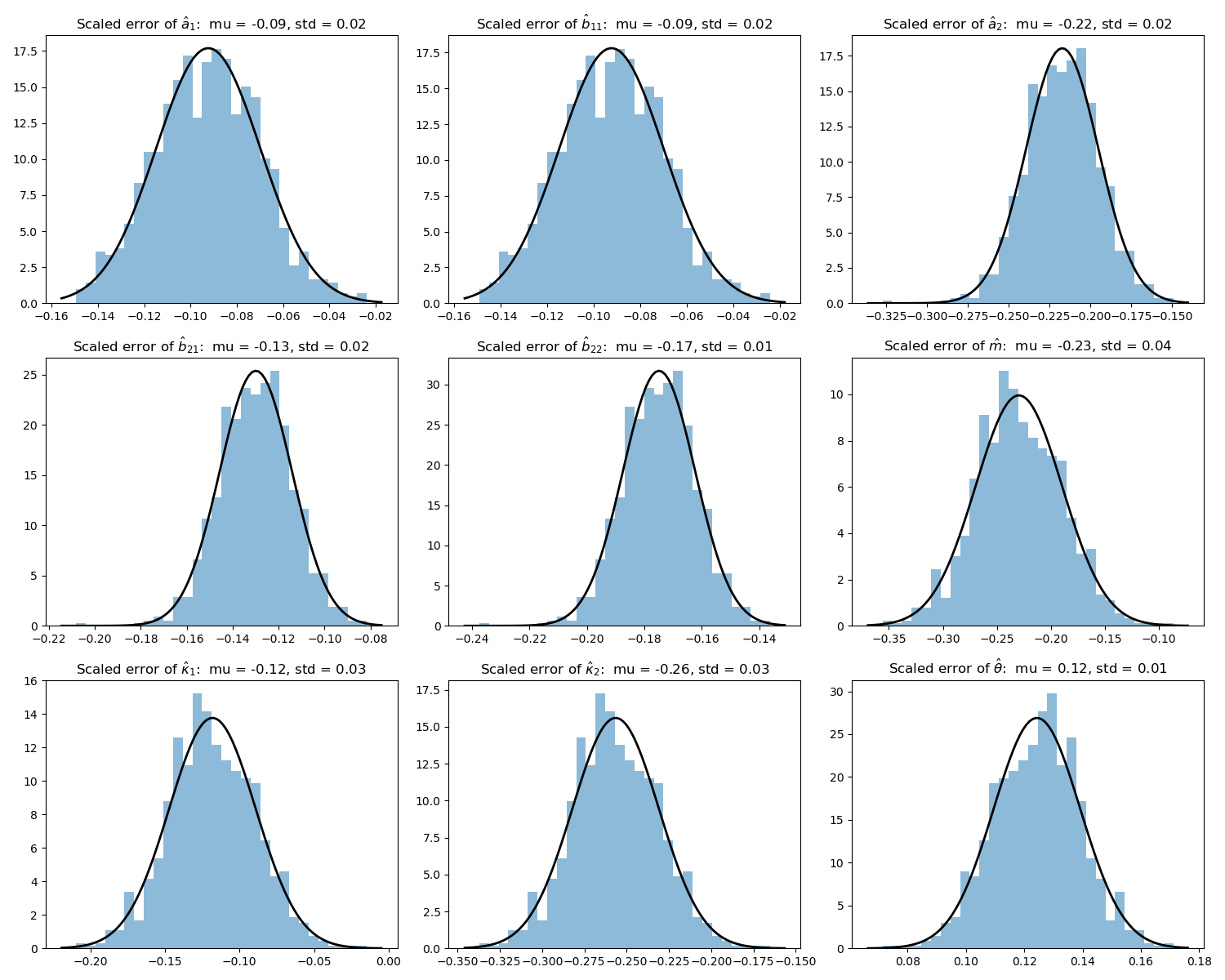}
    \caption{CLSE: Distribution of the scaled error: $\sqrt{T}(\check{\tau}_i-\tau_i)$, with $T=3\times10^5$.}   \label{Distribution of the scaled error CLSE}
\end{figure}

In this case, we observe that the MLE estimation method achieves asymptotic centered normality faster than the CLSE estimation method. Next, for all $i=1,\cdots,9$, we compute the error term using $1000$ simulations. Namely, we consider the empirical bias
$\left\vert\hat{\tau}^i_T-\tau_i\right\vert:=\dfrac{1}{1000}\displaystyle\sum_{k=1}^{1000} \left\vert\left(\hat{\tau}^i_T-\tau_i\right)_k\right\vert.$

\begin{table}[H]
\caption{ The MLE empirical bias $\left\vert\hat{\tau}^i_T-\tau_i\right\vert$}
\begin{tabular}{l|lllllllll|}
\cline{2-10} 
\multicolumn{1}{l|}{\multirow{-2}{*}{\cellcolor[HTML]{FFFFFF}}}                                               & \multicolumn{1}{l|}{$\tau_1=a_1$} & \multicolumn{1}{l|}{$\tau_2=b_{11}$} & \multicolumn{1}{l|}{$\tau_3=a_2$} & \multicolumn{1}{l|}{$\tau_4=b_{21}$} & \multicolumn{1}{l|}{$\tau_5=b_{22}$} & \multicolumn{1}{l|}{$\tau_6=m$} & \multicolumn{1}{l|}{$\tau_7=\kappa_1$}& \multicolumn{1}{l|}{$\tau_8=\kappa_2$} & $\tau_9=\theta$ \\ 
\hline
\multicolumn{1}{|l|}{{\color[HTML]{000000}$T=50$ }} & \multicolumn{1}{l|}{$0.0300$}        & \multicolumn{1}{l|}{$0.0301$}                  & \multicolumn{1}{l|}{$0.0018$}                         & \multicolumn{1}{l|}{$0.0013$}                         & \multicolumn{1}{l|}{$0.0054$}                         & \multicolumn{1}{l|}{$0.0149$}                         & \multicolumn{1}{l|}{$0.0204$}                         & \multicolumn{1}{l|}{$0.0140$}                         &                          \multicolumn{1}{l|}{$0.0090$}\\\hline
\multicolumn{1}{|l|}{{\color[HTML]{000000}$T=100$ }} & \multicolumn{1}{l|}{$0.0137$}        & \multicolumn{1}{l|}{$0.0143$}                  & \multicolumn{1}{l|}{$ 0.0018$}                         & \multicolumn{1}{l|}{$ 0.0012$}                                                & \multicolumn{1}{l|}{$0.0029$}                         & \multicolumn{1}{l|}{$0.0105$}                         & \multicolumn{1}{l|}{$0.0077$}                         &                          \multicolumn{1}{l|}{$0.0060$}& \multicolumn{1}{l|}{$0.0027$}  \\ \hline
\multicolumn{1}{|l|}{{\color[HTML]{000000}$T=500$ }} & \multicolumn{1}{l|}{$0.0048$}        & \multicolumn{1}{l|}{$0.0047$}        & \multicolumn{1}{l|}{$0.0015$}                  & \multicolumn{1}{l|}{$0.0010$}    & \multicolumn{1}{l|}{$0.0010$}                      & \multicolumn{1}{l|}{$0.0027$}                         & \multicolumn{1}{l|}{$0.0027$}                         & \multicolumn{1}{l|}{$0.0003$}                         & \multicolumn{1}{l|}{$0.0013$}   \\ \hline
\end{tabular}
\end{table}

\begin{table}[H]
\caption{ The CLSE empirical bias $\left\vert\check{\tau}^i_T-\tau_i\right\vert$}
\begin{tabular}{l|lllllllll|}
\cline{2-10}
\multicolumn{1}{l|}{\multirow{-2}{*}{\cellcolor[HTML]{FFFFFF}}}                                               & \multicolumn{1}{l|}{$\tau_1=a_1$} & \multicolumn{1}{l|}{$\tau_2=b_{11}$} & \multicolumn{1}{l|}{$\tau_3=a_2$} & \multicolumn{1}{l|}{$\tau_4=b_{21}$} & \multicolumn{1}{l|}{$\tau_5=b_{22}$} & \multicolumn{1}{l|}{$\tau_6=m$} & \multicolumn{1}{l|}{$\tau_7=\kappa_1$}& \multicolumn{1}{l|}{$\tau_8=\kappa_2$} & $\tau_9=\theta$ \\ 
\hline
\multicolumn{1}{|l|}{{\color[HTML]{000000}$T=10^2$ }} & \multicolumn{1}{l|}{$0.3449$}        & \multicolumn{1}{l|}{$0.3443$}                  & \multicolumn{1}{l|}{$0.4897$}                         & \multicolumn{1}{l|}{$0.3510$}                         & \multicolumn{1}{l|}{$0.2735$}                         & \multicolumn{1}{l|}{$0.4702$}                         & \multicolumn{1}{l|}{$0.2991$}                         & \multicolumn{1}{l|}{$0.3862$}                         &                          \multicolumn{1}{l|}{$0.1792$}\\\hline
\multicolumn{1}{|l|}{{\color[HTML]{000000}$T=10^3$ }} & \multicolumn{1}{l|}{$0.0636$}        & \multicolumn{1}{l|}{$0.0632$}                  & \multicolumn{1}{l|}{$ 0.0914$}                         & \multicolumn{1}{l|}{$ 0.0661$}                         & \multicolumn{1}{l|}{$0.0498$}                         & \multicolumn{1}{l|}{$0.0998$}                         & \multicolumn{1}{l|}{$0.0692$}                         & \multicolumn{1}{l|}{$0.0699$}                         &                          \multicolumn{1}{l|}{$0.0369$}\\ \hline
\multicolumn{1}{|l|}{{\color[HTML]{000000}$T=10^4$ }} & \multicolumn{1}{l|}{$0.0061$}        & \multicolumn{1}{l|}{$0.0060$}                  & \multicolumn{1}{l|}{$0.0103$}                         & \multicolumn{1}{l|}{$0.0076$}                         & \multicolumn{1}{l|}{$0.0053$}                         & \multicolumn{1}{l|}{$0.0109$}                         & \multicolumn{1}{l|}{$0.0075$}                         & \multicolumn{1}{l|}{$0.0077$}                         &                          \multicolumn{1}{l|}{$0.0041$}\\ \hline
\end{tabular}
\end{table}

As a result, according to our simulations, we observe that both the MLE and CLSE methods are consistent and have a Gaussian asymptotic distribution with a mean equal to 0. For the parameter values and simulation settings considered here, the MLE exhibits smaller reported errors than the CLSE. In fact, based on the tables above, we observe that the empirical bias reaches $10^{-2}$ when $T = 500$ for the MLE method and reaches the same value when $T = 10^4$ for the CLSE method. Numerical discretization and approximation may contribute to the observed differences in finite-sample performance.

\section{Conclusion}
This paper provides a comprehensive theoretical and statistical analysis of the ergodic double Heston stochastic volatility model under continuous-time observations. We first classified the model into subcritical, critical, and supercritical regimes, establishing the existence of a unique stationary distribution and exponential ergodicity in the subcritical case. Our work offers rigorous proofs of the consistency and asymptotic normality of both the maximum likelihood and conditional least squares estimators for the drift parameters, thereby validating their reliability in practical applications. The numerical results are consistent with the theoretical consistency and asymptotic normality results.\\

These findings extend the existing literature by providing the first detailed asymptotic analysis and stability results for the double Heston model in the ergodic regime. Future research should focus on adapting these estimation techniques to discrete-time observations, exploring robustness to jumps, and addressing the challenges of parameter estimation in non-ergodic regimes where stability is not guaranteed and standard methods are not applicable.

\appendix
\begin{center}
\bf{\Large Appendix}
\end{center}

\newtheorem{Theoreme}{Theorem}
\newtheorem{Remarque}{Remark}
\newtheorem{Lemme}{Lemma}
\setcounter{Theoreme}{0}

In what follows, we recall some limit theorems for continuous local martingales used in the study of the asymptotic behavior of the MLE of $\tau$. First, we recall a strong law of large numbers for continuous local martingales; see Liptser and Shiryaev~\cite{Lipster}.
\begin{Lemme}\label{finite Expectation of the inverse of CIR integral}
    Let $(Y_t)_{t\in\R_+}$ be a CIR process solution to the SDE 
    \begin{equation*}
    \mathrm{d}Y_t=\left(a-bY_t\right)\mathrm{d}t+\sigma\sqrt{Y_t}\mathrm{d}B_t,\quad Y_0=x\in\R_{+}
    \end{equation*}
    with $a,b,\sigma\in\R_{++}$ and $(B_t)_{t\in\R_+}$ is a standard Brownian motion. Then, we have $\mathbb{E}\left(\dfrac{1}{\int_0^1 Y_s \mathrm{d}s}\right)<\infty$.
\end{Lemme}
\begin{proof}
    Taking $\lambda=0$ and $t=1$ in Lemma 2 of~\cite{Alaya1} then considering its integral on $\R_+$ w.r.t. $\mu$, we deduce by the Fubini–Tonelli theorem that
\begin{equation*}
    \mathbb{E}\left(\dfrac{1}{\int_0^1 Y_s \mathrm{d}s}\right)=\displaystyle\int_0^\infty e^{-a\tilde{\phi}_\mu}e^{-x\tilde{\psi}_\mu}\mathrm{d}\mu,
\end{equation*}
where $\tilde{\psi}_\mu=\dfrac{2\mu \left(1-e^{-\sqrt{b^2+2\sigma^2 \mu}}\right)}{\left(\sqrt{b^2+2\sigma^2 \mu}-b\right)e^{-\sqrt{b^2+2\sigma^2 \mu}}+\sqrt{b^2+2\sigma^2 \mu}+b}$ \\and $\tilde{\phi}_\mu=-\dfrac{2}{\sigma^2}\left(\dfrac{b-\sqrt{b^2+2\sigma^2 \mu}}{2}+\log(2)+\log\left(\dfrac{\sqrt{b^2+2\sigma^2 \mu}}{\left(\sqrt{b^2+2\sigma^2 \mu}-b\right)e^{-\sqrt{b^2+2\sigma^2 \mu}}+\sqrt{b^2+2\sigma^2 \mu}+b}\right)\right)$.\\
Furthermore, it is easy to check that near zero (i.e., $\mu\in(0,\varepsilon)$, $\varepsilon\in\R_{++}$), there exists a constant $\alpha\in\R_{+}$ such that we have the following equivalence relation $e^{-a\tilde{\phi}_\mu}e^{-x\tilde{\psi}_\mu}\underset{0}{\sim} \alpha$. For large $\mu$, the explicit expressions of $\tilde{\phi}_\mu$
and $\tilde{\psi}_\mu$ imply 
\[
\lim_{\mu\to\infty}
\frac{a\tilde{\phi}_\mu+x\tilde{\psi}_\mu}{\sqrt{\mu}}
=\frac{\sqrt{2}(a+x)}{\sigma}>0.
\]
Hence, there exist $A,\beta\in\R_{++}$ such that
\[
0\leq e^{-a\tilde{\phi}_\mu-x\tilde{\psi}_\mu}
\leq e^{-\beta\sqrt{\mu}},
\qquad \mu\geq A.
\]
Since $\mu\mapsto e^{-\beta\sqrt{\mu}}$ is integrable
on $[A,\infty)$ and the integrand is bounded by $1$
on $[0,A]$, the desired expectation is finite.
Hence, the proof is completed.
\end{proof}
\begin{Theoreme}(Liptser and Shiryaev (2001)) Let $(\Omega,\mathcal{F},(\mathcal{F})_{t\in \real_{+}},\mathbb{P})$ be a filtered probability space satisfying the usual conditions. Let  $(M_{t})_{t \in \real_{+}}$ be a square-integrable continuous local martingale with respect to the filtration $(\mathcal{F}_t)_{t\in \real_{+}}$ such that $\mathbb{P} \left( M_0 = 0\right)=1.$ Let $(\xi_t)_{t\in \real_{+}}$ be a progressively measurable process such that
$$\mathbb{P}\left( \displaystyle\int_{0}^{t}  \xi_{u}^{2}\mathrm{~d} \left\langle M\right\rangle_u < \infty\right)=1, \quad t \in \real_{+},$$
and
\begin{align*}
\displaystyle\int_{0}^{t}  \xi_{u}^{2} \mathrm{~d} \left\langle M\right\rangle_u \stackrel{a.s.}{\longrightarrow} \infty,\quad as\;\; t\longrightarrow \infty,
\end{align*}
where $\left( \left\langle M\right\rangle_t  \right)_{t \in \real_{+}} $ denotes the quadratic variation process of $M.$ Then
\begin{align*}
\dfrac{\displaystyle\int_{0}^{t}  \xi_{u} \mathrm{~d} M_u}{\displaystyle\int_{0}^{t}  \xi_{u}^{2} \mathrm{~d} \left\langle M\right\rangle_u}
\stackrel{a.s.}{\longrightarrow} 0,\quad as\;\; t\longrightarrow \infty.
\end{align*}
If $(M_{t})_{t \in \real_{+}}$ is a standard Wiener process, the progressive measurability of $(\xi_t)_{t\in \real_{+}}$ can be relaxed to measurability and adaptedness to the filtration $(\mathcal{F})_{t\in \real_{+}}.$
\label{Lipster Shiryaev theorem}
\end{Theoreme}
The second theorem is about the asymptotic behavior of continuous multivariate local martingales; see \cite[Theorem 4.1]{Zanten}.
\begin{Theoreme}
	For $p\in\N^*$, let $M=(M_t)_{t\in\R_+}$ be a $p$-dimensional square-integrable continuous
	local martingale with respect to the filtration $(\mathcal{F}_t)_{t\in\R_+}$ such that $\mathbb{P}(M_0=0)=1$. Suppose that
	there exists a function $Q:\left[t_0,\infty\right) \to \mathcal{M}_p$ with some $t_0\in\R_+$ such that $Q(t)$ is an invertible
	(non-random) matrix for all $t\in\R_+$, $\underset{t\to\infty}{\lim} \norm{Q(t)}=0$ and
	$$	Q(t) \langle M\rangle_t Q(t)^{\top}\stackrel{\mathbb{P}}{\longrightarrow}\eta\eta^{\top},\quad \text{as }t\longrightarrow \infty,$$
	where $\eta$ is a random matrix in $\mathcal{M}_p$. Then, for each random matrix $A\in\mathcal{M}_{k,l}$,  $k,l\in\N^*$, defined on $(\Omega,\mathcal{F},(\mathcal{F})_{t\in \real_{+}},\mathbb{P})$, we have
	$$
	(Q(t)M_t,A) \stackrel{\mathcal{D}}{\longrightarrow}
	(\eta Z,A),\quad \text{as }t\longrightarrow \infty,$$
	where $Z$ is a $p$-dimensional standard normally distributed random vector independent of $(\eta,A)$.
	\label{CLT Van Zanten}
\end{Theoreme}
\section*{}
\makeatother
\bibliographystyle{natdin}
\bibliography{doubleHestonModel}
\newpage		

\end{document}